\documentclass[acmsmall,review=false]{acmart}
\usepackage[utf8]{inputenc}
\usepackage[T1]{fontenc}
\usepackage[british]{babel}
\usepackage{xspace}
\usepackage[textsize=small]{todonotes}
\usepackage{graphicx,csquotes}
\usepackage{amsmath,amsthm,mathrsfs,mathtools,bm,dsfont,xfrac}
\usepackage{bbm}
\usepackage[all,2cell,cmtip]{xy}
\usepackage{caption}
\usepackage{aliascnt}
\usepackage{microtype}
\usepackage{cleveref}

\usepackage[footnote,marginclue,nomargin,draft]{fixme}

\hypersetup{unicode}

\usepackage[inline]{enumitem}
\setlist[enumerate,1]{label=(\arabic*),font=\normalfont,align=left,leftmargin=0pt,labelindent=0pt,listparindent=\parindent,labelwidth=0pt,itemindent=!,topsep=3pt,parsep=0pt,itemsep=3pt,start=1}
\setlist[enumerate,2]{label=(\arabic{enumi}\alph*),font=\normalfont,align=left,leftmargin=0pt,labelindent=0pt,listparindent=\parindent,labelwidth=0pt,itemindent=!,topsep=3pt,parsep=0pt,itemsep=3pt,start=1}
\setlist[enumerate,3]{label=(\roman*),font=\normalfont,align=left,leftmargin=0pt,labelindent=0pt,listparindent=\parindent,labelwidth=0pt,itemindent=!,topsep=3pt,parsep=0pt,itemsep=3pt,start=1}

\setlist[itemize]{labelindent=*,leftmargin=*}
\AtEndPreamble{
\theoremstyle{plain}

\newtheorem{thm}[theorem]{Theorem}
\newtheorem{cor}[theorem]{Corollary}
\newtheorem{lem}[theorem]{Lemma}
\newtheorem{prop}[theorem]{Proposition}
\theoremstyle{definition}
\newtheorem{defn}[theorem]{Definition}
\newtheorem{expl}[theorem]{Example}

\newtheorem{notation}[theorem]{Notation}
\newtheorem{construction}[theorem]{Construction}
\newtheorem{remark}[theorem]{Remark}
\newtheorem{assumption}[theorem]{Assumption}
}

\newdir{ >}{{}*!/-5pt/@{>}}
\newdir{ (}{{}*!/-10pt/\dir^{(}}
\newdir{ )}{{}*!/-10pt/\dir_{(}}

\numberwithin{equation}{section}

\newcommand{\Stone}{{\mathbf{Stone}}}

\newcommand{\Priest}{\mathbf{Priest}}

\newcommand{\Set}{\mathbf{Set}}
\newcommand{\Pos}{\mathbf{Pos}}

\newcommand{\Alg}{\mathsf{Alg}}
\newcommand{\A}{\mathscr{A}}
\newcommand{\B}{\mathscr{B}}
\newcommand{\Cat}{\mathscr{C}}
\newcommand{\D}{\mathscr{D}}
\newcommand{\Df}{\D_{\mathsf{f}}}

\newcommand{\E}{\mathcal{E}}
\newcommand{\M}{\mathcal{M}}
\newcommand{\V}{\mathcal{V}}
\renewcommand{\L}{\mathscr{L}}
\newcommand{\MT}{\mathbf{T}}

\newcommand{\Pro}[1]{\mathop{\mathsf{Pro}}#1}
\newcommand{\Ind}[1]{\mathop{\mathsf{Ind}}#1}

\newcommand{\hatD}{\widehat\D}
\newcommand{\hatT}{{\widehat\MT}}
\newcommand{\hatt}{\widehat{T}}

\newcommand{\rel}{\mathsf{rel}}

\newcommand{\ol}{\overline}
\newcommand{\hatE}{\widehat{\E}}
\newcommand{\hatM}{\widehat{\M}}

\newcommand{\dash}{\mathord{-}}

\newcommand{\hatX}{{\widehat{X}}}
\newcommand{\hateta}{\widehat\eta}
\newcommand{\hatmu}{\widehat\mu}
\renewcommand{\epsilon}{\varepsilon}
\newcommand{\id}{\mathsf{id}}
\newcommand{\Id}{\mathsf{Id}}
\newcommand{\seq}{\subseteq}
\newcommand{\xto}{\xrightarrow}

\newcommand{\defeq}{\coloneqq}
\newcommand{\Ran}{\mathsf{Ran}}

\newcommand{\op}{\mathrm{op}}
\newcommand{\set}[2]{\{\, #1 \mid #2 \,\}}
\newcommand{\<}{\langle}
\renewcommand{\>}{\rangle}

\renewcommand{\o}{\cdot}

\newcommand{\takeout}[1]{\empty}

\renewcommand{\phi}{\varphi}

\newcommand{\mc}{\mathcal}

\renewcommand{\S}{\mc{S}}

\newcommand{\BA}{\mathbf{BA}}

\newcommand{\KCat}{\mathscr{K}}

\newcommand{\limit}{\mathop{\mathsf{lim}}}
\newcommand{\colim}{\mathop{\mathsf{colim}}}

\newcommand{\subto}{\hookrightarrow}
\newcommand{\epito}{\twoheadrightarrow}

\newcommand{\monoto}{\rightarrowtail}

\newcommand\epidownarrow{\mathrel{\rotatebox[origin=c]{90}{$\twoheadleftarrow$}}}

\newcommand{\J}{\mathscr{J}}

\newcommand{\SigmaAlg}{\mathop{\mathbf{Alg}}\Sigma}
\newcommand{\SigmaOAlg}{\mathop{\mathbf{Alg}}\Sigma_{\leq}}
\newcommand{\Var}{\mathsf{Var}}
\newcommand{\K}{\mathscr{K}}

\newcommand{\Nat}{\mathbb{N}}

\newcommand{\f}{{\mathsf{f}}}

\DeclareBoldMathCommand\boldlsqbracket{\left[}
\DeclareBoldMathCommand\boldrsqbracket{\right]}
\newcommand{\func}[2]{\boldlsqbracket#1, #2\boldrsqbracket}

\newcommand{\SigStr}{\Sigma\text{-}\mathbf{Str}}
\newcommand{\SigopStr}{\Sigma_\op\text{-}\mathbf{Str}}

\begin{document}
%
% For Fixme comments
%
\FXRegisterAuthor{sm}{asm}{SM}%Stefan
\FXRegisterAuthor{ja}{aja}{JA}%Jirka
\FXRegisterAuthor{hu}{ahu}{HU}%Henning
\FXRegisterAuthor{lt}{alt}{LT}%Liang-Ting

%
% Frontmatter
%
%%%%%%%%%%%%%%%%%%%%%%%%%%%%%%%%%%%%%%%%%%%%%%%%%%%%%%%%%%%%%%%%%%%%%%%%%%%%%%%%
% Title, Author, and Affiliation Information 
%\usepackage{authblk}

\title{Reiterman's Theorem on Finite Algebras for a Monad}

\author[J.~Ad\'amek]{Ji\v{r}\'i Ad\'amek}
\authornote{Ji\v{r}\'i Ad\'amek acknowledges support by the Grant Agency of the Czech Republic under the grant 19-00902S.}          %% \authornote is optional;
                                        %% can be repeated if necessary
%\orcid{nnnn-nnnn-nnnn-nnnn}             %% \orcid is optional
\affiliation{
  %\position{Position1}
  \department{Department of Mathematics, Faculty of Electrical Engineering}              %% \department is recommended
  \institution{Czech Technical University in Prague, Czech Recublic, and Technische Universität Braunschweig}          %% \institution is required
  %\streetaddress{Martensstr. 3}
  %\city{Erlangen}
  %\state{State1}
  %\postcode{91058}
  \country{Germany}                    %% \country is recommended
}
\email{j.adamek@tu-braunschweig.de}          %% \email is recommended

\author[L.-T.~Chen]{Liang-Ting Chen}
%\authornote{Henning Urbat acknowledges support by Deutsche Forschungsgemeinschaft (DFG) under project SCHR~1118/8-2.}          %% \authornote is optional;
                                        %% can be repeated if necessary
%\orcid{nnnn-nnnn-nnnn-nnnn}             %% \orcid is optional
\affiliation{
  %\position{Position1}
  \department{Institute of Information Science}              %% \department is recommended
  \institution{Academia Sinica}            %% \institution is required
  %\streetaddress{Martensstr. 3}
  %\city{Erlangen}
  %\state{State1}
  %\postcode{91058}
  \country{Taiwan}                    %% \country is recommended
}
\email{liang.ting.chen.tw@gmail.com}          %% \email is recommended

\author[S.~Milius]{Stefan Milius}
\authornote{Stefan Milius acknowledges support by the Deutsche
  Forschungsgemeinschaft (DFG) under project MI 717/5-2 and as part of the Research and Training Group 2475 ``Cybercrime and Forensic Computing'' (393541319/GRK2475/1-2019)}          %% \authornote is optional;
                                        %% can be repeated if necessary
%\orcid{nnnn-nnnn-nnnn-nnnn}             %% \orcid is optional
\affiliation{
  %\position{Position1}
  %\department{Department1}              %% \department is recommended
  \institution{Friedrich-Alexander-Universität Erlangen-Nürnberg}            %% \institution is required
  \streetaddress{Martensstr. 3}
  \city{Erlangen}
  %\state{State1}
  \postcode{91058}
  \country{Germany}                    %% \country is recommended
}
\email{mail@stefan-milius.eu}          %% \email is recommended

\author[H.~Urbat]{Henning Urbat}
\authornote{Henning Urbat acknowledges support by Deutsche Forschungsgemeinschaft (DFG) under project SCHR~1118/8-2.}          %% \authornote is optional;
                                        %% can be repeated if necessary
%\orcid{nnnn-nnnn-nnnn-nnnn}             %% \orcid is optional
\affiliation{
  %\position{Position1}
  %\department{Department1}              %% \department is recommended
  \institution{Friedrich-Alexander-Universität Erlangen-Nürnberg}            %% \institution is required
  \streetaddress{Martensstr. 3}
  \city{Erlangen}
  %\state{State1}
  \postcode{91058}
  \country{Germany}                    %% \country is recommended
}
\email{henning.urbat@fau.de}          %% \email is recommended

\keywords{Monad, Pseudovariety, Profinite Algebras}

\terms{Theory}

\begin{CCSXML}
<ccs2012>
<concept>
<concept_id>10003752.10003766.10003767.10003768</concept_id>
<concept_desc>Theory of computation~Algebraic language theory</concept_desc>
<concept_significance>500</concept_significance>
</concept>
</ccs2012>
\end{CCSXML}

\ccsdesc[500]{Theory of computation~Algebraic language theory}

\begin{abstract}
  Profinite equations are an indispensable tool for the algebraic
  classification of formal languages. Reiterman's theorem states that
  they precisely specify pseudovarieties, i.e.~classes of finite
  algebras closed under finite products, subalgebras and quotients. In
  this paper, Reiterman's theorem is generalized to finite
  Eilenberg-Moore algebras for a monad~$\MT$ on a category~$\D$: we
  prove that a class of finite $\MT$-algebras is a pseudovariety iff
  it is presentable by profinite equations. As a key technical
  tool, we introduce the concept of a profinite monad $\hatT$
  associated to the monad $\MT$, which gives a categorical view of the
  construction of the space of profinite terms.
\end{abstract}

\maketitle
%
% Sections
%
\section{Introduction}\label{sec:intro}

One of the main principles of both mathematics and computer science is the specification of structures in terms of equational properties. The first systematic study of equations as mathematical objects was pursued by Birkhoff~\cite{Birkhoff35} who proved that a class of algebraic structures over a finitary signature $\Sigma$ can be specified by equations between $\Sigma$-terms if and only if it is closed under quotient algebras (a.k.a.~homomorphic images), subalgebras, and products. This fundamental result, known as the \emph{HSP theorem}, lays the ground for universal algebra and has been extended and generalized in many directions over the past 80 years, including categorical approaches via Lawvere theories~\cite{arv11,lawvere63} and monads~\cite{manes76}.

While Birkhoff's seminal work and its categorifications are concerned with general algebraic structures, in many computer science applications the focus is on \emph{finite} algebras. For instance, in automata theory, regular languages (i.e.~the behaviors of classical finite automata) can be characterized as precisely the 
languages recognizable by finite monoids. This algebraic point of view leads to important insights, including decidability results. As a prime example, Sch\"utzenberger's 
theorem~\cite{sch65} asserts that star-free regular languages 
correspond to \emph{aperiodic} finite monoids, i.e.~monoids where the unique idempotent power $x^\omega$ of any element $x$ satisfies $x^\omega = x\o x^\omega$. As an immediate application, one obtains the
decidability of star-freeness. However, the identity $x^\omega=x\o x^\omega$ is not an equation in Birkhoff's sense since the operation $(\dash)^\omega$ is not a part of the signature of monoids. Instead, it is an instance of a \emph{profinite equation}, a topological generalization of Birkhoff's concept introduced by Reiterman~\cite{Reiterman1982}. (Originally, Reiterman worked with the equivalent concept of an \emph{implicit equation}, cf.~Section \ref{sec:profinite-terms}.) Given a set $X$ of variables and $x\in X$, the expression $x^\omega$ can be interpreted as an element of the Stone space $\widehat{X^*}$ of \emph{profinite words}, constructed as the cofiltered limit of all finite quotient monoids of the free monoid $X^*$. Analogously, over general signatures $\Sigma$ one can form the Stone space of \emph{profinite $\Sigma$-terms}. Reiterman proved that a class of finite $\Sigma$-algebras can be specified by profinite equations (i.e.~pairs of profinite terms) if and only if it is closed under quotient algebras, subalgebras, and {finite} products. This result establishes a finite analogue of Birkhoff's HSP theorem.

In this paper, we develop a categorical approach to Reiterman's theorem and the theory of profinite equations. The idea is to replace monoids (or general algebras over a signature) by Eilenberg-Moore algebras for a monad $\MT$ on an arbitrary base category $\D$. As an important technical device, we introduce a categorical abstraction of the space of profinite words. To this end, we
consider a full subcategory $\Df$ of $\D$ of ``finite''
objects and form the category $\Pro \Df$, the free completion of $\Df$ under cofiltered limits. We then show that the monad $\MT$ naturally induces a monad $\hatT$ on $\Pro{\Df}$, called the
\emph{profinite monad} of~$\MT$, whose free algebras $\hatT X$ serve as domains for profinite equations. For example, for $\D = \Set$ and the full subcategory $\Set_\f$ of finite sets, we get $\Pro{\Set_\f}=\Stone$, the category of Stone spaces. Moreover, if $\MT X= X^*$ is
the finite-word monad (whose algebras are precisely monoids), then $\hatT$
is the monad of profinite words on $\Stone$; that is, 
$\hatT$ associates to each finite Stone space (i.e.~a finite set with
the discrete topology) $X$ the space 
$\widehat{X^*}$ of profinite words on $X$. Our overall approach can thus be summarized by the following diagram, where the skewed functors are inclusions and the horizontal ones are forgetful functors.

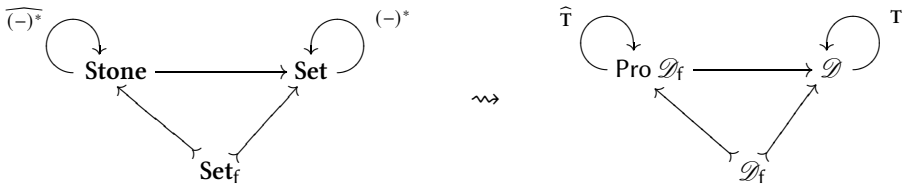
\[
\vcenter{
  \xymatrix@C-1em{
    \quad\ \Stone
    \ar[rr]
    \turnradius{10pt}
    \ar@{<-} `u^r []
    \POS(-10,7) *{\labelstyle\widehat{(\dash)^*}}  
    & &
    \Set
    \turnradius{10pt}
    \ar@{<-} `u_l []
    \POS(39,7) *{\labelstyle(\dash)^*}
    \\
    &
    \Set_\f \ar@{ >->}[ul] \ar@{>->}[ur]
  }
}
\qquad\rightsquigarrow\qquad 
\vcenter{
  \xymatrix@C-1em{
    \quad\ \Pro{\Df}
    \ar[rr] 
    \turnradius{10pt}
    \ar@{<-} `u^r []
    \POS(-9,7) *{\labelstyle\hatT}
    & &
    \D\!\!
    \turnradius{10pt}
    \ar@{<-} `u_l []
    \POS(35,7) *{\labelstyle\MT}
    \\
    &
    \Df
    \ar@{>->}[ul]
    \ar@{>->}[ur]  
  }
}
\]
It turns out that many familiar properties of the space of profinite words can be developed at the abstract level of profinite monads and their algebras. Our main result is the 

\medskip\noindent \textbf{Generalized Reiterman Theorem.} A class of finite $\MT$-algebras is presentable by profinite equations if and only if it is closed under quotient algebras, subalgebras, and finite products.

\medskip\noindent Here, \emph{profinite equations} are modelled categorically as finite quotients $e\colon \hatt X \epito E$ of the object $\hatt X$ of generalized profinite terms. If the category $\D$ is $\Set$ or, more generally, a category of first-order structures, we will see that this abstract concept of an equation is equivalent to the familiar one: $\hatt X$ is a topological space and quotients $e$ as above can be identified with sets of pairs $(s,t)$ of profinite terms $s,t\in \hatt X$. Thus, our categorical results instantiate to the original Reiterman theorem~\cite{Reiterman1982} ($\D=\Set$), but also to its versions for ordered algebras ($\D=\Pos$) and for first-order structures due to Pin
 and Weil~\cite{PinWeil1996}.

Our proof of the Generalized Reiterman Theorem is purely categorical and  relies on general properties of (codensity) monads, free completions and locally finitely copresentable categories. It does not employ any topological methods, as opposed to all known proofs of Reiterman's theorem and its variants. The insight that topological reasoning can be completely avoided in the profinite world is quite surprising, and we consider it as one of the main contributions of our paper.

\paragraph{\textbf{Related work}} This paper is the full version of an extended abstract~\cite{camu16} presented at FoSSaCS 2016. Besides providing complete proofs of all results, the presentation is significantly more general than in \emph{op.~cit.}: there we restricted ourselves to base categories $\D$ which are varieties of (possibly ordered) algebras, and the development of the profinite monad and its properties used results from topology. In contrast, the present paper works with general categories $\D$ and develops all required profinite concepts in full categorical abstraction, with  topological arguments only appearing in the verification that our concrete instances satisfy the required categorical properties.

An important application of the Generalized Reiterman Theorem and the
profinite monad can be found in algebraic language theory: we showed
that given a category $\Cat$ dually equivalent to $\Pro{\Df}$, the
concept of a profinite equational class of finite $\MT$-algebras
dualizes to the concept of a \emph{variety of $\MT$-recognizable
  languages in $\Cat$}. For instance, for $\D = \Set$ and
$\Pro{\Df}=\Stone$, the classical Stone duality yields the category
$\Cat=\BA$ of boolean algebras, and for the monad $\MT X = X^*$ on
$\Set$ the dual correspondence gives Eilenberg's fundamental
\emph{variety theorem} for regular
languages~\cite{Eilenberg1976}. Using our duality-theoretic approach
we established a categorical generalization of Eilenberg's theorem and
showed that it instantiates to more than a dozen Eilenberg-type
results known in the literature, along with a number of new
correspondence results~\cite{uacm17}. Let us also mention some of the
very few known instances of Eilenberg-type results not obtained using
an application of the Generalized Reiterman Theorem. The first one is
a recent Eilenberg-type correspondence for regular
languages~\cite{bmu21}, which is based on lattice bimodules, a new
algebraic structure for recognition originally proposed by Pol\'ak and
Klima~\cite{KlimaP19} under the name lattice algebras. The second
result is our first Eilenberg-type corresponding for nominal
languages~\cite{um19}. Finally, there are Eilenberg-type
correspondences for varieties of non-regular languages; they appear in work of
Behle et al.~\cite{BehleEA11} and as instances of Salamanca's general 
framework~\cite{Salamanca17}.

Recently, an abstract approach to HSP-type theorems~\cite{mu19} has
been developed that not only provides a common roof over Birkhoff's
and Reiterman's theorem, but also applies to classes of algebras with
additional underlying structure, such as ordered, quantitative, or
nominal algebras. The characterization of pseudovarieties in terms of
pseudoeuqations given in \Cref{prop:pseudovariety=pseudoequations}
is a special case of the HSP theorem in \emph{op.~cit}.

\section{Profinite Completion}\label{sec:profinite-comp}

In this preliminary section, we review the profinite completion
(commonly known as pro-completion) of a category and describe it for
the category $\Sigma$-$\mathsf{Str}$ of structures over a first-order
signature $\Sigma$.

\begin{remark}\label{rem:cofiltered}
Recall that a category is \emph{cofiltered} if every finite
subcategory has a cone in it.  For example, every cochain (i.e.~a
poset dual to an ordinal number) is cofiltered.  A \emph{cofiltered
  limit} is a limit of a diagram with a small cofiltered diagram
scheme. A functor is \emph{cofinitary} if it preserves cofiltered
limits. An object $A$ of a category $\mathscr{C}$ is called
\emph{finitely copresentable} if the
functor~$\mathscr{C}(-, A)\colon \mathscr{C} \to \Set^\op$ is
cofinitary.  The latter means that for every limit cone
$c_i\colon C \to C_i$ ($i \in \mc{I}$) of a cofiltered diagram,
\begin{enumerate}
\item each morphism $f\colon C \to A$ factorizes through some
  $c_i\colon C \to C_i$ as $f = g \o c_i$, and
  
\item the morphism $g\colon C_i \to
  A$ is \emph{essentially unique}, i.e.~given another factorization
  $f = h \o c_i$, there is a connecting morphism $c_{ji} \colon C_j
  \to C_i$ with $g\o c_{ji}=h\o c_{ji}$:
  \[
    \xymatrix{
      & C \ar[r]^{f} \ar[d]_{c_i} \ar[ld]_{c_j} & A \\
      C_j \ar@{-->}[r]_{c_{ji}} & C_i \ar@<.5ex>[ru]^{g} \ar@<-.5ex>[ru]_{h}
    }
  \]
\end{enumerate}
The dual concept is that of a \emph{filtered colimit}.
\end{remark}
\begin{notation}\label{not:proc}
\begin{enumerate}
\item  The free completion of a category $\Cat$ under cofiltered limits,
  i.e.~the \emph{pro-completion}, is denoted by
  \[
    \Pro\Cat.
  \]
  This is a category with cofiltered limits together with a full
  embedding $E\colon \Cat\monoto \Pro{\Cat}$ satisfying the following universal
  property:
  \begin{enumerate}
  \item Every functor $F\colon \Cat\to\KCat$ into a category $\KCat$
    with cofiltered limits admits a cofinitary extension
    $\overline{F}\colon \Pro\Cat\to \KCat$, i.e.~the triangle below
    commutes: 
    \[
      \xymatrix{
        \Cat \ar[r]^E \ar[dr]_{F} & \Pro{\Cat} \ar@{-->}[d]^{\overline{F}} \\
        & \KCat
      }
    \]
  \item The functor $\overline{F}$ is \emph{essentially unique}, i.e.~for every
    cofinitary extension $G$ of $F$ there exists a unique natural
    isomorphism $i\colon \overline{F}\xto{\cong} G$ with $iE = \id_F$.
  \end{enumerate}
  More precisely, the full embedding $E$ \emph{is} the pro-completion,
  but we will often simply refer to $\Pro\Cat$ as the pro-completion instead.
\item
Dually, the free completion of $\Cat$ under filtered colimits, i.e.~the \emph{ind-completion}, is denoted by
\[ \Ind{\Cat}.\]
\end{enumerate}
\end{notation}
Some standard results on ind- and pro-completions can be found in the Appendix.

\begin{expl}\label{ex:pro}
  \begin{enumerate}
  \item\label{ex:pro:1} Let $\Set_\f$ be the category of finite sets
    and functions.  Its pro-completion is the category
      \[
        \Pro\Set_\f = \Stone
      \]
      of \emph{Stone spaces}, i.e.~compact topological spaces in which
      distinct elements can be separated by clopen subsets.  Morphisms
      are the continuous functions.  The embedding
      $\Set_\f\monoto \Stone$ identifies finite sets with finite
      discrete spaces.  This is a consequence of the Stone
      duality~\cite{Johnstone1982} between $\Stone$ and the category
      $\BA$ of boolean algebras, and its restriction to finite sets and
      finite Boolean algebras. In fact, since $\BA$ is a finitary variety, it
      is the ind-completion of its full subcategory $\BA_\f$ of finitely presentable objects, which are precisely the finite Boolean
      algebras.  Therefore
      \[
        \Pro\Set_\f = (\Ind \Set_\f^\op)^\op \cong (\Ind \BA_\f)^\op
        \cong \BA^\op \cong \Stone.
      \]
      %Since finite sets correspond, under Stone duality, precisely to
      %finite boolean algebras, this implies that
      %$\Pro\Set_\f = \Stone$.

    \item\label{ex:pro:2} For the category of finite posets and
      monotone functions, denoted by $\Pos_\f$, we obtain the category
      \[
        \Pro\Pos_\f = \Priest
      \]
      of \emph{Priestley spaces}, i.e.~ordered Stone spaces such that
      any two distinct elements can be separated by clopen upper sets.
      Morphisms in $\Priest$ are continuous monotone functions.  This
      follows from the Priestley duality~\cite{Priestley1972}
      between~$\Priest$ and bounded distributive lattices.  The
      argument is analogous to item~\ref{ex:pro:1}: finite,
      equivalently finitely presentable, distributive lattices dualize
      to finite posets with discrete topology.
  \end{enumerate}
\end{expl}

\begin{notation}[First-order structures]
  We will often work with the category 
  \[
    \SigStr
  \]
  of $\Sigma$-structures and $\Sigma$-homomorphisms for a first-order
 many-sorted signature $\Sigma$.  Given a set $\mc{S}$ of sorts, an
  \emph{$\mc{S}$-sorted signature} $\Sigma$ consists of
(1) operation symbols $\sigma\colon s_1, \ldots, s_n \to s$ where $n\in\Nat$,
    the sorts $s_i$ form the domain of $\sigma$ and $s$ is its
    codomain, and (2) relation symbols $r\colon s_1,\ldots, s_m$ where
    $m \in \Nat^+= \Nat \setminus \{0\}$.
  A \emph{$\Sigma$-structure} is an $\mc{S}$-sorted set
  \[
    A = (A^s)_{s \in \mc{S}} \quad \text{in} \quad \Set^{\mc{S}}
  \]
  with (1) an operation~$\sigma_A\colon A^{s_1} \times \dots \times A^{s_n} \to A^s$ for every operation symbol $\sigma\colon s_1,\ldots,s_n \to s$, and (2) a relation $r_A \subseteq A^{s_1} \times \dots A^{s_n}$ for every relation symbol $r\colon s_1,\ldots, s_n$.
  A \emph{$\Sigma$-homomorphism} is an $\mc{S}$-sorted function
  $f\colon A \to B$ which preserves operations and relations in the
  usual sense.  We denote by $\SigStr_\f$ the full subcategory of
  $\SigStr$ given by all $\Sigma$-structures $A$ where each $A^s$ is
  finite.

  When $\mc{S}$ is a singleton, the notion of $\Sigma$-structures
  boils down to a more common situation.  Namely, the arity of an
  operation symbol is given solely by $n \in \Nat$ and that of a
  relation symbol by $m \in \Nat^+$.  A $\Sigma$-structure is a set
  $A$ equipped with an operation $\sigma_A\colon A^n \to A$ for every
  $n$-ary operation symbol $\sigma$ and with a relation $r_A \subseteq A^m$
  for every $m$-ary relation symbol $r$.
\end{notation}

\begin{assumption} \label{assum:finite-sort}
  Throughout the paper, we assume that every signature has a finite set of sorts
  and finitely many relation symbols.
  There is no restriction on the number of operation symbols. 
\end{assumption}

\begin{remark} \label{remark:SigStr-limits}
  \begin{enumerate}
  \item The category $\SigStr$ is complete with limits created at the
    level of $\Set^{\mc{S}}$. More precisely, consider a diagram $D$
    in $\SigStr$ indexed by $\mc{I}$. Let
    $U^s\colon \Set^{\mc{S}} \to \Set$ be the projection sending $B$
    to $B^s$, and let
    \[
      b^s_i\colon B^s \to D_i^s\quad (i \in \mc{I})
    \]
    form limit cones of the diagrams $U^s D$ in $\Set$ for every
    $s \in \mc{S}$. Then the limit of $D$ is the $\mc{S}$-sorted set
    $B \defeq (B^s)$, with
    operations~$\sigma_B\colon B^{s_1} \times \dots \times B^{s_n} \to
    B^s$ uniquely determined by the requirement that each $b_i\colon B
    \to D_i$ preserves~$\sigma$, and with relations
    $r_B \subseteq B^{s_1} \times \dots \times B^{s_n}$ consisting of
    all $n$-tuples $(x_1, \dots, x_n)$ that each function
    $b^{s_1}_i \times \dots \times b^{s_n}_i$ maps into $r_{D_i}$ for
    all $i \in \mc{I}$.  The limit cone is given by
    $(b_i^s)_{s \in \mc{S}}\colon B\to D_i$ for $i\in \mc{I}$.

  \item The category $\SigStr$ is also cocomplete. Indeed, let
    $\Sigma_\op$ be the subsignature of all operation symbols in
    $\Sigma$. Then $\SigopStr$ is a monadic category over
    $\Set^{\mc{S}}$. Since epimorphisms split in $\Set^{\mc{S}}$, all
    monadic categories are cocomplete, see e.g.~\cite{Adamek1977}. The
    category $\SigStr$ has colimits obtained from the corresponding
    colimits in $\SigopStr$ by taking the smallest relations making
    each of the colimit injections a $\Sigma$-homomorphism.
  \end{enumerate}
\end{remark}

\begin{notation}
  The category of Stone topological $\Sigma$-structures and
  continuous $\Sigma$-homo\-mor\-phisms is denoted by
  \[
    \Stone(\SigStr).
  \]
  A \emph{topological $\Sigma$-structure} is an $\mc{S}$-sorted topological space $A = (A^s)$ endowed with a $\Sigma$-structure such that every operation $\sigma_s\colon A^{s_1} \times \dots \times A^{s_n} \to A$ is continuous and for every relation symbol $r$ the relation $r_A \subseteq A^{s_1} \times \cdots \times A^{s_n}$ is a closed subset.
\end{notation}
\begin{remark}\label{rem:stronesstr_limits}
The category $\Stone(\SigStr)$ is complete with limits formed on the level of~$\Set^\S$. This follows from the construction of limits in $\Stone^\S$ and in $\SigStr$.
Thus, the forgetful functor from $\Stone(\SigStr)$ to $\SigStr$ preserves limits.
\end{remark}
The following proposition describes the pro-completion of $\SigStr_\f$.
It is a categorical reformulation of results by Pin and
Weil~\cite{PinWeil1996} on topological $\Sigma$-structures, and also
appears in Johnstone's book~\cite[Prop.~\& Rem.~VI.2.4]{Johnstone1982} for the special case of single-sorted algebras.
We provide a full proof for the convenience of the reader.

\begin{defn}
A Stone topological $\Sigma$-structure is called \emph{profinite} if it is a cofiltered limit in $\Stone(\SigStr)$ of finite $\Sigma$-structures.
\end{defn}

\begin{prop}\label{prop:procomp-sigstr}
  The category $\Pro(\SigStr_\f)$ is the full subcategory of
  $\Stone(\SigStr)$ on all profinite $\Sigma$-structures.
\end{prop}
\begin{proof}
  \begin{enumerate}
  \item \label{item:profinite-Stone-1} We first observe that
    cofiltered limits of finite sets in $\Stone$ have the following
    property: If $b_i\colon B \to B_i$ $(i \in \mc{I})$ is a
    cofiltered limit cone such that all $B_i$ are finite, then for
    every $i \in \mc{I}$ there exists a connecting morphism of our
    diagram $h\colon B_j \to B_i$ with the same image as $b_i$:
      \begin{equation}
        b_i[B] = h[B_j].
      \end{equation}
      Since under Stone duality finite Stone spaces dualizes  to finite boolean algebras, it suffices to verify the dual statement about filtered colimits of finite Boolean algebras:
      if $c_i\colon C_i \to C$ ($i \in \mc{I}$) is a filtered colimit cocone of finite Boolean algebras, then for every $i$ there exists a connecting morphism $f\colon C_i \to C_j$ with the same kernel as $c_i$.
      But this is clear: given any pair $x, y \in C_i$ merged by $c_i$, there exists a connecting morphism $f$ merging $x$ and $y$, since filtered colimits are formed on the level of $\Set$.
      Due to $C_i \times C_i$ being finite, we can choose one $f$ for all such pairs.
      
    \item The argument is similar for cofiltered limits of finite $\Sigma$-structures in $\Stone(\SigStr)$:
      Consider a limit cone
      \[
        b_i \colon B \to B_i\qquad (i \in \mc{I})
      \]
      of a cofiltered diagram $D$ in $\Stone(\SigStr)$.
      For every $i \in \mc{I}$, we verify that there is a connecting morphism $h\colon B_j
      \to B_i$ with sorts $h^s$ for $s \in \mc{S}$ such that
      \begin{equation} \label{eq:profinite-connecting}
         b_i^s[B^s] = h^s[B_j^s]
        \qquad\text{for all $s \in \mc{S}$,}  
      \end{equation}
      and
      \begin{equation} \label{eq:profinite-relation} b_i^{s_1} \times
        \dots \times b_i^{s_n}[r_B] = h^{s_1} \times \dots \times
        h^{s_n}[r_B] \qquad \text{for all $r\colon s_1, \ldots, s_n$
          in $\Sigma$.}
      \end{equation}
      Indeed, if we only consider~\eqref{eq:profinite-connecting} then the existence of such an $h$ follows from~\ref{item:profinite-Stone-1} by the assumption that $\mc{S}$ is finite and that $\mc{I}$ is cofiltered.
      For every sort $s$, we have a cofiltered limit $b_j^s\colon B^s \to B_j^s$ in $\Stone$, thus we can apply~\ref{item:profinite-Stone-1} and obtain a connecting morphism $h\colon B_j \to B_i$.
      Again, $\mc{S}$ is finite, so the choice of $h$ can be made independent of $s \in \mc{S}$.

      Next consider~\eqref{eq:profinite-relation} for a fixed relation symbol
      $r\colon s_1, \ldots, s_n$.
      Form the diagram $D_r$ in $\Stone$ with the above diagram scheme $\mc{I}$ and with objects
      \[
        D_r i = r_{B_i} \text{ (a finite discrete space)}.
      \]
      Connecting morphisms are the domain-codomain restrictions of all connecting
      morphisms $B_j \xto{h} B_k$: since $h$ preserves the relation
      $r$, we have
      \[
        h^{s_1} \times \dots \times h^{s_n}[r_{B_j}] \subseteq r_{B_k},
      \]
      and we form the corresponding connecting morphism
      $\overline{h}\colon r_{B_j} \to r_{B_k}$ of $D_r$.  From the
      description of limits in $\SigStr$ in
      \Cref{remark:SigStr-limits} and the fact that limits in
      $\Stone(\SigStr)$ are preserved by the forgetful functor into
      $\SigStr$ by \Cref{rem:stronesstr_limits} we deduce that the
      limit of $D_r $ in $\Stone$ is the space
      $r_B \subseteq B^{s_1} \times \dots \times B^{s_n}$ and the
      limit cone $r_B \to r_{B_j}$, $j \in \mc{I}$, is formed by
      domain-codomain restrictions of
      $b_j^{s_1} \times \dots \times b_j^{s_n}$ for $j \in \mc{I}$.
      Apply~\ref{item:profinite-Stone-1} to this cofiltered limit to
      find a connecting morphism $h \colon B_j \to B_i$ of $D$ satisfying
      \eqref{eq:profinite-relation} for any chosen relation symbol~$r$
      of $\Sigma$.  Since we only have finitely many relation symbols
      by \Cref{assum:finite-sort}, we conclude that $h$ can be
      chosen to satisfy~\eqref{eq:profinite-relation}.

    \item Denote the full subcategory formed by profinite $\Sigma$-structures by
      \[
        \L \subseteq \Stone(\SigStr).
      \]
      In order to prove that $\L$ forms the pro-completion of
      $\SigStr_\f$, we verify the conditions given in
      \Cref{lem:char-of-procompletion}. By construction,
      conditions~\ref{itm:char-of-procompletion-1}
      and~\ref{itm:char-of-procompletion-3} hold.  It remains to
      prove condition~\ref{itm:char-of-procompletion-2}: every finite
      $\Sigma$-structure $A$ is finitely copresentable in $\L$.
      Hence, consider a limit cone
      \[
        b_i \colon B \to B_i\quad(i \in \mc{I})
      \]
      of a cofiltered diagram $D$ in $\L$.  Due to the definition of
      $\L$, each $B_i$ is a cofiltered limit of finite structures. Therefore, without loss of generality, we may assume that all $B_i$ are finite.  We need to show
      that for every homomorphism
      $f = (f^s)_{s \in \mc{S}} \colon B \to A$ into a finite
      $\Sigma$-structure $A=(A^s)_{s\in \S}$, there is an essentially
      unique factorization through some~$b_i$.  For every sort~$s$, we
      have a projection~$V^s\colon \L \to \Stone$, and the cofiltered
      diagram $V^s D$ has the limit cone
      $f_i^s\colon B^s \to B_i^s\;(i \in \mc{I})$.  Since each $A^s$
      is finite, the fact that $\Stone$ is the pro-completion of
      $\Set_\f$ implies that for every sort~$s$ there is $i \in \mc{I}$
      and an essentially unique factorization of $f^s$ as follows
        \[
          \xymatrix{
            B^s \ar[r]^{f^s} \ar[d]_{b_i^s} & A^s \\
            B_i^s \ar[ru]_{g^s} 
          }
        \]
        By \Cref{assum:finite-sort} the set $\mc{S}$ is finite, so we can choose $i$ independent of~$s$ and thus obtain a continuous $\mc{S}$-sorted function
        \[
          g = (g^s) \colon B_i \to A
          \quad\text{in } \Stone^{\mc{S}}
        \]
        which factorizes $f$, i.e.~$f = g \o b_i$.

        All we still need to prove is that we can choose our $i$
        and $g$ so that, moreover, $g$ is a $\Sigma$-homomorphism. 
        The essential uniqueness of $g$ then follows from the corresponding property of $g$ in $\Stone$.

        Let $h\colon B_j \to B_i$ be a connecting map
        satisfying~\eqref{eq:profinite-connecting}
        and~\eqref{eq:profinite-relation}. Choose $j$ in lieu of $i$ and
        $\overline{g} = g \o h$ in lieu of $g$. We conclude that $\overline{g}$
        is a morphism of $\Stone^{\mc{S}}$ factorizing $f$ through the limit map
        $b_j$:
        \[
          \xymatrix@-.5em{
            B \ar[dd]_{b_j} \ar[rr]^{f} \ar@{-->}[rd]^{b_i} & & A \\
                                                            & B_i \ar[ru]^{g} \\
            B_j \ar[ru]^{h} \ar `rd[ru] `[rruu]_{\overline{g}} [rruu]
          }
        \]
        Moreover, we prove that $\overline{g}$ is a $\Sigma$-homomorphism:
      \end{enumerate}
      \begin{enumerate}[label=(3\alph*)]
      \item For every operation symbol
        $\sigma\colon s_1\dots s_n \to s$ in $\Sigma$ and every
        $n$-tuple
        $(x_1, \ldots, x_n) \in B_j^{s_1} \times \dots \times
        B_j^{s_n}$ we have
            \[
              \overline{g}^s \o \sigma_{B_j}(x_1, \ldots, x_n) =
              \sigma_A\left(\overline{g}^{s_1}(x_1), \ldots,
              \overline{g}^{s_n}(x_n)\right).
            \]
            Indeed, choose $y_k \in B^{s_k}$ with
            $b_i^{s_k}(y_k) = h^{s_k}(x_k)$, $k = 1, \ldots, n$, using
            \eqref{eq:profinite-connecting}. Then
            \begin{align*}
              \overline{g}^s \o \sigma_{B_j}(x_1,\ldots,x_n)
              & = g^s \o h^s \o \sigma_{B_j}(x_1,\ldots,x_n)
              && \overline{g} = g \o h
              \\
              & = g^s \o\sigma_{B_i}(h^{s_1}(x_1),\ldots,
              h^{s_n}(x_n))
              && \text{$h$ a $\Sigma$-homomorphism}
              \\
              & = g^s \o \sigma_{B_i}(b_i^{s_1}(y_1),\ldots,b_i^{s_n}(y_n)) && b_i^{s_k}(y_k) =
              h^{s_k}(x_k) \\
              & = g^s \o b_i^s \o \sigma_B(y_1,\ldots,y_n) && \text{$b_i$ a $\Sigma$-homomorphism} \\
              & = \sigma_A(g^{s_1} \o b_i^{s_1}(y_1), \ldots, g^{s_n} \o b_i^{s_n}(y_n)) && \text{$g \o b_i = f$ a
              $\Sigma$-homomorphism} \\
              & = \sigma_A(g^{s_1} \o h^{s_1}(x_1),\ldots, g^{s_n} \o h^{s_n}(x_n) ) && b_i^{s_k}(y_k) =
              h^{s_k}(x_k) \\
              & = \sigma_A(\overline{g}^{s_1}(x_1),\ldots, \overline{g}^{s_n}(x_n)) && \overline{g} = g \o h.
            \end{align*}
          
            %\pagebreak

          \item For every relation symbol $r\colon s_1,\ldots,s_n$ in $\Sigma$, we have
            that 
            \[
              (x_1, \ldots, x_n) \in r_{B_j}
              \text{ implies }
              (\overline{g}^{s_1}(x_1), \ldots, \overline{g}^{s_n}(x_n)) \in
              r_A. 
            \]
            Indeed, using \eqref{eq:profinite-relation}, we can choose $(y_1, \ldots, y_n) \in
            r_B$ with
            \[
              (b_i^{s_1}(y_1), \ldots, b_i^{s_n}(y_n)) = (h^{s_1}(x_1), \ldots,
              h^{s_n}(x_n)). 
            \]
            Then the $n$-tuple
            \[
              (\overline{g}^{s_1}(x_1), \ldots, \overline{g}^{s_n}(x_n))
              = (g^{s_1}\o b_i^{s_1}(y_1), \ldots, g^{s_n} \o b_i^{s_n}(y_n))
            \]
            lies in $r_A$ because $g \o b_i = f$ is a $\Sigma$-homomorphism. \qedhere
        \end{enumerate}
\end{proof}

\begin{notation}
  Let $\D$ be a full subcategory of $\SigStr$.
  We denote by
  \[
    \Stone\D
  \]
  the full subcategory of $\Stone(\SigStr)$ on all Stone topological $\Sigma$-structures whose $\Sigma$-structure lies in $\D$.
  Moreover, let $\Df$ denote the full subcategory of $\D$ on all finite objects, i.e.~$D\in \Df$ if each $D^s$ is finite.
\end{notation}

\begin{cor}\label{coro:pro-Df} Let $\D$ be a full subcategory
  of $\SigStr$ closed under cofiltered limits.  Then $\Pro{\Df}$ is
  the full subcategory of $\Stone\D$ given by all profinite $\D$-structures,
  i.e.~cofiltered limits of finite $\Sigma$-structures in $\D$.
\end{cor}
The proof is completely analogous to that of \Cref{prop:procomp-sigstr}: the
only fact we used in that proof was the description of cofiltered limits in
$\SigStr$. 

%\pagebreak

\begin{expl}
  For $\D = \Pos$, we get an alternative description of the category
  $\Priest$ of \Cref{ex:pro}\ref{ex:pro:2}.  For the signature $\Sigma$ with
  a single binary relation, $\Pos$ is a full subcategory of $\SigStr$.
  The category $\Stone(\SigStr)$ is that of graphs on Stone spaces.
  By \Cref{coro:pro-Df}, $\Pro(\Pos_\f)$ is the category of all
  profinite posets, i.e.~Stone graphs that are cofiltered limits of
  finite posets.  Note that every such limit $B = (V, E)$ is a poset:
  given $x \in V$ we have $(x, x) \in E$ because every object of the
  given cofiltered diagram has its relation reflexive.  Analogously,
  $E$ is transitive and (since limit cones are collectively monic)
  antisymmetric.

  Moreover, $B$ is a Priestley space: given $x, y \in V$ with
  $x \not\leq y$, then there exists a member $b_i\colon B \to B_i$ of
  the limit cone with $b_i(x) \not\leq b_i(y)$.  Since $B_i$ is
  finite, and thus carries the discrete topology, the upper set
  $b_i^{-1}(\mathord{\uparrow} x)$ is clopen, and it contains $x$ but
  not $y$.  Conversely, every Priestley space is a profinite poset, as
  shown by Speed~\cite{speed72}.
\end{expl}
\begin{expl}\label{ex:prodf}
  Johnstone~\cite[Thm.~VI.2.9]{Johnstone1982} proves that for a number
  of ``everyday'' varieties of algebras~$\D$, we simply have
  \[
    \Pro{\Df} = \Stone\D.
  \]
  This holds for semigroups, monoids, groups, vector spaces,
  semilattices, distributive lattices, etc.  In contrast, for some
  important varieties $\Pro{\Df}$ is a proper subcategory of
  $\Stone\D$, e.g.~for the variety of lattices or the variety of
  $\Sigma$-algebras where $\Sigma$ consists of a single unary
  operation.
\end{expl}

\begin{remark} \label{remark:factorization}
  \begin{enumerate}
  \item The category $\SigStr$ has a factorization system $(\E,\M)$ where $\E$ consists of all surjective $\Sigma$-homomorphisms (more precisely, every sort is a surjective function) and $\M$ consists of all injective $\Sigma$-homomorphisms reflecting all relations.
    That is, a $\Sigma$-homomorphism $f\colon X \to Y$ lies in $\M$
    iff for every sort~$s$ the function $f^s\colon X^s \to Y^s$ is injective, and for every relation symbol $r\colon s_1,\ldots, s_n$ in $\Sigma$ and every $n$-tuple $(x_1,\ldots, x_n)\in X^{s_1}\times\ldots\times X^{s_n}$ one has
    \[
      (x_1,\ldots,x_n)\in r_X \quad\text{iff}\quad (f^{s_1}(x_1),\ldots, f^{s_n}(x_n))\in r_Y.
    \]
    The $(\E,\M)$-factorization of a $\Sigma$-homomorphism
    $g\colon X \to Z$ is constructed as follows.  Define a
    $\Sigma$-structure $Y$ by $Y^s = g^s[X^s]$ for all sorts
    $s \in \mc{S}$, let the operations of $Y$ be the domain-codomain
    restriction of those of $Z$, and for every relation symbol
    $r\colon s_1,\ldots, s_n$ define $r_Y$ to be the restriction of
    $r_Z$ to $Y$,
    i.e.~$r_Y = r_Z \cap Y^{s_1}\times\ldots\times Y^{s_n}$.  Then the
    codomain restriction of $g$ is a surjective $\Sigma$-homomorphism
    $e\colon X \epito Y$, and the embedding $m\colon Y \monoto Z$ is a
    injective $\Sigma$-homomorphism reflecting all relations.

  \item\label{R:fs:2} Similarly, the category $\Stone(\SigStr)$ has the
    factorization system $(\E,\M)$ where $\E$ consists of all
    surjective morphisms and $\M$ of all relation-reflecting
    monomorphisms.  Indeed, if $f\colon X \to Z$ is a continuous
    $\Sigma$-homomorphism, and if its factorization in $\SigStr$ is given by a
    $\Sigma$-structure $Y$ and $\Sigma$-homomorphisms
    $e\colon X \epito Y$ (surjective) and $m\colon Y \monoto Z$ (injective
    and relation-reflecting), then the Stone topology on $Y$ inherited
    from $Z$ yields, due to $Y = e[X]$ being closed in $Z$, the
    desired factorization in $\Stone(\SigStr)$.
  \end{enumerate}
\end{remark}

\takeout{
\begin{prop}
  Let $\D$ be a full subcategory of $\SigStr$ closed under products and
  subobjects. Then every object of $\Pro{\Df}$ is a cofiltered limit of all of
  its quotients in $\Df$. 
\end{prop}

\begin{proof}
  \begin{enumerate}
  \item We work here with $\Pro\Df$ as described in
    \Cref{coro:pro-Df}. Then every object $X$ is a cofiltered limit of
    a diagram $R\colon I \to \Df$.  Let $b_i \colon X \to B_i$
    ($i \in I$) be a limit cone for $R$. Factorize $b_i = m_i \o e_i$
    with $e_i \colon X \epito A_i$ surjective and
    $m_i\colon A_i\monoto B_i$ monic and reflecting relations, see
    \Cref{remark:factorization}. Then $B_i \in \Df$
      implies $A_i \in \Df$. We get a quotient functor
      $\overline{R}\colon I \to \Df$ of $R$ with $\overline{R}(i) = A_i$. To
      every morphism $h\colon i \to j$ of $I$ it assigns the domain-codomain
      restriction $\overline{R} h$ of the connecting morphism $\overline{R}h$
      obtained by the diagonal fill-in:
      \[
        \xymatrix@C-1.5em{
          & & X \ar@{->>}[ld]_{e_i} \ar@{->>}[rd]^{e_j} \\
          & A_i \ar@{ >->}[ld]_{m_i} \ar@{-->}[rr]_{\overline{R}h} & & A_j \ar@{ >->}[rd]^{m_j} \\
          B_i \ar[rrrr] & & & & B_j
        }
      \]
      Since $I$ is a cofiltered category, $\overline{R}$ is a cofiltered diagram
      of finite quotients of $X$. We claim that $e_i\colon X \epito A_i$ ($i \in
      I$) is a limit cone of $\overline{R}$. Indeed, given a cone $f_i\colon Z
      \to A_i$ ($i \in I$) of $\overline{R}$, then for the cone $m_i \o
      f_i\colon Z \to B_i$ of $D$ we have a unique $f\colon Z \to X$ with
      $b_i\o f = m_i \o f_i$. Since $m_i$ is monic, this implies $e_i \o f =
      f_i$. 

      It remains to verify that the diagram of all finite quotients of $X$ has
      the same limit.  The diagram of all finite quotients of $X$ is cofiltered due to $\Df$ being closed under finite products and subobjects in $\D$. Thus, it is sufficient to prove that all $e_i$, $i \in I$, form a final subdiagram, i.e.~for every quotient $q\colon X \to C$ in $\Df$ there exists $i \in I$ with $q \leq e_i$.
      Indeed, the diagram $\overline{R}$ is cofiltered and $C \in \Df$, hence, $q$ factorizes through some of the limit morphisms.\qedhere
  \end{enumerate}
\end{proof}
}
\begin{remark}\label{rem:arrowcat}
  Recall that the \emph{arrow category} $\A^\to$ of a category $\A$
  has as objects all morphisms $f\colon X\to Y$ in $\A$. A morphism
  from $f\colon X\to Y$ to $g\colon U\to V$ in $\A^\to$ is given by a
  pair of morphisms $m\colon X\to U$ and $n\colon Y\to V$ in $\A$ with
  $n\o f = g\o m$.  Identities and composition are defined
  componentwise.  If $\A$ has limits of some type, then also $\A^\to$
  has these limits, and the two projection functors from $\A^\to$ to
  $\A$ mapping an arrow to its domain or codomain, respectively,
  preserve them.
\end{remark}
\begin{lem}\label{lem:epi}
\begin{enumerate}
\item\label{itm:cofilteredepi} For every cofiltered diagram $D$ in $\Set_\f$ with epic connecting maps, the limit cone of $D$ in $\Stone$ is formed by epimorphisms.
\item\label{itm:cofilteredepiarrowcat} For every cofiltered diagram $D$ in $\Stone^\to$ whose objects are epimorphisms in $\Stone$, also $\lim D$ is epic.
\end{enumerate}
\end{lem}

\begin{proof}
  These properties follow easily from standard results about
  cofiltered limits in the category of compact Hausdorff spaces, see
  e.g. Ribes and Zalesskii~\cite[Sec.~1]{Ribes2010}. Here, we give
  an alternative proof using Stone duality, i.e.~we verify that the
  category $\BA$ of boolean algebras satisfies the statements dual
  to~\ref{itm:cofilteredepi} and~\ref{itm:cofilteredepiarrowcat}.

  The dual of~\ref{itm:cofilteredepi} states that a filtered diagram
  of finite boolean algebras with monic connecting maps has a colimit
  in $\BA$ whose colimit maps are monic. This follows from the fact
  that filtered colimits in $\BA$ are created by the forgetful functor
  to $\Set$, and that filtered colimits of monics in $\Set$ clearly
  have the desired property.

  Similarly, the dual of~\ref{itm:cofilteredepiarrowcat} states that a
  filtered colimit of monomorphisms in $\BA^\to$ is a monomorphism,
  which follows from the corresponding property in $\Set^\to$.
\end{proof}

\section{Pseudovarieties}\label{sec:pseudovar}

In universal algebra, a pseudovariety of $\Sigma$-algebras is defined to
be a class of finite algebras closed under finite products,
subalgebras, and quotient algebras. In the present section, we
introduce an abstract concept of pseudovariety in a given category
$\D$ with a specified full subcategory $\Df$. The objects of $\Df$
are called ``finite'', but this is just terminology. Our approach
follows the footsteps of Banaschewski and
Herrlich~\cite{Banaschewski1976} who introduced varieties of objects
in a category $\D$, and proved that they are precisely the full
subcategories of $\D$ presentable by an abstract notion of equation
(see \Cref{D:BH}). Here, we establish a similar result for
pseudovarieties: they are precisely the full subcategories of $\Df$
that can be presented by pseudoequations
(\Cref{prop:pseudovariety=pseudoequations}), which are shown to be
equivalent to profinite equations in many examples
(\Cref{thm:reiterman}).

\begin{assumption}\label{assum:factorization-system}
  For the rest of our paper, we fix a complete category $\D$ with a
  proper factorization system $(\E, \M)$, that is, all morphisms in
  $\E$ are epic and all morphisms in $\M$ are monic. \emph{Quotients}
  and \emph{subobjects} in $\D$ are represented by morphisms in $\E$
  and $\M$, respectively, and denoted by $\epito$ and
  $\monoto$. Moreover, we fix a small full subcategory $\Df$ whose
  objects are called the \emph{finite} objects of $\D$, and denote by
  $\E_\f$ and $\M_\f$ the morphisms of $\Df$ in $\E$ or $\M$,
  respectively. We assume that
  \begin{enumerate}
  \item the category $\Df$ is closed under finite limits and subobjects, and
  \item\label{A:fs:2} every object of $\Df$ is a quotient of some projective object of $\D$. 
  \end{enumerate}
  Here, recall that an object $X$ is called \emph{projective} (more
  precisely, \emph{$\E$-projective}) if for every quotient
  $e \colon P \epito P'$ and every morphism $f\colon X \to P'$ there
  exists a morphism $g\colon X\to P$ with $e\o g = f$.
\end{assumption}

%In~\cite{Banaschewski1976} only regular factorization were considered, that all
%concepts of that paper make sense for proper factorization systems, and the
%results we now recall have completely analogous proofs to those of \emph{op.\
%cit.}

%\pagebreak

\begin{defn}[Banaschewski and Herrlich~\cite{Banaschewski1976}]\label{D:BH}
  \begin{enumerate}
    \item A \emph{variety} is a full subcategory of $\D$ closed under products,
      subobjects, and quotients.
    \item An \emph{equation} is a quotient $e\colon X\epito E$ of a projective object
      $X$. An object $A$ is said to \emph{satisfy} the equation
      $e$ provided that $A$ is
      \emph{injective} w.r.t.~$e$, that is, if for every morphism
      $g\colon X\to A$ there exists a morphism $h\colon E\to A$ making
      the triangle below commute:
\[ 
\xymatrix@C-1.5em
{
& X \ar@{->>}[dl]_e \ar[dr]^g & \\
E \ar@{-->}[rr]_h && A
}
 \]
  \end{enumerate}
\end{defn}
We note that Banaschewski and Herrlich worked with the factorization
system of regular epimorphisms and monomorphisms. However, all their
results and proofs apply to general proper factorization systems, as
already pointed out in their paper~\cite{Banaschewski1976}.

\begin{expl}\label{ex:eqsigmaalg}
  Let $\Sigma$ be a one-sorted signature of operation symbols. If
  $\D = \Sigma\text{-}\Alg$ is the category of $\Sigma$-algebras with
  its usual factorization system
  ($\E = \text{surjective homomorphisms}$ and $\M =$ injective
  homomorphisms), then the above definition of a variety gives the
  usual concept in universal algebra: a class of $\Sigma$-algebras
  closed under product algebras, subalgebras, and homomorphic images. Moreover, equations in the above categorical sense are expressively equivalent to equations  $t = t'$ between $\Sigma$-terms in the usual sense:
\begin{enumerate}
\item Given a term equation $t = t'$, where
  $t, t' \in T_\Sigma X_0$ are taken from the free algebra of all $\Sigma$-terms in the set
  $X_0$ of variables, let $\sim$ denote the least
  congruence on $T_\Sigma X_0$ with $t \sim t'$. The corresponding quotient morphism
  $e\colon T_\Sigma X_0 \twoheadrightarrow T_\Sigma
  X_0/\mathord{\sim}$ is a categorical equation satisfied by precisely those
  $\Sigma$-algebras that satisfy $t = t'$ in the usual sense.

\item  Conversely, given a projective $\Sigma$-algebra $X$ and a surjective
  homomorphism $e\colon X \twoheadrightarrow E$, then for any set
  $X_0$ of generators of $X$ we have a split epimorphism
  $q\colon T_\Sigma X_0 \epito X$ using the projectivity of $X$.
  Consider the set of term equations $t = t'$ where $(t, t')$ ranges over
  the kernel of $e \o q\colon T_\Sigma X_0 \epito E$. Then a
  $\Sigma$-algebra $A$ satisfies all these equations iff it satisfies
  $e$ in the categorical sense.
\end{enumerate}
\end{expl}
%\pagebreak
Recall that the category $\D$ is \emph{$\E$-co-well-powered} if for every object $X$ of $\D$ the quotients with domain $X$ form a small set.
\begin{thm}[Banaschewski and Herrlich~\cite{Banaschewski1976}] Let $\D$ be a category with a proper factorization system $(\E,\M)$. Suppose that $\D$ is complete, $\E$-co-well-powered, and has enough projectives, i.e.~every object is a quotient of a projective one.
  Then, a full subcategory of $\D$ is a variety iff it can be presented by a class of equations. That is, it consists of precisely those objects satifying each of these equations.
\end{thm}

Note that the category of $\Sigma$-algebras satisfies all conditions of the theorem. Thus, in view of \Cref{ex:eqsigmaalg}, Banaschewski and Herrlich's result subsumes Birkhoff's HSP theorem~\cite{Birkhoff35}. In the following, we are going to move from varieties in $\D$ to pseudovarieties in $\Df$.

\begin{defn}
  A \emph{pseudovariety} is a full subcategory of $\Df$ closed under finite
  products, subobjects, and quotients.
\end{defn}

%\begin{expl}
%  For every signature $\Sigma$ we consider the category $\D = \SigStr$ always
%  together with the factorization system of
%  \Cref{remark:factorization}, i.e.~$\E=$ surjective morphisms and 
%  $\M=$ relation-reflecting monomorphisms, and take $\Df$ to be the full subcategory $\SigStr_\f$ of finite $\Sigma$-structures.
%\end{expl}
%
\begin{remark}
  Quotients of an object $X$ are ordered by factorization: 
  given $\E$-quotients $e_1, e_2$, we put $e_1 \leq e_2$ if $e_1$
  factorizes through $e_2$
  \[
    \xymatrix@C-1.5em{
      & X \ar@{->>}[ld]_{e_1} \ar@{->>}[rd]^{e_2} \\
      E_1 & & E_2 \ar@{-->}[ll]
    }
  \]
  Every pair of quotients $e_i\colon X \epito E_i$ has a least upper
  bound, or \emph{join}, $e_1\vee e_2$ obtained by
  $(\E,\M)$-factorizing the mediating morphism
  $\<e_1, e_2\>\colon X \to E_1 \times E_2$ as follows:
  \begin{equation} \label{eq:join-quotient}
    \begin{gathered}
      \xymatrix@C+1em{
      X \ar@{->>}[d]_{e_i} \ar@{->>}[r]^{e_1 \vee e_2} \ar[rd]|{\<e_1, e_2\>} & F \ar@{ >->}[d] \\
        E_i & E_1 \times E_2 \ar[l]^-{\pi_i}.
      }
    \end{gathered}
  \end{equation}
  A nonempty collection of quotients closed under joins is called a
  \emph{semilattice of quotients}.
\end{remark}

\begin{defn}
  A \emph{pseudoequation} is a semilattice $\rho_X$ of quotients of a
  projective object $X$ (of ``variables''). A finite object $A$ of
  $\D$ \emph{satisfies} $\rho_X$ if $A$ is cone-injective
  w.r.t. $\rho_X$, that is,
  %\pagebreak
  for every morphism $h\colon X \to A$, there exists a member $e\colon X
  \twoheadrightarrow E$ of $\rho_X$ through which $h$ factorizes:
  \[
    \xymatrix@C-1.5em{
      & X \ar@{->>}[ld]_{\exists e} \ar[rd]^{\forall h} \\
      E \ar@{-->}[rr]^-{\exists} && A 
    }
  \]
\end{defn}

%\begin{expl}
%  Prior to Reiterman's paper~\cite{Reiterman1982} pseudovarieties of $\Sigma$-algebras
%  (where $\Sigma$ is a one-sorted algebraic signature) were described
%  by \emph{sequences of equations}:
%  Given a set $X_0$ of variables and a sequence
%  \[
%    t_n = t_n'
%    \quad (n \in \Nat)
%  \]
%  of pairs in $T_\Sigma X_0$, an algebra $A$ \emph{satisfies} this sequence provided
%  that every homomorphism $h\colon T_\Sigma X_0 \to A$ (i.e.~every
%  interpretation of the variables in $A$) fulfils $h(t_n) = h(t_n')$ for all but
%  finitely many $n$. 
%
%\HU[inline]{Reference? This is not the concept of satisfaction that Eilenberg and Schützenberger considered.}
%
%  Let $\sim_k$ be the least congruence of $T_\Sigma X_0$ with $t_n \sim t_n'$
%  for all $n \leq k$. Then the collection of quotients 
%  \[
%    e_k\colon T_\Sigma X_0 \twoheadrightarrow T_\Sigma X_0/\mathord{\sim_k},
%    \quad 
%    \text{for all $k \in \Nat$,}
%  \]
%  is a pseudoequation. And a finite algebra satisfies this pseudoequation iff it
%  satisfies the above sequence. 
%\HU[inline]{The converse direction from filters to sequences of equations between terms is nontrivial and works only for finite signatures}
%\end{expl}

%\pagebreak

\begin{prop}\label{prop:pseudovariety=pseudoequations}
  A collection of finite objects of $\D$ forms a pseudovariety iff it can be presented by pseudoequations, i.e.~it consists of precisely those finite objects that satisfy each of the given pseudoequations.
\end{prop}
\begin{proof}
  \begin{enumerate}
  \item\label{itm:prop:pseudovariety=pseudoequations-1} We first
    prove the \emph{if} direction. Since the intersection of a family
    of pseudovarieties is a pseudovariety, it suffices to prove that
    for every pseudoequation $\rho_X$ over a projective object $X$,
    the class $\mathscr{V}$ of all finite objects satisfying $\rho_X$
    forms a pseudovariety, i.e.~is closed under finite products,
    subobjects, and quotients.
    \begin{enumerate}
    \item \emph{Finite products.} Let $A, B\in \mathscr{V}$. Since $A$
      and $B$ satisfy $\rho_X$, for every morphism
      $\<h, k\>\colon X \to A \times B$ there exists
      $e\colon X\epito E$ in $\rho_X$ such that both $h\colon X\to A$
      and $k\colon X\to B$ factorize through $e$ -- this follows from
      the closedness of pseudoequations under binary joins.  Given
      $h = e \o h'$ and $k = e \o k'$, then
      $\<h', k'\>\colon X \to E_i$ is the desired factorization:
      \[
        \<h, k\> = e \o \<h', k'\>.
      \]
      Thus $A\times B\in \mathscr{V}$. Since the terminal object $1$ clearly
      satisfies every pseudoequation, we also have $1\in \mathscr{V}$.

    \item \emph{Subobjects.} Let $m \colon A \monoto B$ be a morphism
      in $\M_\f$ with $B\in \mathscr{V}$.  Then for every morphism
      $h\colon X \to A$ we know that $m \o h$ factorizes as $e \o k$
      for some $e\colon X\epito E$ in $\rho_X$ and some
      $k \colon E \to B$. The diagonal fill-in property then shows
      that $h$ factorizes through $e$:
      \[
        \xymatrix{
          X \ar[d]_{h} \ar@{->>}[r]^{e} & E \ar[d]^{k} \ar@{-->}[ld] \\
          A \ar@{ >->}[r]_m & B
        }
      \]
      Thus, $A\in \mathscr{V}$.
      
    \item \emph{Quotients.} Let $q\colon B \epito A$ be a morphism in
      $\E_\f$ with $B \in \mathscr{V}$. Every morphism
      $h\colon X \to A$ factorizes, since $X$ is projective, as
      \[
        h = q\o k
        \quad\text{for some $k\colon X \to B$}
      \]
      Since $k$ factorizes through some $e\in \rho_X$, so does
      $h$. Thus, $A\in\mathscr{V}$.
    \end{enumerate}
    
  \item For the ``only if'' direction, suppose that $\mathscr{V}$ is a
    pseudovariety. For every projective object $X$ we form the
    pseudoequation $\rho_X$ consisting of all quotients
    $e\colon X \twoheadrightarrow E$ with $E \in \mathscr{V}$. This is
    indeed a semilattice: given $e,f\in \rho_X$ we have
    $e\vee f\in \rho_X$ by \eqref{eq:join-quotient}, using that
    $\mathscr{V}$ is closed under finite products and subobjects. We
    claim that $\mathscr{V}$ is presented by the collection of all the
    above pseudoequations $\rho_X$.
    \begin{enumerate}
    \item Every object $A \in \mathscr{V}$ satisfies all
      $\rho_X$. Indeed, given a morphism $h\colon X \to A$, factorize
      it as $e \colon X \twoheadrightarrow E$ in $\E$ followed by
      $m \colon E \rightarrowtail A$ in $\M$. Then $E \in \mathscr{V}$
      because $\mathscr{V}$ is closed under subobjects, so $e$ is a
      member of $\rho_X$. Therefore $h=m\o e$ is the desired
      factorization of $h$, proving that $A$ satisfies $\rho_X$.
          
    \item Every finite object $A$ satisfying all the pseudoequations
      $\rho_X$ lies in $\mathscr{V}$. Indeed, by
      \Cref{assum:factorization-system} there exists a quotient
      $q\colon X \twoheadrightarrow A$ for some projective object
      $X$. Since $A$ satisfies $\rho_X$, there exists a factorization
      $q = h \o e$ for some $e\colon X\epito E$ in $\rho_X$ and some
      $h \colon E \to A$. We know that $E \in \mathscr{V}$, and from
      $q \in \E$ we deduce $h \in \E$. Thus $A$, being a quotient of
      an object of $\mathscr{V}$, lies in $\mathscr{V}$.\qedhere
    \end{enumerate}
  \end{enumerate}
\end{proof}

\begin{remark}\label{re:projective-objects}
  \begin{enumerate}
    \item \Cref{prop:pseudovariety=pseudoequations} would remain valid if we defined pseudoequations
      as semilattices of \emph{finite} quotients of a projective object. 
      This follows immediately from the above proof. 
      %\pagebreak

    \item Let us assume that a collection $\mathsf{Var}$ of projective
      objects of $\D$ is given such that every finite object is a
      quotient of an object of $\mathsf{Var}$
      (cf.~\Cref{assum:factorization-system}\ref{A:fs:2}). Then we could
        define pseudoequations as semilattices of quotients of members of
        $\mathsf{Var}$ with finite codomains.  Again, from the above
        proof we see that \Cref{prop:pseudovariety=pseudoequations}
        would remain true.
  \end{enumerate}
\end{remark}
We would like to reduce pseudoequations to equations in the sense of
Banaschewski and Herrlich. For that we need to move from the category
$\D$ to the pro-completion of~$\Df$.
\begin{notation}\label{not:v}
  Since $\D$ has (cofiltered) limits and $\Pro\Df$ is the free completion of $\Df$ under cofiltered limits, the embedding $\Df\monoto \D$
  extends to an essentially unique cofinitary functor
  \[
    V\colon \Pro\Df \to \D.
  \]
\end{notation}

\begin{expl}
  If $\D$ is a full subcategory of $\SigStr$ closed under cofiltered limits, we have seen that $\Pro{\Df}$ can be described as a full subcategory of $\Stone{\D}$ by \Cref{coro:pro-Df}.
  The above functor 
  \[
    V\colon \Pro\Df \to \D 
  \]
  is the functor forgetting the topology.  Indeed, the corresponding
  forgetful functor from $\Stone(\SigStr)$ to $\SigStr$ is cofinitary,
  hence, so is $V$.
\end{expl}

\begin{remark}\label{rem:kanext0}
  Recall, e.g.~from Mac Lane~\cite{maclane}, that the \emph{right Kan
    extension} of a functor $F\colon \A \to \Cat$
  along~$K\colon \A \to \B$ is a
  functor~$R=\Ran_{K} F\colon \B \to \Cat$ with a universal natural
  transformation $\epsilon\colon RK \to F$, that is, for every functor
  $G\colon \B \to \Cat$ and every natural
  transformation~$\gamma\colon GK \to F$ there exists a unique natural
  transformation $\gamma^\dagger\colon G \to R$ with
  $\gamma = \epsilon \o \gamma^\dagger K$. If $\A$ is small and $\Cat$
  is complete, then the right Kan extension exists
  \cite[Theorem X.3.1, X.4.1]{maclane}, and the object $RB$ ($B\in \B$) can be constructed as
  the limit
  \[
    RB = \lim(B/K \xto{Q_B} \A \xto{F} \Cat),
  \]
  where $B/K$ denotes the slice category of all morphisms
  $f\colon B\to KA$ ($A\in \A)$ and $Q_B$ is the projection functor
  $f\mapsto A$. Equivalently, $RB$ is given by the end
  \[
    RB =  \int_{A\in \A} \B(B,KA) \pitchfork FA,
  \]
  with $S\pitchfork C$ denoting $S$-fold power of $C\in \Cat$.
\end{remark}
\begin{lem}\label{lem:the-left-adjoint-to-V}
  The functor $V$ has a left adjoint
  \[
    \widehat{(\dash)}=\Ran_J E\colon \D \to \Pro\Df
  \]
given by the right Kan extension of the embedding $E\colon \Df\monoto \Pro{\Df}$ along the embedding $J\colon \Df\monoto \D$ and making the following triangle commute up to isomorphism:
\[  
\xymatrix{
\Df \ar@{>->}[rr]^J \ar@{>->}[dr]_E && \D \ar[dl]^{\widehat{(\dash)}} \\
& \Pro{\Df}  &
}
\]
\end{lem}
\begin{proof}
Recall that, up to equivalence, $\Pro\Df$ is the full subcategory
  of $[\Df, \Set]^\op$ on cofiltered limits of representables with
  $ED = \Df(D, -)$ for every $D\in \Df$ (see \Cref{rem:profinitecompletion}), and the functor $V$ is given by \[ V= \Ran_E J\colon \Pro{\Df}\to \D.\] Consider the following
  chain of isomorphisms natural in $D\in \D$ and $H\in \Pro{\Df}$:
  \begin{align*}
    \D(D, VH) & \cong \D(D, \left(\Ran_{E} J\right) H) \\
    & \cong \D\left(D, \int_{X} \Pro\Df(H, E X) \pitchfork JX\right) 
    && \text{by the end formula for $\Ran$,}\\
    & = \D\left(D, \int_{X} [\Df, \Set](E X, H) \pitchfork JX \right) 
    && \text{$\Pro\Df$ full subcategory of $[\Df, \Set]^\op$,}\\
    & \cong \D\left(D, \int_{X} HX \pitchfork JX\right) 
    && \text{by the Yoneda lemma,}\\
    & \cong \int_{X} \D(D, HX \pitchfork JX)
    && \text{$\D(D, -)$ preserves ends,}\\
    & \cong \int_{X} \Set(HX, \D(D, JX)) 
    && \text{by the universal property of power,}\\
    & \cong [\Df, \Set](H, \D(D, J-)) 
    && \text{the set of natural transf. as an end,}\\
    & = \Pro\Df(\D(D, J-), H)
    && \text{$\Pro\Df$ full subcategory of $[\Df, \Set]^\op$.}
    \end{align*}
    Hence, the functor $\widehat{(\dash)}\colon D \mapsto \D(D,J-)$ is a
    left adjoint to $V$. Moreover, $\widehat{(\dash)}$ extends $E$: for each $D \in \Df$, we have
    \[
      \widehat D = \D(JD,J-) = \Df(D, -) = ED,
    \]
    and similarly on morphisms, since $J$ is a full inclusion. It remains to verify that the functor $\widehat{(\dash)}$ coincides with $\Ran_J E$. This follows from
    the fact that every presheaf is a canonical colimit of
    representables expressed as a coend in $[\Df, \Set]$:
    \[
      \D(D, J-) \cong \int^{X} \D(D, JX) \bullet EX,
    \]
with $\bullet$ denoting copowers. This corresponds to an end in $[\Df, \Set]^\op$:
    \[
      \int_X \D(D, JX) \pitchfork E X = (\Ran_J E) D.
    \]
Thus $\widehat{(\dash)}=\Ran_J E$, as claimed.
\end{proof}

\begin{construction}\label{constr:the-left-adjoint-to-V}
  By expressing the right Kan extension
\[ \widehat{(\dash)}=\Ran_J E\colon \D\to\Pro{\Df} \] as a limit, the action $D\to \widehat{D}$ on objects, $f\mapsto \widehat{f}$ on morphisms, the unit, and the counit of the adjunction $\widehat{(\dash)}\dashv V$ are given as follows.
\begin{enumerate}
\item\label{C:laV:1} For every object $D$ of $\D$, the object $\widehat{D}\in \Pro{\Df}$ is a limit of the diagram 
  \[
    P_D \colon D/{\Df} \to \Pro\Df, \quad P_D(D \xto{a} A) = A.
  \]
We use the following notation for the limit cone of $P_D$:
  \[
    \frac{D \xto{a} A}{\widehat{D} \xto{\widehat{a}} A}
  \]
  where $(A, a)$ ranges over $D/{\Df}$. For finite $D$ we choose the
  trivial limit: $\widehat{D} = D$ and $\widehat{a} = a$.
  
\item Given $f\colon D \to D'$ in $\D$, the morphisms
  $\widehat{a\o f}$ with $a$ ranging over $D'/\Df$ form a cone over
  $P_{D'}$. Define $\widehat{f}\colon \widehat{D}\to\widehat{D'}$ to
  be the unique morphism such that the following triangles commute for
  all $a \colon D' \to A$ with $A \in \Df$:
  \[
    \xymatrix@C-1.5em{
      \widehat{D} \ar[rr]^{\widehat{f}} \ar[rd]_{\widehat{a\cdot f}} & &
      \widehat{D'} \ar[ld]^{\widehat{a}} \\
                  & A
    }
  \]
  Note that overloading the notation $\widehat{(-)}$ causes no problem
  because if $\widehat{D'} = D' \in \Df$ then $\widehat{f}$ is
  a projection of the limit cone for $P_{D}$ (see
  item~\ref{C:laV:1}), since for $a = \id_{D'}$ we have
  $\widehat{a} = \id_{D'}$.
\item The unit $\eta$ at $D\in \D$ is given by the unique
  morphism \[\eta_D\colon D \to V \widehat{D}\] in $\D$ such that the
  following triangles commute for all $h \colon D \to A$ with
  $A \in \Df$:
  \[
    \xymatrix{
      D \ar[r]^{\eta_D} \ar[dr]_h & V\widehat{D} \ar[d]^{V\widehat{h}} \\
      &  A 
    }
  \]
  Here one uses that $V$ is cofinitary, and thus the morphisms
  $V\widehat h$ form a limit cone in $\D$.
  
\item\label{C:laV:4} The counit $\epsilon$ at $D\in \Pro{\Df}$ is the unique
  morphism \[\epsilon_D\colon \widehat{VD} \to D\] in $\hatD$ such
  that the following triangles commute, where $a\colon D \to A$ ranges
  over the slice category $D/{\Df}$:
  \[
    \xymatrix@C-1.5em{
      \widehat{VD} \ar[rr]^{\epsilon_D} \ar[rd]_{\widehat{Va}}
      & &
      D \ar[ld]^{a}
      \\
      & A
    }
  \]
\end{enumerate}
\end{construction}

\begin{notation}
 Recall that $\E_\f$ and $\M_\f$ are the morphisms of $\Df$ in $\E$ and $\M$, respectively. We denote by
  \[
    \widehat{\E} \qquad\text{and}\qquad\widehat{\M}
  \]
  the collection of all morphisms of $\Pro{\Df}$ that are cofiltered limits of members of $\E_\f$ or $\M_\f$ in the arrow category
  $(\Pro\Df)^\rightarrow$, respectively. 
\end{notation}

\begin{remark} \label{ex:profinite}
  \begin{enumerate}
\item Let us recall that a functor $P:\J\to \J'$ is \emph{final} if
\begin{enumerate}[label=(\alph*)]
\item for every object $J'$ of $\J'$ a morphism from $J'$ into $PJ$ exists for some $J\in \J$;
\item given two morphisms $f_i\colon J'\to PJ_i$ ($i=1,2$) there exist morphisms $g_i\colon J_i\to J$ in $\J$ with $Pg_1\o f_1 = Pg_2\o f_2$. 
\end{enumerate}
Finality of $P$ implies that for every diagram $D\colon \J'\to \K$ one has $\colim D = \colim D\o P$ whenever one of the colimits exists. The dual concept is that of an \emph{initial functor} $P\colon \J\to \J'$. 
    \item\label{itm:ex:profinite-1} For every finite $\E$-quotient $e\colon X \epito E$ in $\D$, the corresponding limit projection $\widehat{e}\colon \widehat{X} \epito E$ lies in $\widehat{\E}$.
      Indeed, since $E$ is finitely copresentable in $\Pro{\Df}$, the morphism $\widehat e$ factorizes through $\widehat{h}$ for some $h$ in $X/\Df$, which can be assumed to be a quotient in $\D$.
      Otherwise, take the $(\E,\M)$-factorization $h=m\o q$ of $h$ and replace $h$ by~$q$.

      Thus, we obtain an initial subdiagram
      $P'_X\colon \mathscr{I}\to \Pro{\Df}$ of
      $P_X\colon X/\Df\to\Pro{\Df}$ by restricting $P_X$ to the full
      subcategory of finite quotients $h\colon X\epito A$ in $X/\Df$
      through which $e$ factorizes, i.e.~where $e=e_h\o h$ for some
      $e_h\colon A\to E$.  Note that $e_h\in \E$ because $e,h\in \E$.
      The quotients $e_h$ ($h\in \mathscr{I}$) form a cofiltered
      diagram in $(\Pro{\Df})^\to$ with limit cone $(\widehat h,\id_E)$:
      \[  
        \xymatrix{
        \widehat X \ar[d]_{\widehat h} \ar[r]^{\widehat e} & E \ar[d]^{\id_e} \\
        A \ar@{->>}[r]_{e_h} & E 
        }
      \]
Thus, $\widehat e\in \hatE$.

%
%
%
% express $\widehat{X}$ as a cofiltered limit of
%      $\widehat{a}_i\colon \widehat{X} \to A_i$, $i \in I$, with $A_i$ finite
%      and $a_i\colon X \epito A_i$ in $\E$, see \Cref{rem:epi}. Since $E$ is finite, there exist a
%      factorization of $\widehat{e}$ through some $\widehat{a}_{i_0}$, $i_0 \in
%      I$. Without loss of generality, we can assume that $\widehat{e}$
%      factorizes through each $a_i$ (since the full subcategory $I_0$ of $I$ on
%      all objects with a morphism into $I_0$ is cofiltered, and $X$ is a
%      cofiltered limit of the corresponding subdiagram). Thus, we have
%      commutative triangles as follows:
%      \begin{equation}
%        \xymatrix{
%          \widehat{X} \ar[d]_{\widehat{a}_{i}} \ar[r]^{\widehat{e}} & E \\
%          A_i \ar[ru]_{f_i}
%        }
%      \end{equation}
%      When we apply $V$ and precompose with $\eta_X$, we conclude that 
%      \[
%        e = f_i \o a_i \in \E
%      \]
%      Thus, $f_i \in \E$ for all $i$. And all $f_i$ form an obvious cofiltered
%      diagram in $\Df^\to$. Since $X$ is a cofiltered limit of $A_i$, $i \in
%      I$, we see that $e$ is a cofiltered limit of $f_i$, $i \in I$. Thus $e \in
%      \widehat{\E}$. 

\item\label{itm:ex:profinite-2} For every cofiltered diagram
  $B\colon I \to \Df$ with connecting morphisms in $\E$, the limit
  projections in $\Pro\Df$ lie in $\widehat{\E}$. Indeed, let
  $b_i\colon X \to B_i$ ($i \in I)$ denote the limit cone. Given
  $j \in I$, we are to show $b_j \in \widehat{\E}$. Form the diagram
  in $\Df^{\to}$ whose objects are all connecting morphisms of $B$
  with codomain $B_j$ and whose morphisms from $h\colon B_i \epito B_j$
  to $h'\colon B_{i'} \epito B_j$ are all connecting maps
  $k\colon B_i \epito B_{i'}$ of $B$.  This is a cofiltered diagram in
  $\Pro{\Df}$ with limit $b_j$ and the following limit cone:
  \[
    \xymatrix{
      X \ar[d]_{b_i} \ar[r]^{b_j} & B_j \ar[d]^{\id} \\
      B_i \ar@{->>}[r]_{h} & B_j
    }
  \]
  Since each $h$ lies in $\E_\f$, this proves $b_j \in \widehat{\E}$. 

\item\label{itm:ex:profinite-3} For every cone $p_i\colon P \to B_i$
  of the diagram $B$ in~\ref{itm:ex:profinite-2} with
  $p_i \in \widehat{\E}$ for all $i \in I$, the unique factorization
  $p\colon P \to X$ through the limit of $B$ lies in
  $\widehat{\E}$. Indeed, $p$ is the limit of $p_i$, $i \in I$, with
  the following limit cone:
  \[
    \xymatrix{
      P \ar[d]_{\id} \ar[r]^{p} & X \ar[d]^{b_i} \\
      P \ar[r]_{p_i} & B_i 
    }
  \]
\end{enumerate}
\end{remark}

\begin{prop}\label{prop:factorization-of-proD}
  The pair $(\widehat{\E}, \widehat{\M})$ is a proper factorization
  system of $\Pro\Df$.
\end{prop}
\begin{proof}
  \begin{enumerate}
  \item\label{P:hatfs:1} All morphisms of
    $\hatE$ are epic. This follows from the dual
    of~\cite[Prop.~1.62]{Adamek1994}; however, we give a direct
    proof. Given $e\colon X\to Y$ in $\hatE$, we have a limit cone of
    a cofiltered diagram $D$ in $(\Pro{\Df})^\to$ as follows:
    \[ 
      \vcenter{\xymatrix{
          X \ar[r]^e \ar[d]_{a_i} & Y \ar[d]^{b_i} \\
          A_i \ar@{->>}[r]_{e_i} & B_i 
        } }
      \qquad\qquad (i\in I)
    \]
    where $e_i\in \E_\f$ for each $i\in I$. Let $p,q\colon Y\to Z$ be
    two morphisms with $p\o e = q\o e$; we need to show $p=q$. Without
    loss of generality we can assume that the object $Z$ is finite
    because $\Df$ is limit-dense in $\Pro{\Df}$. Since $(b_i)$ is a
    cofiltered limit cone in $\Pro{\Df}$, there exists $i\in I$ such
    that $p$ and $q$ factorize through $b_i$, i.e.~there exists
    morphisms $p',q'$ with $p'\o b_i=p$ and $q'\o b_i = q$. The limit
    projection $a_i$ of the cofiltered limit $X=\lim A_i$ in
    $\Pro{\Df}$ merges $p'\o e_i$ and $q'\o e_i$ (since $e$ merges $p$
    and $q$). Since $Z$ is finite, there exists a connecting morphism
    $(a_{ji},b_{ji})$ of $D$ such that $p'\o e_i$ and $q'\o e_i$ are
    merged by $a_{ji}$.
    
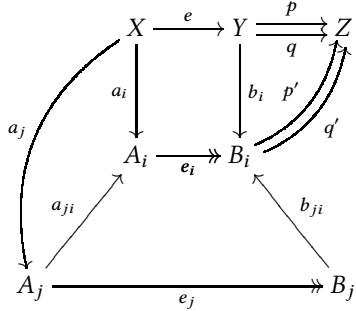
\[
      \xymatrix@R+1em{
        & X \ar@/_2em/[ddl]_{a_j} \ar[r]^e \ar[d]_{a_i} & Y \ar[d]^{b_i} \ar@<0.5ex>[r]^p \ar@<-0.5ex>[r]_q & Z \\
        & A_i  \ar@{->>}[r]_{e_i} \ar@{->>}[r]_{e_i} & B_i  \ar@/_1em/@<0.5ex>[ur]^{p'} \ar@/_1em/@<-0.5ex>[ur]_{q'} &  \\
        A_j \ar[ur]^{a_{ji}} \ar@{->>}[rrr]_{e_j} & & & B_j
        \ar[ul]_{b_{ji}} }
    \]
    Therefore
    \[
      p'\o b_{ji}\o e_j = p'\o e_i\o a_{ji} = q'\o e_i\o a_{ji} =
      q'\o b_{ji}\o e_j.
    \]
    Since $e_j$ is an epimorphism in $\Df$, this implies
    $p'\o b_{ji} = q'\o b_{ji}$. Thus
    \[p=p'\o b_i = p'\o b_{ji}\o b_j = q' \o b_{ji}\o b_j = q'\o b_i
      = q. \]
    
  \item All morphisms of $\hatM$ are monic.  Indeed, given
    $m\colon X\to Y$ in $\hatM$, we have a limit cone of a cofiltered
    diagram $D$ in $(\Pro{\Df})^\to$ as follows:
    \[ 
      \vcenter
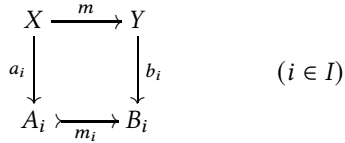
{\xymatrix{
          X \ar[r]^m \ar[d]_{a_i} & Y \ar[d]^{b_i} \\
          A_i \ar@{ >->}[r]_{m_i} & B_i 
        } }
      \qquad\qquad (i\in I)
    \]
    where $m_i\in \M_\f$ for each $i\in I$. Suppose that
    $f,g\colon Z\to X$ with $m\o f = m\o g$ are given. Express $Z$ as
    a cofiltered limit $z_j\colon Z\epito Z_j$ ($j\in J$) of finite
    objects with epimorphic limit projections $z_j$. For each
    $i\in I$, since $A_i$ is finitely copresentable, we obtain a
    factorization of $a_i\o f$ and $a_i\o g$ through some $z_{j_i}$,
    say $f_i\o z_{j_i}=f$ and $g_i \o z_{j_i}=g$.
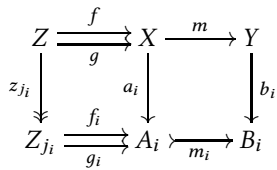
    \[ 
      \xymatrix{
        & Z \ar@<0.5ex>[r]^f \ar@<-0.5ex>[r]_g \ar@{->>}[d]_{z_{j_i}}   & X \ar[r]^m \ar[d]_{a_i} & Y \ar[d]^{b_i} \\
        & Z_{j_i} \ar@<0.5ex>[r]^{f_i} \ar@<-0.5ex>[r]_{g_i} &  A_i
        \ar@{ >->}[r]_{m_i} & B_i 
      }
    \]
    From $m\o f = m\o g$ it follows that
    $m_i\o f_i\o z_{j_i} = m_i\o g_i\o z_{j_i}$ for each $i$. This
    implies $m_i\o f_i = m_i\o g_i$ because $z_{j_i}$ is epic, and
    thus $f_i = g_i$ because $m_i$ is monic in $\Df$. Therefore,
    $a_i\o f = a_i\o g$ for each $i$, thus $f=g$ because the limit
    projections $a_i$ are collectively monic.

  \item Every morphism $g\colon X \to Y$ of $\Pro{\Df}$ has an
    $(\hatE,\hatM)$-factorization.  Indeed, $(\Pro\Df)^\to$ is the
    pro-completion of $\Df^\to$; see~\cite[Cor.~1.54]{Adamek1994} for
    the dual statement. Thus, there exists a cofiltered diagram
    $R\colon I\to \Df^\to$ with limit $g$. Let the following morphisms
      \[
        \vcenter
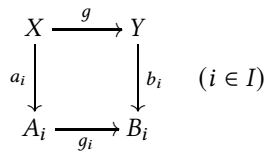
{\xymatrix{
          X \ar[d]_{a_i} \ar[r]^{g} & Y \ar[d]^{b_i}\\
          A_i \ar[r]_{g_i} & B_i 
        }}
        \quad (i \in I)
      \] 
form the limit cone. Factorize $g_i$ into an $\E$-morphism $e_i \colon A_i \epito C_i$ and followed by an $\M$-morphism $m_i\colon C_i \monoto B_i$. Since $\Df$ is closed under
      subobjects, we have $e_i \in \E_\f$ and $m_i \in \M_\f$. Diagonal
      fill-in yields a diagram $\overline{R}\colon I \to \Df$ with objects
      $C_i$, $i \in I$, and connecting morphisms derived from those of $R$. Let
      $Z \in \Pro\Df$ be a limit of $\overline{R}$ with the limit cone
      \[
        c_i\colon Z \to C_i \quad (i\in I).
      \]
      Then there are unique morphisms $e = \lim e_i \in \widehat{\E}$, and $m =
      \lim m_i \in \widehat{\M}$ such that the following diagrams commute for all $i\in I$:
      \[
        \xymatrix@R-1em{
          X \ar[dd]_{a_i} \ar[rr]^{g}\ar[rd]_{e} & & Y \ar[dd]^{b_i} \\
                                                 & Z \ar[ru]_{m} \ar[d]^{c_i} \\
          A_i \ar@{->>}[r]_{e_i} & C_i \ar@{ >->}[r]_{m_i} & B_i
        }
      \]
      
    \item\label{itm:prop:factorization-proD-2} We verify the diagonal
      fill-in property. Let a commutative square
      \[
        \xymatrix{
          X \ar[r]^e \ar[d]_{u} & Y \ar[d]^{v} \\
          P \ar[r]_{m} & Q 
        }
      \]
      with $e \in \hatE$ and $m \in \hatM$ be given. 
      \begin{enumerate}
      \item\label{itm:prop:factorization-proD-2a} Assume first that
        $m \in \M_\f$. Express $e$ as a cofiltered limit of objects
        $e_i \in \E_\f$ with the following limit cone:
        \[
          \xymatrix{
            X \ar[r]^{e} \ar[d]_{a_i} & Y \ar[d]^{b_i}\\
            A_i \ar[r]_{e_i} & B_i
          }
        \]
        Since $P$ is finite and $X = \lim A_i$ is a cofiltered limit,
        $u$ factorizes through some $a_i$. Analogously for $v$ and
        some $b_i$; the index $i$ can be chosen to be the same since
        the diagram is cofiltered. Thus we have morphisms $u'$, $v'$
        such that in the following diagram the left-hand triangle and
        the right-hand one commute:
        \[
          \xymatrix{
            X
            \ar[r]^-{e}
            \ar `l[d] `[dd]_{u} [dd]
            \ar[d]_{a_i}
            &
            Y
            \ar `r[d] `[dd]^{v} [dd]
            \ar[d]^{b_i}
            \\
            A_i \ar[d]_{u'}
            \ar@{->>}[r]^-{e_i}
            &
            B_i \ar[d]^{v'}
            \\
            P \ar[r]_{m}
            &
            Q 
          }
        \]
        Without loss of generality, we can assume that the lower part
        also commutes.  Indeed, $Q$ is finite and the limit map $a_i$
        merges the lower part:
        \[
          (m\o u')\o a_i = m\o u = v\o e = v'\o b_i \o e = (v'\o e_i)
          \o a_i.
        \]
        Since our diagram is cofiltered, some connecting morphism from
        $A_j$ to $A_i$ also merges the lower part. Hence, by choosing
        $j$ instead of $i$ we could get the lower part commutative.
        
        Since $e_i \in \E$ and $m \in \M$, we use diagonal fill-in to
        get a morphism $d\colon B_i \to P$ with $d\o e_i=u'$ and
        $m\o d=v'$. Then $d\o b_i\colon Y \to P$ is the desired diagonal in
        the original square.

      \item Now suppose that $m\in \hatM$ is arbitrary, i.e.~a
        cofiltered limit a diagram $D$ whose objects are morphisms
        $m_t$ of $\M_\f$ with a limit cone as follows:
        \[
          \vcenter{\xymatrix{
              P \ar[r]^{m} \ar[d]_{p_t} & Q \ar[d]^{q_t} \\
              P_t \ar@{ >->}[r]_{m_t} & Q_t
            }}
          \quad
          (t \in T)
        \]
        For each $t$ we have, due
        to item~\ref{itm:prop:factorization-proD-2a} above, a diagonal fill-in
        \[
          \xymatrix{
            X \ar[r]^{e} \ar[d]_{u} & Y \ar[d]^{v}
            \ar[ldd]|*+{\labelstyle d_t} \\
            P \ar[d]_{p_t} & Q \ar[d]^{q_t}\\
            P_t \ar@{ >->}[r]_{m_t} & Q_t
          }
        \]
        Given a connecting morphism $(p, q)\colon m_t \to m_s$ 
        ($t, s \in T$) of the diagram $D$, the following triangle 
        \[
          \xymatrix@C-1em{
            & Y \ar[ld]_{d_t} \ar[rd]^{d_s} \\
            P_t\ar[rr]_{p} && P_s
          }
        \]
        commutes, that is, all $d_t$ form a cone of the diagram
        $D_0\cdot D$, where $D_0\colon \Df^\to \to \Df$ is the domain
        functor, with limit $p_t\colon P \to P_t$ ($t \in T$). Indeed,
        $e$ is epic by item~\ref{P:hatfs:1}, and from the fact that
        $p_s = p \o p_t$ we obtain
        \[
          (p \o d_t) \o e = p \o p_t \o u = p_s \o u = d_s \o e.
        \]
        Thus, there exists a unique $d\colon Y \to P$ with $d_t = p_t \o d$
        for all $t \in T$. This is the desired diagonal: $u = d \o e$ follows
        from $(p_t )_{t \in T}$ being collectively monic, since 
        \[
          p_t\o u = d_t \o e = p_t \o d \o e.\tag*{\qedhere}
        \]
This implies $v=d\o m$ because $v\o e = m\o u = m\o d\o e$ and $e$ is epic.
      \end{enumerate}
    \end{enumerate}
\end{proof}

\begin{prop}\label{prop:factorsation-system-subcategory}
  Let $\D$ be a full subcategory of $\SigStr$ closed under products and
  subobjects. Then in $\Pro\Df \subseteq \Stone\D$ we have
  \begin{align*}
    \widehat{\E} & = \text{surjective morphisms, and} \\
    \widehat{\M} & = \text{relation-reflecting injective morphisms},
  \end{align*}
  cf.~\Cref{remark:factorization}\ref{R:fs:2}.
\end{prop}
\begin{proof}
  \begin{enumerate}
  \item Let $e \colon X \to Y$ be a surjective morphism of
    $\Pro\Df$. We shall prove that $e\in \hatE$ by expressing it as a
    cofiltered limit of a diagram of quotients in $\Df^\to$. In
    $\Stone\D$ we have the factorization system $(\E_0, \M_0)$ where
    $\E_0$ = surjective homomorphisms, and $\M_0$ = injective
    relation-reflecting homomorphisms. This follows from
    \Cref{remark:factorization} and the fact that $\D$, being
    closed under subobjects in $\SigStr$, inherits the factorization
    system $\SigStr$.
      
 The category $\D$ is closed under products and subobjects, so it is closed under all limits. Since $\Pro\Df$ is the closure of $\Df$ under cofiltered limits in
    $\Stone(\D)$ by \Cref{coro:pro-Df}, also
    $(\Pro{\Df})^\to = \Pro{(\Df^\to)}$ is the closure of $\Df^\to$
    under cofiltered limits in $(\Stone{\D})^\to$.  Thus for $e$ there
    exists a cofiltered diagram $D$ in $\Df^\to$ of morphisms
    $h_i\colon A_i \to B_i$ ($i \in I$) of $\Df$ with a limit cone in
    $(\Stone{\D})^\to$ as follows:
    \[
      \xymatrix{
        X \ar[r]^{e} \ar[d]_{a_i} & Y \ar[d]^{b_i} \\
        A_i \ar[r]_{h_i} & B_i
      }
    \]
    Using the factorization system $(\E_0, \M_0)$ we factorize
    \[
      a_i = m_i \o \overline{a}_i
      \quad\text{and}\quad
      b_i = n_i \o  \overline{b}_i
      \quad
      \text{for $i\in I$}, 
    \]
    and use the diagonal fill-in to define morphisms $\overline{h}_i$
    as follows:
    \[
      \xymatrix{
        X \ar@{->>}[r]^{e}\ar@{->>}[d]_{\overline{a}_i} &
        Y\ar@{->>}[d]^{\overline{b}_i} \\
        \overline{A}_i \ar@{-->}[r]^{\overline{h}_i}\ar@{ >->}[d]_{m_i}
        &
        \overline{B}_i\ar@{ >->}[d]^{n_i}
        \\
        A_i \ar[r]_{h_i} & B_i
      }
    \]
    We obtain a diagram $\overline{D}$ with objects
    $\overline{h}_i\colon \overline{A}_i \to \overline{B}_i$
    ($i \in I$) in $(\Stone\D)^\to$. Connecting morphisms are derived
    from those of $D$: given $(p, q)\colon h_i \to h_j$ in $D$
    \[
      \xymatrix{
        A_i \ar[r]^{h_i}\ar[d]_{p} & B_i \ar[d]^{q} \\
        A_j \ar[r]_{h_j} & B_j
      }
    \]
    the diagonal fill-in property yields morphisms $\ol p$ and $\ol q$
    as follows:
    \[
      \vcenter{
        \xymatrix@R-1em{
          X \ar[dd]_{\overline{a}_j} \ar@{->>}[r]^{\overline{a}_i} 
          & \overline{A}_i \ar[d]^{m_i}\ar@{-->}[ldd]_{\overline{p}} \\
          & A_i \ar[d]^{p} \\
          \overline{A}_j \ar@{ >->}[r]_{m_j} & A_j 
        }
      }
      \qquad
      \vcenter{
        \xymatrix@R-1em{
          Y \ar[dd]_{\overline{a}_j} \ar@{->>}[r]^{\overline{b}_i} 
          & \overline{B}_i \ar[d]^{n_i}\ar@{-->}[ldd]_{\overline{q}} \\
          & B_i \ar[d]^{q} \\
          \overline{B}_j \ar@{ >->}[r]_{n_j} & B_j 
        }
      }
    \]
    It is easy to see that $(\overline{p}, \overline{q})$ is a
    morphism from $\overline{h}_i$ to $\overline{h}_j$ in
    $(\Stone\D)^\to$. This yields a cofiltered diagram
    $\overline{D}$. Since
    $\overline{h}_i \o \overline {a}_i = \overline{b}_i \o e$ is
    surjective, it follows that $\overline{h}_i$ is also
    surjective. We claim that the morphisms
    \[
      (\overline{a}_i, \overline{b}_i)\colon e \to \overline{h}_i
      \quad (i\in I)
    \]
    form a limit cone of $\overline{D}$. To see this, note first that
    since the morphisms $(a_i, b_i) \colon e \to h_i$, $i \in I$, form
    a cone of $D$ and all $m_i$ and $n_i$ are monic, the morphisms
    $(\overline{a}_i, \overline{b}_i)$, $i \in I$, form a cone of
    $\overline{D}$. Now let another cone be given with domain
    $r\colon U \to V$ as follows:
    \[
      \xymatrix@R-1em{
        U \ar[r]^{r}\ar[d]_{u_{i}} & V \ar[d]^{v_{i}} \\
        \overline{A}_i \ar@{ >->}[d]_{m_i}
        \ar[r]^{\overline{h}_i} & \overline{B}_i\ar@{ >->}[d]^{n_i} \\
        A_i \ar[r]_{h_i} & B_i
      }
    \]
    Then we get a cone of $D$ for all $i\in I$ by the morphisms
    $(m_iu_i, n_iv_i)\colon r \to h_i$. The unique factorization
    $(u,v)$ through the limit cone of $D$:
    \[
      \xymatrix@R-1em{
        {\,U} \ar[r]^{r} \ar[d]_{u} \ar@<.3ex>`l[]`[dd]_{m_iu_i} & V\; \ar[d]^{v} \ar`r[]`[dd]^{n_iv_i}[dd] \\
        X \ar[r]^{e} \ar[d]_{a_i} & Y \ar[d]^{b_i} \\
        A_i \ar[r]_{h_i} & B_i
      }
    \]
    is a factorization of $(u_i, v_i)$ through the cone
    $(\overline{a}_i, \overline{b}_i)$. Indeed, in the following
    diagram
    \[
      \xymatrix@R-1em{
        {\,U} \ar[r]^{r} \ar[d]_{u} \ar@<.3ex>`l[]`[dd]_{u_i} & V\;
        \ar[d]^{v}
        \\
        X \ar[r]^{e} \ar[d]_{\overline{a}_i} & Y \ar[d]^{\overline{b}_i} \\
        \overline{A}_i \ar[d]_{m_i} \ar[r]^{\overline{h}_i}
        &
        \overline{B}_i
        \ar[d]^{n_i}
        \ar@{<-} `r[u]`[uu]_{v_i}[uu]
        \\
        A_i \ar[r]_{h_i} & B_i
      }
    \]
    the desired equality $v_i = \overline{b}_i v$ follows since $n_i$
    is monic; analogously for $u_i = \overline{a}_i u$. The uniqueness
    of the factorization $(u, v)$ also follows from the last diagram:
    if the upper left-hand and right-hand parts commute, then $(u, v)$
    is a factorization of the cone $(m_iu_i, n_iv_i)$ through the
    limit cone of $D$. Thus, it is unique.

  \item Conversely, every cofiltered limit of quotients in $\Df^\to$
    is surjective in $\Pro\Df$. Indeed, cofiltered limits in $\Pro\Df$
    are formed in $\Stone(\SigStr)$ by \Cref{coro:pro-Df}, and the
    forgetful functor into $\Stone$ thus preserves them. Hence the
    same is true about the forgetful functor from $(\Pro\Df)^\to$ to
    $\Stone^\to$. Thus, the claim follows from \Cref{lem:epi}.
    
  \item We show that every morphism of $\Pro{\Df}$ which is monic and
    reflects relations is an element of $\hatM$.
    \begin{enumerate}
    \item\label{itm:monocolim} We first prove a property of filtered
      colimits in $\BA^\to$. Let $D$ be a filtered diagram with
      objects $h_i\colon A_i\to B_i$ ($i\in I$) in $\BA^\to$. Let
      $h_i=m_i\o e_i$ be the factorization of $h_i$ into an
      epimorphism $e_i\colon A_i\epito \ol{B}_i$ followed by a
      monomorphism $m_i\colon \ol{B}_i\monoto B_i$ in $\BA$. Using
      diagonal fill-in we get a filtered diagram $\ol D$ with objects
      $e_i$ ($i\in I$) and with connecting morphisms
      $(u,\ol v)\colon e_i\to e_j$ derived from the connecting
      morphisms $(u,v)\colon h_i\to h_j$ of $D$ using diagonal fill-in:
      
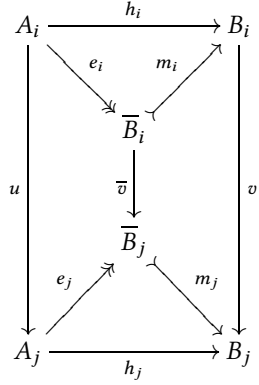
\[
        \xymatrix{
          A_i \ar[rr]^{h_i} \ar[ddd]_{u} \ar@{->>}[dr]^{e_i} & & B_i \ar[ddd]^v \\
          & \ol{B}_i  \ar[d]_{\ol v} \ar@{>->}[ur]^{m_i} & \\
          & \ol{B}_j \ar@{>->}[dr]^{m_j} & \\
          A_j \ar@{->>}[ur]^{e_j} \ar[rr]_{h_j} & &  B_j
        }
      \] 
      Our claim is that if the colimit $h=\colim h_i$ in $\BA^\to$ is
      an epimorphism of $\BA$, then one has $h=\colim e_i$. To see
      this, suppose that a colimit cocone of $D$ is given as follows:
      
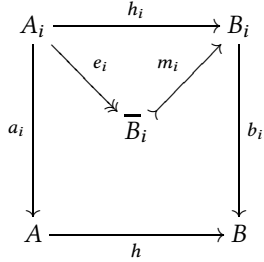
\[
        \xymatrix{
          A_i \ar@{->>}[dr]^{e_i} \ar[rr]^{h_i} \ar[dd]_{a_i} & & B_i \ar[dd]^{b_i} \\
          & \ol{B}_i \ar@{>->}[ur]^{m_i} & \\
          A \ar[rr]_h & & B
        }
      \]
      Then we prove that $\ol D$ has the colimit cocone
      $(a_i, b_i\o m_i)$, $i\in I$. Indeed, since $A=\colim A_i$ with
      colimit cocone $(a_i)$, all we need to verify is that
      $B=\colim \ol{B}_i$ with cocone $(b_i\o m_i)$. This cocone is
      collectively epic because every element $x$ of $B$ has the form
      $x=h(y)$ for some $y\in A$, using that $h$ is epic by
      hypothesis, and that the cocone $(a_i)$ is collectively epic.  The
      diagram $\ol D$ is filtered, thus, to prove that
      $B=\colim \ol{B}_i$, we only need to verify that whenever a
      pair $x_1,x_2\in \ol{B}_i$ (for some $i\in I$) is merged by
      $b_i\o m_i$, there exists a connecting morphism
      $\ol v\colon \ol{B}_i\to\ol{B}_j$ merging $x_1,x_2$. Since $m_i$
      is monic and $B=\colim B_i$, some connecting morphism
      $v\colon B_i\to B_j$ merges $m_i(x_1)$ and $m_i(x_2)$. Then
      \[ m_j\o \ol{v}(x_1) = v\o m_i(x_1) = v\o m_i(x_2) = m_j\o \ol{v}(x_2), \]
      whence $\ol{v}(x_1)=\ol{v}(x_2)$ because $m_j$ is monic.
      
    \item\label{itm:wfunc} Denote by
      $W\colon \Stone(\SigStr)\to\Set^\S$ the forgetful functor
      mapping a Stone-topological $\Sigma$-structure to its underlying
      sorted set. Moreover, letting $\Sigma_\rel \seq \Sigma$ denote
      the set of all relation symbols in $\Sigma$, we have the
      forgetful functors
      \[
        W_r\colon \Stone(\SigStr)\to \Set\quad (r\in \Sigma_\rel)
      \]
      assigning to every object $A$ the corresponding subset
      $r_A\seq A^{s_1}\times\cdots\times A^{s_n}$.  From the
      description of limits in $\SigStr$ in
      \Cref{remark:SigStr-limits}, it follows that the functors $W$
      and $W_r$ ($r\in \Sigma_\rel$) collectively preserve and reflect
      limits.  That is, given a
      diagram $D$ in $\Stone(\SigStr)$, a
      cocone of $D$ is a limit cone if and only if its image under $W$
      is a limit cone of $W\o D$ and its image under $W_r$ is a limit
      cone of $W_r\o D$ for all $r\in \Sigma_\rel$.
      
    \item We are ready to prove that if $h\colon A\to B$ in
      $\Pro{\Df}$ is a relation-reflecting monomorphism, then
      $h\in \hatM$. We have a cofiltered diagram $D$ in $\Df^\to$ with
      objects $h_i\colon A_i\to B_i$ and a limit cone
      $(a_i,b_i)\colon h_i\to h$ ($i\in I$). Let $h_i=m_i\o e_i$ be
      the image factorization in $\SigStr$.
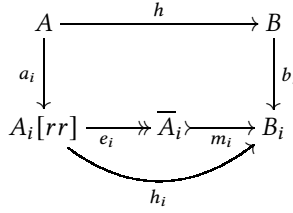
      \[
        \xymatrix{
          A \ar[rr]^h \ar[d]_{a_i} && B \ar[d]^{b_i} \\
          A_i
          \ar@{->>}[r]_{e_i}
          \ar@/_2em/[rr]_{h_i} [rr]
          &
          \ol{A}_i \ar@{>->}[r]_{m_i} & B_i
        }
      \]
      It is our goal to prove that $h=\lim_{i\in I} m_i$. More
      precisely: we have $m_i$ in $\Df^\to$ and diagonal fill-in
      yields a cofiltered diagram $\ol D$ of these objects in
      $\Df^\to$. We will prove that $(e_i\o a_i, b_i)\colon h\to m_i$
      ($i\in I$) is a limit cone. By part \ref{itm:wfunc} above it
      suffices to show that the images of that cone under $W^\to$ and
      $W_r^\to$ ($r\in \Sigma_\rel$) are limit cones.

      For $W^\to$ just dualize \ref{itm:monocolim}: from the fact that
      $Wh=\lim Wh_i$ we derive $Wh=\lim Wm_i$. We need to show that $W_r$ preserves the limit of the diagram of all $m_i$'s. Given
      $r\colon s_1,\ldots, s_n$ in $\Sigma_\rel$, we know that $r_A$
      consists of the $n$-tuples $(x_1,\ldots, x_n)$ with
      $(a_i(x_1),\ldots,a_i(x_n)) \in r_{A_i}$ for every $i\in I$ (see
      \Cref{remark:SigStr-limits}). In particular, for
      $(x_1,\ldots, x_n)\in r_A$ we have
      $(e_i\o a_i(x_1),\ldots, e_i\o a_i(x_n))\in
      r_{\ol{A}_i}$. Conversely, given $(x_1,\ldots,x_n)$ with the
      latter property, then
      $(m_i\o e_i\o a_i(x_1),\ldots, m_i\o e_i\o a_i(x_n))\in
      r_{B_i}$, i.e.~$(b_i\o h(x_1),\ldots, b_i\o h(x_n))\in r_{B_i}$
      for all $i\in I$. Since $B=\lim B_i$, this implies
      $(h(x_1),\ldots, h(x_n))\in r_B$, whence
      $(x_1,\ldots, x_n)\in r_A$ because $h$ is relation-reflecting.
    \end{enumerate}
    
  \item It remains to prove that every morphism $m\in \hatM$ is a
    relation-reflecting monomorphism. Let a cofiltered limit cone be
    given as follows:
    \[ 
      \vcenter
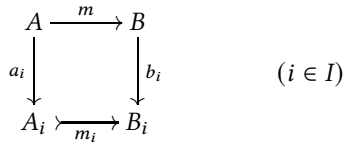
{\xymatrix{
          A \ar[r]^m \ar[d]_{a_i} & B \ar[d]^{b_i} \\
          A_i \ar@{ >->}[r]_{m_i} & B_i 
        } }
      \qquad\qquad (i\in I)
    \]
    where each $m_i$ lies in $\M_\f$, i.e.~is a relation-reflecting
    monomorphism in $\Df$. Then $m$ is monic: given $x\neq y$ in $A$,
    there exists $i\in I$ with $a_i(x)\neq a_i(y)$ because the limit
    projections $a_i$ are collectively monic. Since $m_i$ is monic,
    this implies $b_i\o m(x)\neq b_i\o m(y)$, whence $m(x)\neq m(y)$.

    Moreover, for every relation symbol $r\colon s_1,\ldots, s_n$ in
    $\Sigma$ and
    $(x_1,\ldots,x_n)\in A^{s_1}\times\cdots\times A^{s_n}$, we have
    that
    \[
      (x_1,\ldots, x_n)\in r_A
      \quad\text{iff}\quad
      (m(x_1),\ldots, m(x_n))\in r_B.
    \]
    Indeed, the \emph{only if} direction follows from the fact that the
    maps $m_i\o a_i$ preserve relations and the maps $b_i$
    collectively reflect them. For the \emph{if} direction, suppose that
    $(m(x_1),\ldots m(x_n))\in r_B$. Since for every $i\in I$ the
    morphism $b_i$ preserves relations and $m_i$ reflects them, we get
    $(a_i(x_1),\ldots, a_i(x_n))\in r_{A_i}$ for every $i$. Since the
    maps $a_i$ collectively reflect relations, this implies
    $(x_1,\ldots, x_n)\in r_A$.\qedhere
  \end{enumerate}
\end{proof}
We now introduce the crucial property of factorization systems needed for our main result. Actually it only concerns the class $\E$ of quotients and asserts it to be well-behaved with respect to cofiltered limits.

\begin{defn}
  The factorization system $(\E,\M)$ of $\D$ is called
  \emph{profinite} if $\E$ is closed in $\D^\to$ under cofiltered
  limits of finite quotients; that is, for every cofiltered diagram
  $D$ in $\D^\to$ whose objects are elements of $\E_\f$, the limit of
  $D$ in $\D^\to$ lies in $\E$.
\end{defn}

\begin{expl}\label{ex:profinite-factorization}
%  \begin{enumerate}
%    \item\label{itm:ex:profinite-factorization-1} $\Set$ has a profinite factorization system. Indeed, since the forgetful
%      functor $\Stone \to \Set$ creates limits, so does the forgetful functor
%      $\Stone^\to \to \Set^\to$. And every cofiltered limit of finite quotients
%      (i.e.~surjective maps between finite sets) in $\Stone^\to$ is a
%      surjective morphism. This follows from Stone Duality, which translates this to the property of $\BA^\to$ that
%        every filtered colimit of monomorphisms between finite boolean algebras is a
%        monomorphism. This is true and follows immediately from the fact that the forgetful
%      functor from $\BA$ to $\Set$ creates filtered colimits. 
%
%    \item\label{itm:ex:profinite-factorization-2} 
  For every full subcategory $\D \subseteq \SigStr$ closed under
  limits and subobjects, the factorization system of surjective
  morphisms and relation-reflecting injective morphisms is
  profinite. This follows from \Cref{lem:epi} and the fact that
  limits in $\D$ are formed at the level of underlying sets (see
  \Cref{remark:SigStr-limits}).
\end{expl}

\begin{prop}\label{prop:prop-of-profinite-fact-system}
If the factorization system $(\E,\M)$ of $\D$ is profinite, 
 the following holds:
  \begin{enumerate}
  \item The forgetful functor $V\colon \Pro\Df \to \D$ is faithful and
    satisfies $V(\widehat{\E}) \subseteq \E$.
  \item For every $\E$-projective object $X\in \D$, the object
    $\widehat{X}\in \hatD$ is $\widehat{\E}$-projective.
  \item Every object of $\Df$ is an $\widehat{\E}$-quotient of some
    $\hatE$-projective object in $\Pro{\Df}$. 
  \end{enumerate}
\end{prop}
\begin{proof}
  \begin{enumerate}
    \item $V(\widehat{\E}) \subseteq \E$ is clear: given $e \in \hatE$
      expressed as a cofiltered limit of finite quotients $e_i$, $i \in I$, in
      $(\Pro\Df)^\to$, then since $V$ is cofinitary, we see that
      $Ve$ is a cofiltered limit of $Ve_i = e_i$ in $\D^\to$, thus $Ve \in \E$
      by the definition of a profinite factorization system.  

      To prove that $V$ is faithful, recall that a right adjoint is
      faithful if and only if each component of its counit is epic.
      Thus, it suffices to prove that $\epsilon_D \in \hatE$ (and use
      that by \Cref{prop:factorization-of-proD} every
      $\hatE$-morphism is epic).  The triangles defining $\epsilon_D$
      in \Cref{constr:the-left-adjoint-to-V}\ref{C:laV:4} can be restricted
      to those with $a \in \hatE$.  Indeed, in the slice category
      $D/\Df$ all objects $a\colon D \to A$ in $\hatE$ form an initial
      subcategory.

% To see this, given $a\colon D \to A$ in $D/\Df$,
%      factorize it as $a = m \o \overline{a}$ in $(\widehat{\E}, \widehat{\M})$:
%      \[
%        \xymatrix@C-1.5em{
%          & D \ar[ld]_{\overline{a}} \ar[rd]^{a} \\
%        \overline{A} \ar@{ >->}[rr]_{m} & & \overline{A} 
%        }
%      \]
%      Since $A$ is finite, we have $m \in \M$, and thus $\overline{A}$ is also finite.
%      Thus, $\overline{a}\colon D \to \overline{A}$ is an object of $D/\Df$
%      with a connecting morphism $m$ into the given object. 
      Now given such a triangle with $a \in \widehat{\E}$ we know that
      $Va \in \E$. Thus all those objects $A$ form a cofiltered
      diagram with connecting morphisms in $\E$. Moreover,
      $\widehat{Va} \in \widehat{\E}$ by
      \Cref{ex:profinite}\ref{itm:ex:profinite-1}. This implies
      $\epsilon_D\in \hatE$ by
      \Cref{ex:profinite}\ref{itm:ex:profinite-3} .
      
    \item\label{P:profs:2} Let $X$ be an $\E$-projective object. To show that
      $\widehat X$ is $\hatE$-projective, suppose that a quotient
      $e\colon A\epito B$ in $\hatE$ and a morphism
      $f\colon \widehat X \to B$ are given. Since $\widehat{(\dash)}$
      is left adjoint to $V$ and $V(\hatE)\seq \E$, the morphism $f$
      has an adjoint transpose $f^*\colon X\to VB$ that factorizes through
      $VA$ via $g^*$ for some $g\colon \widehat X\to A$. Then
      $e\o g = f$, which proves that $\widehat X$ is projective.
      \[
        \xymatrix{
          & X \ar[dl]_{g^*} \ar[dr]^{f^*}& \\
          VA \ar@{->>}[rr]_{Ve} && VB 
        }
        \qquad\text{iff}\qquad
        \xymatrix{
          & \widehat X \ar[dl]_{g} \ar[dr]^{f}& \\
          A \ar@{->>}[rr]_{e} && B 
        }
      \]
      
    \item Given $A\in \Df$, by \Cref{assum:factorization-system}
      there exists an $\E$-projective object $X\in \D$ and a quotient
      $e\colon X\epito A$. The limit projection
      $\widehat e\colon \hatX \epito A$ lies in $\hatE$ by
      \Cref{ex:profinite}\ref{itm:ex:profinite-1}, and item~\ref{P:profs:2}
      above shows that $\widehat X$ is $\hatE$-projective.\qedhere
  \end{enumerate}
\end{proof}
We are ready to prove the following general form of the Reiterman Theorem: given
the factorization system $(\widehat{\E}, \widehat{\M})$ on the pro-completion of $\Df$, we have the concept of an equation in $\Pro\Df$.
We call it a profinite equation for $\D$, and prove that pseudovarieties in $\D$ are precisely the classes in $\Df$ that can be presented by profinite equations.

\begin{defn}
  A \emph{profinite equation} is an equation in $\Pro\Df$, i.e.~a morphism
  $e\colon X \epito E$ in $\widehat{\E}$ whose domain $X$ is $\hatE$-projective. It is \emph{satisfied} by a finite object
  $D$ provided that $D$ is injective w.r.t.\ $e$. 
\end{defn}

\begin{thm}[Generalized Reiterman Theorem]\label{thm:reiterman}
  Given a profinite factorization system on $\D$, a class of finite objects is a
  pseudovariety iff it can be presented by profinite equations. 
\end{thm}

\begin{proof}
  Every class $\V \subseteq \Df$ presented by profinite equations is a
  pseudovariety: this is proved precisely
  as~\ref{itm:prop:pseudovariety=pseudoequations-1} in
  \Cref{prop:pseudovariety=pseudoequations}. 

  Conversely, every pseudovariety can be presented by profinite equations.
  Indeed, following the same proposition, it suffices to construct, for every
  pseudoequation $e_i\colon X \epito E_i$ ($i \in I$), a profinite equation
  satisfied by the same finite objects. 

  For every $i \in I$, we have the corresponding limit projection 
  \[
    \widehat{e}_i \colon \widehat{X} \epito E_i
    \quad
    \text{with $e_i = V\widehat{e}_i \o \eta_X$}.
  \]
  Let $R$ be the diagram in $\Df$ of objects $E_i$. The connecting morphism $k\colon E_i\to E_j$ are given by the factorization
  \[
    \xymatrix@C-1em{
      {} & X \ar@{->>}[rd]^{e_j} \ar@{->>}[ld]_{e_i} \\
    E_i \ar@{->>}[rr]_{k} & & E_j 
    }
  \]
  iff $e_j \leq e_i$. Since the pseudoequation is closed under finite
  joins, $R$ is cofiltered. Form the limit of $R$ in $\Pro{\Df}$ with the
  limit cone
  \[
    p_i \colon E \epito E_i \quad (i \in I). 
  \]
  The morphisms $\widehat{e}_i$ above form a cone of $R$: given $e_j = k \o
  e_i$, then $V\widehat{e}_j \o \eta_X = V(\widehat{k\o e_i})\o\eta_X = k\o
  V\widehat{e}_i \o \eta_X$ implies $\widehat{e}_j = k \o \widehat{e}_i$. Here we apply the
  universal property of $\eta_X$: the morphism $\widehat{e}_j$ is uniquely determined by $V\widehat{e}_j\o \eta_X$. Thus we have a unique morphism $e\colon
  \widehat{X} \epito E$ making the following triangles commutative:
  \[
    \vcenter{
    \xymatrix{
      \widehat{X} \ar@{->>}[r]^{e} \ar@{->>}[rd]_{\widehat{e}_i} & E\ar@{->>}[d]^{p_i} \\
                  & E_i
              }}
              \qquad
              (i\in I)
  \]
  The connecting morphisms of $R$ lie in $\E$ (since $k\o e_i \in \E$ implies
  $k\in \E$). Thus
 each $\widehat{e}_i$ lies in $\hatE$ since $e_i\in\E$, see \Cref{ex:profinite}\ref{itm:ex:profinite-2}.
 Therefore, $e \in \hatE$ by \Cref{ex:profinite}\ref{itm:ex:profinite-3}. Since $\widehat X$ is $\hatE$-projective by \Cref{prop:prop-of-profinite-fact-system}, we have thus obtained a profinite equation $e\colon \widehat{X} \epito E$.

  We are going to prove that a finite object $A$ satisfies the pseudoequation $(e_i)_{i\in I}$ iff it satisfies the profinite equation $e$.
  \begin{enumerate}
    \item Let $A$ satisfy the pseudoequation $(e_i)$. For every morphism
      $f\colon \widehat{X} \to A$ we present a factorization through $e$. The
      morphism $Vf \o \eta_X$ factorizes through some $e_j$, $j\in I$:
      \[
        \xymatrix{
          X \ar[r]^{\eta_{X}} \ar@{->>}[d]_{e_j} & V \widehat{X}\ar[d]^{Vf} \\
          E_{j} \ar[r]_{g} & A
        }
      \]
      Since $e_j = V\widehat{e}_j\o\eta_X$, we get $V(g\o\widehat{e}_j)\o \eta_X
      = Vf \o \eta_X$. By the universal property of $\eta_X$ this implies 
      \[
        g\o\widehat{e}_j = f.
      \]
      The desired factorization is $g\o p_j$:
      \[
        \xymatrix{
          \widehat{X}
          \ar[d]_{f}
          \ar@{->>}[r]^-{e}
          \ar@{->>}[rd]|-*+{\labelstyle \widehat{e}_j}
          &
          E \ar@{->>}[d]^{p_j}
          \\
          A & E_j
          \ar[l]^{g}
        }
      \]

    \item Let $A$ satisfy the profinite equation $e$. For every morphism
      $h\colon X \to A$ we find a factorization through some $e_j$. The morphism
      $\widehat{h}\colon\widehat{X} \to A$ factorizes through $e$:
      \[
        \widehat{h} = u \o e
        \quad
        \text{with $u\colon E \to A$}.
      \]
  The codomain of $u$ is finite, thus, $u$ factorizes through one of the
      limit projection of $E$, i.e.
      \[
        u = v \o p_j
        \quad\text{with $j\in I$ and $v \colon E_j \to A$}.
      \]
      This gives the following commutative diagram:
      \begin{equation}\label{diag:u}
        \vcenter{
        \xymatrix{
          \widehat{X} \ar[d]_{\widehat{h}} \ar[r]^{e} & E \ar[d]^{p_j} 
          \ar@{-->}[ld]_{u} \\
          A & E_j \ar[l]^{v}
        }}
      \end{equation}
      That $v$ is the desired factorization of $h$ is now shown using
      the following diagram:
      \[
        \xymatrix{
          X \ar `d[dr] [rrrd]_-h \ar[r]^{\eta_X}
          &
          V\widehat{X}
          \ar[rrd]_(.4){V\widehat{h}}\ar[r]^{Ve}
          &
          VE \ar[r]^-{Vp_j}
          &
          E_j \ar[d]^{v}
          \ar@{<<-} `u[l] `[lll]_-{e_j} [lll]
          \\
          &&&
          A
        }                                         
      \]
      Indeed, the upper part commutes since since
      \[
        e_j = V\widehat{e}_j \o \eta_X = Vp_j \o Ve \o \eta_X,
      \]
      the lower left-hand part commutes since $h = V\widehat{h}\o
      \eta_X$, and for the remaining lower right-hand part apply $V$
      to~\eqref{diag:u} and use that $Vv = v$ since $v$ lies in $\Df$.\qedhere
  \end{enumerate}
\end{proof}

%%% Local Variables:
%%% mode: latex
%%% TeX-master: "reiterman"
%%% End:

\section{Profinite Monad}\label{sec:profinite-monad}

In the present section we establish the main result of our paper: a generalization of Reiterman's theorem from algebras over a signature to algebras for a given monad $\MT$ in a category $\D$
(\Cref{thm:reiterman-for-monads}). To this end, we introduce and investigate the \emph{profinite monad} $\hatT$ associated to the monad $\MT$. It provides an abstract perspective on the formation of spaces of profinite words or profinite terms and serves as key technical tool for our categorical approach to profinite algebras.  
\begin{assumption}
  Throughout this section, $\D$ is a category satisfying
  \Cref{assum:factorization-system}, and $\MT=(T,\mu,\eta)$ is a
  monad on $\D$ preserving quotients, i.e.~$T(\E)\seq \E$.
\end{assumption}
We denote by $\D^\MT$ the category of $\MT$-algebras and
$\MT$-homomorphisms, and by $\Df^\MT$ the full subcategory of all
\emph{finite algebras}, i.e.~$\MT$-algebras whose underlying object
lies in $\Df$.
\begin{remark}\label{rem:dtsatassumptions}
  The category $\D^{\MT}$ satisfies
  \Cref{assum:factorization-system}. More precisely:
  \begin{enumerate}
    \item Since $\MT$ preserves quotients, the factorization system of $\D$
      lifts to $\D^\MT$: every homomorphism in $\D^\MT$ factorizes as a
      homomorphism in $\E$ followed by one in $\M$. When speaking about
      \emph{quotient algebras} and \emph{subalgebras} of $\MT$-algebras, we refer to this
      lifted factorization system $(\E^\MT, \M^\MT)$. 
      
    \item Since $\D$ is complete, so is $\D^\MT$ with limits created by the forgetful functor into $\D$.

    \item The category $\Df^\MT$ is closed under finite products and
      subalgebras, since $\Df$ is closed under finite products and
      subobjects.

    \item For every $\E$-projective object $X$, the free algebra $(TX, \mu_X)$ is
        $\E^\MT$-projective. Indeed, given $\MT$-homomorphisms $e\colon (A,\alpha)\epito (B,\beta)$ and $h\colon (TX,\mu_X)\to (B,\beta)$ with $e\in \E$,  then $h \o \eta_{ X} \colon X \to B$ 
  factorizes through $e$ in $\D$, i.e.~$h \o \eta_{ X} = e\o k_0$ for some $k_0$. Then the
  $\MT$-homomorphism $k\colon (T{X}, \mu_X) \to (A, \alpha)$
  extending $k_0$ fulfils $e \o k \o \eta_{{X}} =
  h \o \eta_{{X}}$, hence, $e \o k = h$ by the universal property
  of $\eta_{{X}}$.
 \[
    \xymatrix{
      {X} \ar[r]^-{\eta_{{X}}} \ar@{-->}[d]_{k_0} 
      & (T X, \mu_X) \ar@{-->}[ld]_{k} \ar[d]^{h} \\
    (A, \alpha) \ar@{->>}[r]_-{e} & (B, \beta)
    }
  \] 
  It follows that every finite algebra is a quotient of an
  $\E^\MT$-projective $\MT$-algebra.
  \end{enumerate}
\end{remark}

\begin{notation}
  The forgetful functor of $\Df^\MT$ into $\Pro\Df$ is denoted by
  \[
    K \colon \Df^\MT \to \Pro\Df
  \]
  For example, if $\D = \SigStr$, then $K$ assigns to every finite $\MT$-algebra
  its underlying $\Sigma$-structure, equipped with the discrete topology.
\end{notation}

\begin{remark}\label{rem:kanext}
  For any functor $K\colon \A\to \Cat$, the right Kan extension
  \[R=\Ran_K K\colon \Cat\to\Cat\] can be naturally equipped
  with the structure of a monad. Its unit and multiplication are given
  by
  \[
    \hateta=(\id_K)^\dagger\colon \Id\to R
    \qquad\text{and}\qquad
    \hatmu = (\epsilon\o R\epsilon)^\dagger\colon RR\to R,
  \]
  where $\epsilon\colon RK\to K$ denotes the universal natural
  transformation and $(\dash)^\dagger$ is defined as in
  \Cref{rem:kanext0}. The monad $(R,\hateta,\hatmu)$ is called the
  \emph{codensity monad} of $K$, see e.g.~Linton~\cite{Linton1969}.
\end{remark}

\begin{defn}\label{def:profinitemonad}
  The \emph{profinite monad} 
\[\hatT = (\hatt,\hatmu,\hateta)\] of the monad $\MT$ is the
  codensity monad of the forgetful functor $K\colon \Df^\MT \to \Pro\Df$. 
\end{defn}

\begin{construction}\label{cons:profinite-monad}
  Since $\Pro{\Df}$ is complete and $\Df^\MT$ is small, the limit
  formula for right Kan extensions (see \Cref{rem:kanext0}) yields
  the following concrete description of the profinite monad:
  \begin{enumerate}
  \item\label{itm:cons:profinite-monad-a} To define the action of
    $\hatt$ on an object $X$, form the coslice category $X/K$ of all
    morphisms $a\colon X \to K(A, \alpha)$ with
    $(A, \alpha)\in \Df^\MT$. The projection functor
    $Q_X\colon X/K \to \Pro\Df$, mapping $a$ to $A$, has a limit
      \[
        \hatt X = \lim Q_X.
      \]
      The limit cone is denoted as follows:
      \[
        \frac{X \xto{a} K(A, \alpha)}{\hatt X \xto{\alpha_a^+} A}
      \]
      For every finite $\MT$-algebra $(A, \alpha)$, we write 
      \[
        \alpha^+\colon \hatt A \to A
      \]
      instead of $\alpha^+_{\id_A}$. 

    \item\label{itm:cons:profinite-monad-b}
      The action of $\hatt$ on morphisms $f\colon Y \to X$ is given by the
      following commutative triangles 
      \[
        \vcenter{
        \xymatrix@C-1em{
          \hatt Y \ar[rr]^{\hatt f}\ar[rd]_{\alpha_{af}^+} & & \hatt X
          \ar[ld]^{\alpha_a^+} \\
                  & A 
              }}
        \quad
        \text{for all $a\colon X \to K(A, \alpha)$.} 
      \]
      
    \item\label{itm:cons:profinite-monad-c} The unit $\hateta\colon \Id \to
      \hatt$ is given by the following commutative triangles
      \[
        \vcenter{
          \xymatrix@C-1em{
            X \ar[rr]^{\hateta_X} \ar[rd]_{a} && \hatt X \ar[ld]^{\alpha_a^+} \\
              & A 
          }
        }
        \quad\text{for all $a\colon X \to K(A, \alpha)$.}
      \]
      and the multiplication by the following commutative squares 
      \[
        \vcenter{
        \xymatrix{
          \hatt\hatt X \ar[r]^{\hatmu_X} \ar[d]_{\hatt\alpha_a^+} & \hatt X
          \ar[d]^{\alpha_a^+} \\
        \hatt A \ar[r]_{\alpha^+} & A 
    }}
      \quad\text{for all $a\colon X \to K(A, \alpha)$.}
      \] 
  \end{enumerate}
\end{construction}

\begin{remark}\label{rem:codensitymonad}
  A concept related to the profinite monad was studied by
  Boja\'nczyk~\cite{Bojanczyk2015} who associates to every monad $\MT$
  on $\Set$ a monad $\overline \MT$ on $\Set$ (rather than on
  $\Pro\Set_\mathsf{f}=\Stone$ as in our setting).  Specifically,
  $\overline \MT$ is the monad induced by the composite right adjoint
  $\Stone^{\hatT} \to \Stone \xto{V} \Set$. Its construction also
  appears in the work of Kennison and Gildenhuys~\cite{Kennison1971}
  who investigated codensity monads for $\Set$-valued functors and
  their connection with profinite algebras.
\end{remark}

\begin{remark}\label{rem:hattprops}
\begin{enumerate}
\item\label{itm:cons:profinite-monad-d} Every finite $\MT$-algebra
  $(A, \alpha)$ yields a finite $\hatT$-algebra $(A,
  \alpha^+)$. Indeed, the unit law and the associative law for
  $\alpha^+$ follow from
  \Cref{cons:profinite-monad}\ref{itm:cons:profinite-monad-c} with
  $X = A$ and $a = \id_A$.
  
\item\label{itm:cons:profinite-monad-e} The monad $\hatt$ is cofinitary. To see
  this, let $x_i\colon X\to X_i$ ($i\in I$) be a cofiltered limit cone
  in $\Pro \Df$. For each object of $X/K$ given by an algebra
  $(A,\alpha)$ and morphism $a\colon X\to A$, due to $A\in \Df$ there
  exists $i\in I$ and a morphism $b\colon X_i\to A$ with $a=b\o
  x_i$. From the definition of $\hatt$ on morphisms we get
  \[
    \alpha_a^+ = (\,\hatt X \xrightarrow{~\hatt x_i~} \hatt X_i
    \xrightarrow{~\alpha_b^+~} B\,).
  \]
  To prove that $\hatt x_i\colon \hatt X \to \hatt X_i$ ($i\in I$)
  forms a limit cone, suppose that any cone $c_i\colon C\to \hatt X_i$
  ($i\in I$) is given.  It is easy to verify that then the cone of
  $Q_X$ (see
  \Cref{cons:profinite-monad}\ref{itm:cons:profinite-monad-a})
  assigning to the above $a$ the morphism $\alpha_b^+\o c_i$ is
  well-defined, i.e.~independent of the choice of $i$ and $b$ and
  compatible with $Q_X$.  The unique morphism $c\colon C\to \hatt X$
  factorizing that cone fulfils $c_i = \hatt x_i\o c$ because this
  equation holds when postcomposed with the members of the limit cone
  of $Q_{X_i}$.  This proves the claim.

\item \label{itm:cons:profinite-monad-f} The free $\hatT$-algebra
  $(\hatt X, \hatmu_X)$ on an object $X$ of $\Pro{\Df}$ is a
  cofiltered limit of finite $\hatT$-algebras.  In fact, for the
  squares in
  \Cref{cons:profinite-monad}\ref{itm:cons:profinite-monad-c}
  defining $\hatmu_X$ we have the limit cone $(\alpha_a^+)$ in
  $\Pro{\Df}$, and since all $\alpha_a^+$ are homomorphisms of
  $\hatT$-algebras and the forgetful functor from $(\Pro{\Df})^\hatT$
  to $\Pro{\Df}$ reflects limits, it follows that
  $(\hatt X, \hatmu_X)$ is a limit of the algebras $(A,\alpha^+)$.

\item For ``free'' objects of $\Pro \Df$, i.e.~those of the form
  $\hat X$ for $X\in \D$ (cf.~\Cref{lem:the-left-adjoint-to-V}),
  the definition of $\hatt \hatX$ can be stated in a more convenient
  form: $\hatt \hatX$ is the cofiltered limit of all finite quotient
  algebras of the free $\MT$-algebra $(TX,\mu_X)$. More precisely, let
  $(TX,\mu_X)\mathord{\epidownarrow}\Df^\MT$ denote the full
  subcategory of the slice category $(TX,\mu_X) / \Df^\MT$ on all
  finite quotient algebras of $(TX,\mu_X)$, and consider the diagram
\[ D_X\colon (TX,\mu_X)\mathord{\epidownarrow}\Df^\MT \to \Pro{\Df} \]
that maps $e\colon (TX,\mu_X) \epito (A,\alpha)$ to $A$.
Then we have the following
\end{enumerate}
\end{remark} 
\begin{lem}\label{lem:hattX-is-limit-of-finite-quotients}
  For every object $X$ of $\D$, one has $\hatt \hatX = \lim D_X$. 
\end{lem}
\begin{proof}
  The diagram $D_X$ is the composite
  \[
    (TX,\mu_X)\mathord{\epidownarrow}\Df^\MT \monoto
    (TX,\mu_X)/\Df^\MT \cong \widehat{X}/K \xto{Q_{\hat X}}
    \Pro{\Df},
  \]
  where the isomorphism $(TX,\mu_X)/\Df^\MT \cong \widehat{X}/K$ maps
  $e\colon (TX,\mu_X) \to (A,\alpha)$ to
  $\widehat{e\o \eta_X}\colon \widehat X\to A$. Since every
  $\MT$-homomorphism has an $(\E^\MT,\M^\MT)$-factorization,
  $(TX,\mu_X)\mathord{\epidownarrow}\Df^\MT$ is an initial subcategory
  of $(TX,\mu_X)/\Df^\MT$. Thus, $\hatt X = \lim Q_{\hat X} = \lim D$.
\end{proof}
\begin{notation}\label{not:abuselimitprojections}
  The above proof gives, for every object $X\in \D$, the limit cone
  $\alpha_{\widehat{e\o \eta_X}}^+\colon \hatt \hatX\epito A$ with
  $e\colon (TX,\mu_X)\epito (A,\alpha)$ ranging over
  $(TX,\mu_X)\mathord{\epidownarrow}\Df^\MT$. In the following, we
  abuse notation and simply write $\alpha_e^+$ for
  $\alpha_{\widehat{e\o \eta_X}}^+$.
\end{notation}

\begin{expl}\label{ex:profinitewords}
  Given the monad $TX = X^*$ of monoids on $\D = \Set$, the profinite monad is the
  monad of monoids in $\Stone$ 
  \[
    \hatt X = \text{free monoid in $\Stone$ on the space $X$}.
  \]
  For a finite set $X$, the elements of $\hatt X$ are called the \emph{profinite
  words} over $X$. A profinite word is a compatible choice of a congruence class of
  $X^*/{\sim}$ for every congruence $\sim$ of finite rank. Compatibility
  means that given another congruence $\approx$ containing $\sim$, 
  the class chosen for $\approx$ contains the above class as a subset.
\end{expl}

\begin{lem}\label{coro:hatT-preserves-quotients}
  The monad $\hatt$ preserves quotients,
  i.e.~$\hatt(\widehat{\E}) \subseteq \widehat{\E}$.
\end{lem}

\begin{proof}
Suppose that $e\colon X\to Y$ is a morphism im $\hatE$. This means that it can be expressed as a cofiltered limit in $\hatD^\to$ of morphisms $e_i\in \E_\f$ ($i\in I$):
\[  
\xymatrix{
X \ar@{->>}[r]^e \ar[d]_{p_i} & Y \ar[d]^{q_i} \\
X_i \ar@{->>}[r]_{e_i}& Y_i
}
\]
Since $\hatt$ is cofinitary by \Cref{rem:hattprops}\ref{itm:cons:profinite-monad-e}, it follows that $\hatt e$ is the limit of $\hatt e_i = Te_i$ ($i\in I$) in $\hatD^\to$. Since $T$ preserves $\E$, we have $Te_i\in \E$ for all $i\in I$, which proves that $\hatt e\in \hatE$. 
%Indeed, the cofiltered diagram described in the above lemma has connecting
%morphisms in~$\E^\MT$. Thus the limit maps $\alpha^+_a$ all lie in
%$\widehat{\E}$. Given $e \colon Y \to X$ in $\widehat{\E}$, the definition of
%$\hatt e$ can thus be reduced to the commutativity of the triangles
%\[
%  \xymatrix@C-1.5em{
%    \hatt Y \ar[rr]^{\hatt e} \ar[rd]_{\alpha^+_{a \o e}} & & \hatt X \ar[ld]^{a} \\
%            & A
%  }
%\]
%with $\alpha_a^+ \in \E$ and $\alpha^+_{a \o e} \in \E$. Indeed, since $(A,
%\alpha)$ is a quotient of $(TX, \mu_X)$, i.e.~$\alpha_a^+\in \E$,
%we see that $(A, \alpha)$ is a quotient of $(TY, \mu_Y)$ via $\tl a$ and $e$ lies in $\E$. Thus, $\hatt e$ is the factorization map of a cone of members of
%$\widehat{\E}$, which preserves $\hatt e \in \widehat{\E}$. 
\end{proof}

It follows that the factorization system
$(\widehat{\E}, \widehat{\M})$ of $\Pro\Df$ lifts to the category
$(\Pro{\Df})^{\hatT}$. Moreover, this category with the choice
\[
  (\Pro\Df)^{\hatT}_\f = \text{all $\hatT$-algebras $(A, \alpha)$ with $A \in \Df$}
\]
satisfies all the requirements of
\Cref{assum:factorization-system}; this is analogous to the
corresponding observations for $\D^\MT$ in
\Cref{rem:dtsatassumptions}. Note that we are ultimately interested
in finite $\MT$-algebras, not finite $\hatT$-algebras. However, there
is no clash: we shall prove in
\Cref{prop:finite-T-algbras-are-hatT-algebras} that they coincide.

\begin{notation}
  Recall from \Cref{cons:profinite-monad} the definition of $\hatt X$ as a cofiltered limit $\alpha_a^+\colon
  \hatt X \to A$ of $Q_X\colon X/K\to \Pro{\Df}$. Since the functor $V\colon \Pro\Df \to \D$  (see \Cref{not:v}) preserves that limit,
  and since all morphisms 
  \[
    TVX \xto{TVa} TA \xto{\alpha} A
  \]
   form a cone of $V \o Q_X$, there is a unique  morphism $\varphi_X$ such the
   squares below commute for every finite $\MT$-algebra $(A,\alpha)$:
   \begin{equation}\label{diag:phi}
     \vcenter{
       \xymatrix@C+1pc{
         TVX \ar@{-->}[r]^-{\varphi_X} \ar[d]_{TVa} & V\hatt X
         \ar[d]^{V\alpha_a^+} \\
         TA \ar[r]_-{\alpha} & A
       }
     }
   \end{equation}
 \end{notation}

\begin{expl}
  For the monoid monad $TX = X^*$ on $\Set$, the map
  \[
    \varphi_X\colon (VX)^* \to V\hatt X
  \]
  is the embedding of finite words into profinite words. More
  precisely, by representing elements of $\hatt X$ as compatible
  choices of congruences classes (see \Cref{ex:profinitewords}),
  $\phi_X$ maps $w\in X^*$ to the compatible family of all congruence
  classes $[w]_\sim$ of $w$, where $\sim$ ranges over all congruences
  on $X^*$ of finite rank.
\end{expl}

We now prove that the morphisms $\varphi_X$ are the components of a
monad morphism from $\MT$ to $\hatT$ in the sense of
Street~\cite{Street1972}.
\begin{lem}\label{lem:profinite-monad-morphism}
  The morphisms $\varphi_X$ form a natural transformation 
  \[
    \varphi \colon TV \to V\hatt 
  \]
  such that the following diagrams commute:
  \[
      \xymatrix@C-1.5em{
        & V \ar[ld]_{\eta V} \ar[rd]^{V\hateta} \\
        TV \ar[rr]_{\varphi} & & V\hatt 
      }
      \qquad\qquad
      \xymatrix{
        TTV \ar[r]^{T\varphi} \ar[d]_{\mu V} & TV\hatt \ar[r]^{\varphi\hatt} &
        V\hatt \hatt \ar[d]^{V\hatmu} \\
        TV\ar[rr]_{\varphi} && V\hatt 
      }
  \]
\end{lem}
\begin{proof}
  \begin{enumerate}
  \item We first prove that $\varphi$ is natural. Given a morphism
    $f\colon X \to Y$ in $\Pro\Df$, consider an arbitrary object
    $a\colon Y \to K(A, \alpha)$ of $Q_Y$ (see
    \Cref{cons:profinite-monad}\ref{itm:cons:profinite-monad-a})
    and recall that by the definition of $\hatt$ on the morphism $f$
    we have
    \[
      \alpha_a^+\o \hatt f = \alpha_{a\o f}^+.
    \]
    Consider the following diagram:
    \[
      \xymatrix{
        TVX \ar[rrr]^-{\phi_X} \ar[dd]_{TVf} \ar[dr]^(.6){TV(a\o f)}
        & & &
        V\hatt X \ar[dl]_(.6){V\alpha_{a\o f}^+} \ar[dd]^{V\hatt f}
        \\
        &
        TA \ar[r]^\alpha
        &
        A
        \\
        TVY
        \ar[ur]_(.6){TVa}
        \ar[rrr]_{\phi_Y}
        & & &
        V\hatt Y \ar[ul]^(.6){V\alpha_a^+}
      }
    \] 
    Since all inner parts commute by definition, and the morphisms
    $V\alpha_a^+$ form a collectively monic cone using that $V$ is
    cofinitary, we see that the outside commutes,
    i.e.~$\phi$ is natural.
    
  \item To prove $V\hateta_X = \varphi_X \o \eta_{VX}$, use the
    collectively monic cone $V\alpha_a^+\colon V\hatt X \to VA$, where
    $a\colon X \to K(A, \alpha)$ ranges over $Q_X$. Using the triangle
    in \Cref{cons:profinite-monad}\ref{itm:cons:profinite-monad-c},
    we see that the following diagram
    \[
      \xymatrix{        
        VX
        \ar[r]^-{\eta_{VX}}
        \ar[d]_{Va}
        &
        TVX
        \ar[d]^{TVa}
        \ar[r]^-{\varphi_X}
        &
        V\hatt X \ar[d]^{V\alpha_a^+}
        \ar@{<-} `u[l] `[ll]_-{V\hateta_X} [ll]
        \\
        A \ar[r]_-{\eta_A}
        &
        TA \ar[r]_-{\alpha}
        &
        A
        \ar@{<-} `d[l] `[ll]^-{\id_A} [ll]
      }
    \]
    has the desired upper part commutative, since it commutes when
    post-composed by every $V\alpha_a^+$, which follows from the fact
    that the two lower squares and the outside clearly commute. 

  \item To prove
    $V\hatmu_X \o \varphi_{\hatt X} \o T{\varphi_X} = \varphi_X \o
    \mu_{VX}$, we again use the collectively monic cone
    $V\alpha_a^+$. The square in
    \Cref{cons:profinite-monad}\ref{itm:cons:profinite-monad-c}
    makes it clear that in the following diagram 
    \[
      \xymatrix@C+1pc{
        TTVX
        \ar[rd]^(.6){TTVa}
        \ar[dd]_{T\varphi_X} \ar[rrr]^-{\mu_{VX}}
        & & &
        TVX
        \ar[ld]_(.6){TVa}
        \ar[dddd]^{\varphi_X}
        \\
        &
        TTA
        \ar[r]^{\mu_A}
        \ar[d]_{T\alpha}
        &
        TA \ar[d]^{\alpha}
        \\
        TV\hatt X
        \ar[r]^-{TV\alpha_a^+}
        \ar[dd]_{\varphi_{\hatt X}}
        &
        TA\ar[r]_{\alpha}
        \ar[d]_{\varphi_A}
        &
        A
        \\
        &
        V\hatt A \ar[r]_-{V\alpha^+}
        &
        A
        \ar@{=}[u]
        \\
        V\hatt\hatt X
        \ar[ru]_(.6){V\hatt \alpha_a^+}
        \ar[rrr]_{V\hatmu_X}
        & & &
        V\hatt X
        \ar[lu]^(.6){V\alpha_a^+} 
      }
    \]
    the outside commutes, since it does when post-composed by all
    $V\alpha_a^+$.\qedhere
  \end{enumerate}
\end{proof}
\begin{prop}\label{prop:finite-T-algbras-are-hatT-algebras}
  The categories of finite $\MT$-algebras and finite $\hatT$-algebras are
  isomorphic: the functor taking  $(A, \alpha)$ to $(A, \alpha^+)$ and
  being the identity map on morphisms is an isomorphism. 
\end{prop}
\begin{proof}
  \begin{enumerate}
  \item We first prove that, given finite $\MT$-algebras $(A,
      \alpha)$ and $(B, \beta)$, a morphism $h\colon A \to B$ is a
      homomorphism for $\MT$ iff $h\colon (A, \alpha^+) \to (B, \beta^+)$ is a
      homomorphism for $\hatT$. If the latter holds, then the naturality of
      $\varphi$ yields a commutative diagram as follows
      \[
        \xymatrix{
          TA \ar[d]_{Th} \ar[r]^-{\varphi_A} & V\hatt A \ar[d]_{V\hatt h} \ar[r]^-{V\alpha^+} &
          VA\ar[d]^{Vh} \ar@{=}[r] & A \ar[d]^{h} \\
          TB \ar[r]_-{\varphi_B} & V\hatt B \ar[r]_-{V\beta^+} & VB \ar@{=}[r] & B
        }
      \]
      Thus $h$ is a homomorphism for $\MT$, since the horizontal morphisms are
      $\alpha$ and $\beta$, respectively. 

      Conversely, if $h$ is a homomorphism for $\MT$, then the diagram
      $Q_A$ of
      \Cref{cons:profinite-monad}\ref{itm:cons:profinite-monad-a}
      has the following connecting morphism
      \[
        \xymatrix@C-1em{
          & A \ar[ld]_{\id_A} \ar[rd]^{h} \\
        K(A, \alpha) \ar[rr]_{Kh} & & K(B, \beta).
        }
      \]
      This implies $h \o \alpha^+ = \beta_h^+$. The definition of
      $\hatt h$ yields $\beta^+ \o \hatt h = \beta_h^+$ (see
      \Cref{cons:profinite-monad}\ref{itm:cons:profinite-monad-a}
      again). Thus, $h$ is a homomorphism for $\hatT$:
      \[
        \xymatrix{
          \hatt \ar[rd]|*+{\labelstyle \beta_h^+} 
            \ar[r]^{\alpha^+} \ar[d]_{\hatt h} A & A \ar[d]^{h} \\
          \hatt B \ar[r]_{\beta^+} & B
        }
      \]
      Note that the \emph{only if} part implies that the object
      assignment $(A,\alpha) \mapsto (A, \alpha^+)$ is indeed
      functorial.
      \takeout{%% taken out because it duplicates text above
        To show that the functor above is well-defined, we verify for every
      $\MT$-homomorphism
      \[
        \xymatrix{
        TA \ar[d]_{Th} \ar[r]^{\alpha} & A \ar[d]^{h} \\
        TB \ar[r]_{\beta} & B
        }
      \]
      with $A, B$ finite that $h$ is also a $\hatT$-homomorphism from $(A, \alpha^+)$ to
      $(B, \beta^+)$. In the diagram $Q_A$ of \Cref{cons:profinite-monad}\ref{itm:cons:profinite-monad-a} we
      have the following connecting morphism
      \[
        \xymatrix@C-1em{
          & A \ar[ld]_{\id_A} \ar[rd]^{h} \\
        K(A, \alpha) \ar[rr]_{Kh} && K(B, \beta)
        }
      \]
      This implies
      \[
        h \o \alpha^+ = \beta_h^+.
      \]
      Recall also that the definition of $\hatt h$ yields $\beta^+ \o \hatt h =
      \beta^+_h.$ Thus, the following diagram
      \[
        \xymatrix{
          \hatt A \ar[r]^{\alpha^+} \ar[d]_{\hatt h} \ar[rd]^{\beta_h^+} & A
          \ar[d]^{h} \\
          \hatt B \ar[r]_{\beta^+} & B
        }
      \]
      commutes, as desired.}% end takeout

    \item For every finite $\hatT$-algebra $(A, \delta)$ we prove that the composite
      \begin{equation}\label{eq:defalpha}
        \alpha = TA \xto{\varphi_A} V\hatt A \xto{V\delta} VA = A 
      \end{equation}
      defines a $\MT$-algebra with $\alpha^+ = \delta$. 

      The unit law follows from that of $\delta$, $\delta \o \hateta_A
      = \id$ and from $\varphi_A \o \eta_A = V\hateta_A$ (see
      \Cref{lem:profinite-monad-morphism}):
      \[
        \xymatrix{
          A \ar[d]_{\eta_A} \ar[rd]^(.6){V\hateta_A}
          \ar `r[rrd]^-{\id_A} [rrd]
          \\
          TA
          \ar[r]_-{\varphi_A}
          &
          V\hatt A \ar[r]_-{V\delta}
          &
          A
          \ar@{<-} `d[l] `[ll]^-{\alpha} [ll]
        }
      \]
      The associative law follows from that of $\delta$,
      $\delta \o \hatmu_A = \delta \o \hatt \delta$ and from
      $\varphi_A \o \mu_A = V \hatmu_A \o \varphi_{\hatt A} \o
      T\varphi_A$ (see \Cref{lem:profinite-monad-morphism}):
      \[
        \xymatrix{
          TTA
          \ar[d]_{T\varphi_A}
          \ar`l`[dd]_{T\alpha} [dd]
          \ar[rr]^{\mu_A}
          & &
          TA \ar[d]^{\varphi_A}
          \ar`r`[dd]^{\alpha}[dd]
          \\
          TV\hatt A \ar[d]_{TV\delta} \ar[r]^{\varphi_{\hatt A}}
          &
          V\hatt\hatt A
          \ar[r]^{V\hatmu_A} \ar[d]_{V\hatt\delta}
          &
          V\hatt A \ar[d]^{V\delta}
          \\
          TA \ar[r]_{\varphi_A} 
          &
          V\hatt A \ar[r]_{V\delta}
          &
          A
          \ar@{<-}`d[l] `[ll]^-{\alpha} [ll]
        }
      \]
      To prove that
      \[
        \alpha^+ = \delta,
      \]
      recall from \Cref{lem:hattX-is-limit-of-finite-quotients} and
      \Cref{not:abuselimitprojections} that $\hatt A$ is a
      cofiltered limit of all finite quotients
      $b\colon (TA, \mu_A) \epito (B, \beta)$ in $\D^\MT$ with the
      limit cone $\beta_b^+\colon \hatt A \epito B$.  Since $A$ is
      finite, both $\alpha^+$ and $\delta$ factorize through one of
      the limit projections $\beta_b^+$, i.e.~we have commutative
      triangles as follows:
      \begin{equation}\label{diag:bb+}
        \vcenter{
        \xymatrix@R+.5pc{
          A
          &
          \hatt A
          \ar[d]|*+{\labelstyle \beta_b^+}
          \ar[r]^{\delta} \ar[l]_{\alpha^+}
          &
          A
          \\
          &
          B
          \ar[ru]_{\delta_0}
          \ar[lu]^{\alpha_0} 
        }}
      \end{equation}
      Recall from \Cref{not:abuselimitprojections} that $\beta_b^+$
      denotes $\beta_{\widehat{b \o \eta_A}}^+$, and by
      \Cref{lem:the-left-adjoint-to-V} we have $\widehat{b \o
        \eta_A} = b \o \eta_A\colon A \to B$ since this morphism lies
      in $\Df$. Combining this with the definition~\eqref{diag:phi} of
      $\phi_A$ we have a commutative square
      \begin{equation}\label{diag:beta+}
        \vcenter{
          \xymatrix{
            TVA
            \ar[r]^-{\phi_A}
            \ar[d]_{TV(b \o \eta_A)}
            &
            V \hatt A
            \ar[d]^{\beta_b^+}
            \\
            TB \ar[r]_-\beta & B
          }
        }
      \end{equation}
      Now we compute
      \begin{align*}
          \delta_0 \o \beta \o T(b\o \eta_A)
          & = \delta_0 \o \beta \o TV(b \o \eta_A)
          &
          \text{since $b \o \eta_A$ lies in $\Df$}
          \\
          &=
          \delta_0 \o V\beta_b^+ \o \varphi_A
          &
          \text{by~\eqref{diag:beta+}}
          \\
          & =
          V\delta_0 \o V\beta_b^+ \o \varphi_A
          &
          \text{since $\delta_0$ lies in $\Df$}
          \\
          &
          =
          V\delta \o \varphi_A
          &
          \text{by~\eqref{diag:bb+}.}
      \end{align*}
      Analogously, we obtain
      \begin{equation}\label{eq:al+}
        V\alpha^+ \o \varphi_A = \alpha_0 \o \beta \o T(b\o \eta_A).
      \end{equation}
      From the definition~\eqref{diag:phi} of $\varphi_A$, we also get
      \begin{equation}\label{eq:aldel}
        V\alpha^+\o \phi_A
        =
        V\alpha_{id}^+ \o \phi_A
        =
        \alpha \o TV\id_A
        =
        \alpha
        =
        V\delta \o \varphi_A,
      \end{equation}
      where we use~\eqref{eq:defalpha} in the last step.
      Therefore, we can compute
      \begin{align*}
        \delta_0 \o b
        &
        = \delta_0 \o b \o \mu_A \o T\eta_A
        &
        \text{since $\mu_A \o T\eta_A = \id$}
        \\
        &
        = \delta_0 \o \beta\o Tb \o T\eta_A
        &
        \text{since $b$ is a $\MT$-homomorphism}
        \\
        &
        = V\delta\o \phi_A
        &
        \text{shown previously}
        \\
        &
        = V\alpha^+ \o \phi_A
        &
        \text{by~\eqref{eq:aldel}}
        \\
        &
        = \alpha_0 \o \beta \o Tb \o T\eta_A
        &
        \text{by~\eqref{eq:al+}}
        \\
        &
        = \alpha_0 \o b \o \mu_A \o T\eta_A
        &
        \text{since $b$ is a $\MT$-homomorphism}
        \\
        &
        = \alpha_0 \o b.
        &
        \text{since $\mu_A \o T\eta_A = \id$.}
      \end{align*}
      Since $b$ is epic, this implies
      $\alpha_0=\delta_0$, whence $\alpha^+=\delta$.
    \item Uniqueness of $\alpha$. Let $(A, \alpha)$ be a finite $\MT$-algebra with
      $\alpha^+ =\delta$. By the definition of $\varphi_A$ this implies 
      \[
        \alpha = V \alpha^+ \o \varphi_A = V\delta \o \varphi_A, 
      \]
      so $\alpha$ is unique.\qedhere 
  \end{enumerate}
\end{proof}
From now on, we identify finite algebras for $\MT$ and for $\hatT$. 

\begin{prop}\label{prop:pro-finite-Talgebras}
  The pro-completion of the category $\Df^\MT$ of finite
  $\MT$-algebras is the full subcategory of the category of
  $\hatT$-algebras given by all cofiltered limits of finite
  $\MT$-algebras.
\end{prop}

\begin{proof}
  Let $\L$ denote the full subcategory of $(\Pro{\Df})^\hatT$ given by
  all cofiltered limits of finite $\MT$-algebras. To show that $\L$
  forms the pro-completion of $\Df^\MT$, we verify the three
  conditions of \Cref{lem:char-of-procompletion}. By definition
  $\L$ satisfies condition~\ref{itm:char-of-procompletion-3}, and
  condition~\ref{itm:char-of-procompletion-1} follows from the fact
  that since $\Pro\Df$ has cofiltered limits, so does
  $(\Pro\Df)^{\hatT}$. Thus, it only remains to prove
  condition~\ref{itm:char-of-procompletion-2}: every algebra $(A, \alpha^+)$
  with $(A, \alpha) \in \Df^\MT$ is finitely copresentable in
  $\L$. Let $b_i \colon (B, \beta) \to (B_i, \beta_i)$, $(i \in I)$,
  be a limit cone of a cofiltered diagram $D$ in $\L$. Our task is to
  prove for every morphism $f\colon (B, \beta) \to (A, \alpha^+)$ that
  \begin{enumerate}[label=(\alph*)]
    \item\label{itm:copresentable-1} a factorization through a limit projection
      exists, i.e.~$f = f' \o b_i$ for some $i \in I$ and $f'\colon (B_i, \beta_i)
      \to (A, \alpha^+)$, and

    \item\label{itm:copresentable-2} given another factorization $f = f'' \o
      b_i$ in $\L$, then $f'$ and $f''$ are merged by a connecting morphism
      $b_{ji} \colon (B_j, \beta_j) \to (B_i, \beta_i)$ of $D$ (for some
      $j  \in I$). 
  \end{enumerate}
  
  Ad~\ref{itm:copresentable-1}, since $b_i \colon B \to B_i$ is a
  limit of a cofiltered diagram in $\Pro\Df$ and $A$ is as an object
  of $\Df$ finitely copresentable in $\Pro\Df$, we have $i \in I$ and
  a factorization $f=f'\o b_i$, for some $f'\colon B_i \to A$ in
  $\Pro\Df$.  If $f'$ is a $\MT$-homomorphism, i.e.~if the following
  diagram
  \begin{equation}\label{diag:above}
    \vcenter{
      \xymatrix{
        \hatt B \ar[r]^{\beta}\ar[d]_{\hatt b_i}
        \ar `l[d] `[dd]_{\hatt f} [dd]
        &
        B \ar[d]^{b_i}
        \ar `r[d] `[dd]^f [dd]
        \\
        \hatt B_i\ar[r]^{\beta_i}\ar[d]_{\hatt f'}  & B_i \ar[d]^{f'}\\
        \hatt A \ar[r]_{\alpha^+} & A
      }}
  \end{equation}
  commutes, we are done. In general, we have to change the choice of $i$: from
  \Cref{cons:profinite-monad}\ref{itm:cons:profinite-monad-e} recall that
  $\hatt$ is cofinitary, thus $(\hatt b_i)_{i \in I}$ is a limit
  cone. The parallel pair
  \[
    f' \o \beta_i, \alpha^+ \o \hatt f' \colon \hatt B_i \to A
  \]
  has a finitely copresentable codomain (in $\Pro{\Df}$) and is merged
  by $\hatt b_i$. Indeed, the outside of the above
  diagram~\eqref{diag:above} commutes since $f = f' \o b_i$ is a
  homomorphism. Consequently, that parallel pair is also merged by
  $\hatt b_{ji}$ for some connecting morphism
  $b_{ji} \colon (B_j, \beta_j) \to (B_i, \beta_i)$ of the diagram
  $D$:
  \[
    (\alpha^+ \o \hatt f') \o \hatt b_{ji} = (f' \o \beta_i)\o \hatt b_{ji}.
  \]
  From $b_i = b_{ji} \o b_j$ we get another factorization of $f$:
  \[
    f = (f' \o b_{ji}) \o b_j
  \]
  and this tells us that the factorization morphism
  $\overline{f} = f' \o b_{ji}$ is a homomorphism as desired:
  \[
    \xymatrix{
      \hatt B \ar[r]^{\beta} \ar[d]_{\hatt b_j} & B \ar[d]^{b_j}\\
      \hatt B_j \ar[r]^{\beta_j} \ar `l[d] `[dd]_{\hatt \overline f}[dd]
      \ar[d]_{\hatt b_{ji}} & B_j \ar[d]^{b_{ji}}
      \ar`r[d] `[dd]^{\overline f} [dd]\\
      \hatt B_i \ar[r]^{\beta_i} \ar[d]_{\hatt f'} & B_i \ar[d]^{f'} \\
      \hatt A \ar[r]_{\alpha^+} & A 
    }
  \]
  
  Ad~\ref{itm:copresentable-2}, suppose that $f', f''\colon (B_i, \beta_i) \to
  (A, \alpha)$ are homomorphisms satisfying $f = f' \o b_i = f'' \o b_i$. Since
  $B = \lim B_i$ is a cofiltered limit in $\Pro\Df$ and the limit projection
  $b_i$ merges $f', f''\colon B_i \to A$, it follows that some connecting
  morphism $b_{ji} \colon (B_j, \beta_j) \to (B_i, \beta_i)$ also merges $f',
  f''$, as desired.
\end{proof}

\begin{remark}
  If $(\E, \M)$ is a profinite factorization system on $\D$, then
  $(\E^\MT, \M^\MT)$ is a profinite factorization system on
  $\D^\MT$. Indeed, since $\E$ is closed in $\D^\to$ under cofiltered
  limits of finite quotients, and since the forgetful functor from
  $(\D^\MT)^\to$ to $\D^\to$ creates limits, it follows that $\E^\MT$
  is also closed under cofiltered limits of finite quotients.
\end{remark}

%Indeed, we consider here the collection $\widetilde{\E}$ of those
%$\hatT$-homomorphisms that are, in the category of morphisms of $\Pro\Df^\MT$
%(a subcategory of $\hatT$-algebras), cofiltered limits of quotients in
%$\Df^\MT$.
%
%\begin{enumerate}
%  \item $\widetilde{\E} \subseteq \mathsf{Epi}$. Indeed, for every $e \in
%    \widetilde{\E}$ the underlying morphism of $\D$ is a member of
%    $\widehat{\E}$ since cofiltered limits of $\hatT$-algebras are formed on the
%    level of $\Pro\Df$. 
%
%  \item For the forgetful functor $V^{\MT}\colon \Pro\Df^\MT \to \D^\MT$ we
%    have $V^\MT(\widetilde{\E}) \subseteq \E^\MT$ because cofiltered limits of
%    morphisms of $\Pro\Df^\MT$ are formed in the category of
%    $(\Pro\Df)^{\hatT}$, and thus they are created by the forgetful functor
%    into $\D^\to$. 
%  
%  \item Let $\R$ be a cofiltered diagram of finite $\MT$-algebras and
%    homomorphisms in $\E^\MT$. The limit cone in $\Pro\Df^\MT$ is carried by
%    the limit cone of the underlying diagram in $\D$, hence its members lie in
%    $\E^\MT$. Analogously concerning  factorization morphisms of cones in
%    $\E^\MT$.
%\end{enumerate}

\begin{defn}\label{def:hattsatisfaction}
  A \emph{$\hatT$-equation}
  is an equation in the category of $\hatT$-algebras, i.e.~a
  $\hatT$-homomorphism $e$ in $\hatE^\hatT$ with
  $\widehat{\E}^{\hatT}$-projective domain. A finite $\MT$-algebra
  \emph{satisfies} $e$ if it is injective with respect to $e$ in $(\Pro\Df)^{\hatT}$. 
\end{defn}

\begin{thm}[Generalized Reiterman Theorem for
  Monads]\label{thm:reiterman-for-monads}
  Let $\D$ be a category with a profinite factorization system $(\E,\M)$, and suppose that $\MT$ is
  a monad preserving quotients. Then a class of finite $\MT$-algebras is a
  pseudovariety in $\Df^\MT$ iff it can be presented by $\hatT$-equations. 
\end{thm}
\begin{remark}\label{R:r-for-monads}
  We will see in the proof that the $\hatT$-equations presenting a
  given pseudovariety can be chosen to be of the form
  $e \colon (\hatt\widehat{X}, \hatmu_{\widehat{X}}) \epito (A,
  \alpha)$ where $e \in \widehat{\E}$, the object $X$ is
  $\E$-projective in $\D$, and $A$ is finite. Moreover, we can assume
  $X \in \Var$ for any class $\Var$ of objects as in
  \Cref{re:projective-objects}.
\end{remark}
\begin{proof}[Proof of \Cref{thm:reiterman-for-monads}]
  Every class of finite $\MT$-algebras presented by $\hatT$-equations
  is a pseudovariety -- this is analogous to
  \Cref{prop:pseudovariety=pseudoequations}.

  Conversely, let $\V$ be a pseudovariety in $\Df^\MT$. For every
  finite $\MT$-algebra $(A, \alpha)$ we have an $\E$-projective object
  $X$ in $\D$ and a quotient $e\colon X \epito A$ (see
  \Cref{assum:factorization-system}). Since
  $\widehat{e} \in \widehat{\E}$ by
  \Cref{ex:profinite}\ref{itm:ex:profinite-1}, we have
  $\hatt\widehat{e}\in\widehat{\E}$ by
  \Cref{coro:hatT-preserves-quotients}.  Therefore the homomorphism
  $\overline{e}\colon (\hatt \widehat{X}, \hatmu_X) \to (A, \alpha^+)$
  extending $\widehat{e}$ lies in $\widehat{\E}$: we have
  $\overline{e} = \alpha^+ \o \hatt\widehat{e}$, and $\alpha^+$ is a
  split epimorphism by the unit law $\alpha^+\o \hateta_A =
  \id_A$. Since $(\hatE, \hatM)$ is a proper factorization system and $\Pro{\Df}$ has finite coproducts, every split epimorphism lies in $\hatE$ \cite[Thm.~14.11]{AdamekEA09}, whence $\alpha^+\in \hatE$.
    Thus, we see that every finite $\MT$-algebra is a quotient, in the
  category of $\hatT$-algebras, of
  $(\hatt\widehat{X}, \hatmu_{\widehat{X}})$ for an $\E$-projective
  object $X$ of $\D$. Each such quotient lies in
  $\Pro\Df^\MT$. Indeed, the codomain, being a finite $\MT$-algebra,
  does. To see that the domain also does, combine
  \Cref{rem:hattprops}\ref{itm:cons:profinite-monad-f} and
  \Cref{prop:finite-T-algbras-are-hatT-algebras}.
  
  In \Cref{re:projective-objects} we can thus denote by $\Var$ the
  collection of all free  algebras $(\hatt\widehat{X}, \hatmu_{\widehat{X}})$
  where $X$ ranges over $\E$-projective objects of $\D$. Then
  \Cref{thm:reiterman} and \Cref{re:projective-objects}
  yield our claim that every pseudovariety in $\Df^\MT$ can be presented by
  $\hatT$-equations which are finite quotients of free algebras $(\hatt
  \widehat{X}, \hatmu_{\widehat{X}})$ where $X$ is $\E$-projective in $\D$. 
\end{proof}

%%% Local Variables:
%%% mode: latex
%%% TeX-master: "reiterman"
%%% End:

\section{Profinite Terms and Implicit Operations}\label{sec:profinite-terms}

In our presentation so far, we have worked with an abstract
categorical notion of equations given by quotients of projective
objects.  In Reiterman's original paper~\cite{Reiterman1982} on
pseudovarieties of $\Sigma$-algebras, a different concept is used:
equations between \emph{implicit operations}, or equivalently, equations
between \emph{profinite terms}.  This raises a natural question: which
categories $\D$ allow the simplification of equations in the sense of
\Cref{def:hattsatisfaction} to equations between profinite terms?
It turns out to be sufficient that $\D$ is cocomplete and has a finite
dense set $\S$ of objects that are projective w.r.t.\ strong
epimorphisms. Recall that density of $\mathcal{S}$ means that every object $D$ of $\D$ is a canonical colimit of all morphisms from objects of $\mathcal{S}$ to $D$.
More precisely, if we view $\S$ as a full subcategory of $\D$, then $D$ is the colimit of the diagram
\[
  \mathcal{S}/D \to \D \quad\text{given by}\quad \left(s \xto{f} D\right) \mapsto s 
\]
with colimit cocone given by the morphisms $f$.
\begin{assumption}\label{asm:sec5}
  Throughout this section $\D$ is a cocomplete category with a finite dense set
  $\mathcal{S}$ of objects projective w.r.t.\ strong epimorphisms.
 It follows (see \Cref{prop:char-of-cocomplete-finite-dense} below) that $\D$ has $(\mathsf{StrongEpi}, \mathsf{Mono})$-factorizations, and we work with this
  factorization system. We denote by $\Df$ the collection of all objects $D$
  such that
  \begin{equation}
    \D(s, D) \text{ is finite for every object $s \in \mathcal{S}$.}
  \end{equation}
\end{assumption} 
We will show in \Cref{prop:char-of-cocomplete-finite-dense} below
that every category $\D$ satisfying the above assumptions can be
presented as a category of algebras over an $\S$-sorted signature.
Throughout this section, let $\Sigma$ be an $\mc{S}$-sorted
algebraic signature, i.e.~a signature without relation
symbols. We denote by
\[
  \SigmaAlg
\]
the category of $\Sigma$-algebras and homomorphisms.
\begin{expl}\label{ex:category-with-dense-projective}
  \begin{enumerate}
  \item\label{itm:ex:category-with-dense-projective-1} The category $\Set^{\mc{S}}$
    satisfies~\Cref{asm:sec5}. A finite dense set in $\Set^\S$ is
    given by the objects
      \[
        \mathbf{1}_s \;(s \in \mc{S})
      \]
      where $\mathbf{1}_s$ is the $\S$-sorted set that is empty in all
      sorts except $s$, and has a single element $\ast$ in sort
      $s$. Indeed, let $A$ and $B$ be $\mc{S}$-sorted sets and let a
      cocone of the canonical diagram for $A$ be given:
      \[
        \frac{\mathbf{1}_s \xto{f} A}{\mathbf{1}_s \xto{f^*} B}
      \]
      By this we mean that we have morphisms $f^*\colon \mathbf{1}_s
      \to B$ for every $f\colon \mathbf{1}_s \to A$ (and observe that
      the cocone condition is void in this case because there are no
      connecting morphisms $\mathbf{1} \to \mathbf{1}_t$ for $s \neq t$).
      Then we are to prove that there exists a unique $\mc{S}$-sorted function
      $h\colon A \to B$ with $f^*=h \o f$ for all $f$. Uniqueness is clear:
      given $x \in A$ of sort $s$, let $f_x \colon \mathbf{1}_s \to A$ be the
      map with $f_x(*) = x$. Then $h \o f_x = f_x^*$ implies 
      \[
        h(x) = f_x^*(*).
      \]
      Conversely, if $h$ is defined by the above equation, then for every $s \in
      \mc{S}$ and $f\colon \mathbf{1}_s \to A$ we have $f^* = h \o f$ because $f
      = f_x$ for $x = f(*)$. 

      More generally, every set of objects $\mathbf{K}_s$
      ($ s\in \mc{S}$), where $\mathbf{K}_s$ is nonempty in sort $s$
      and empty in all other sorts, is dense in $\Set^{\mc{S}}$.

    \item\label{itm:ex:category-with-dense-projective-2} The category
      $\SigmaAlg$ satisfies~\Cref{asm:sec5}.  Recall that strong
      epimorphisms are precisely the homomorphisms with surjective
      components, and monomorphisms are the homomorphisms with
      injective components.  It follows easily that for the
      free-algebra functor
      $F_\Sigma\colon \Set^{\mc{S}} \to \SigmaAlg$ all algebras
      $F_\Sigma X$ are projective w.r.t.\ strong epimorphisms. We
      present a finite dense set of free algebras.

      Assume first that $\Sigma$ is a unary signature, i.e.~all operation symbols in $\Sigma$ are of the form $\sigma\colon s\to t$.
      Then the free algebras
      \[
        F_\Sigma \mathbf{1}_s \; (s \in \mc{S})
      \]
      form a dense set in $\SigmaAlg$.  Indeed, let
      $U_\Sigma\colon \SigmaAlg \to \Set^{\mc{S}}$ denote the
      forgetful functor and $\eta\colon \Id \to U_\Sigma F_\Sigma$ the
      unit of the adjunction $F_\Sigma \dashv U_\Sigma$.  Given
      $\Sigma$-algebras $A$ and $B$ and a cocone of the canonical
      diagram as follows:
      \[
        \frac{F_\Sigma \mathbf{1}_s \xto{f} A}{F_\Sigma\mathbf{1}_s \xto{f^*} B}
      \]
      We are to prove that there exists a unique homomorphism
      $h\colon A \to B$ with $f^* = h \o f$ for every $f$. We obtain a
      corresponding cocone in $\Set^{\mc{S}}$ as follows:
      \[
        \frac{\mathbf{1}_s \xto{\eta} U_\Sigma F_\Sigma \mathbf{1}_s \xto{U_\Sigma f}
          U_\Sigma A}{\mathbf{1}_s \xto{\eta} U_\Sigma F_\Sigma \mathbf{1}_s
        \xto{U_\Sigma f^*} U_\Sigma B}
      \]
      Due to~\ref{itm:ex:category-with-dense-projective-1} there
      exists a unique function $k\colon U_\Sigma A \to U_\Sigma B$
      with
      \begin{equation}\label{eq:Ueta}
        U_\Sigma f^* \o \eta = (k \o U_\Sigma f) \o \eta \qquad 
        \text{for all $f$}.
      \end{equation}
Here and in the following we drop the
      subscripts indicating components of $\eta$.
      It remains to prove that $k$ is a homomorphism from $A$ to $B$;
      then the universal property of $\eta$ implies $f^* = k \o f$.
      Thus, given $\sigma\colon s \to t$ in $\Sigma$ and $a \in A_s$
      we need to prove $k (\sigma_A(a)) = \sigma_B(k(a))$.  Consider
      the unique homomorphisms
      \begin{eqnarray*}
        f\colon  F_\Sigma \mathbf{1}_t \to A,
        && f(\ast)=\sigma_A(a),\\
        g\colon F_\Sigma \mathbf{1}_s \to A,
        && g(\ast)=a,\\
        j\colon F_\Sigma \mathbf{1}_t \to F_\Sigma \mathbf{1}_s,
        && j(\ast)=\sigma(\ast).
      \end{eqnarray*}
      Then $f=g \o j$ and thus $f^*=g^*\o j$ because the morphisms
      $(\dash)^*$ form a cocone of the canonical diagram of $A$.  It
      follows that
      \[
        k(\sigma_A(a))
        =
        k(f(*))
        =
        f^*(\ast)
        =
        g^*(j(\ast))
        =
        g^*(\sigma(\ast))
        =
        \sigma_B(g^*(\ast))
        =
        \sigma_B(k(g(\ast)))
        =
        \sigma_B(k(a)),
      \]
      where the last but one equation holds by~\eqref{eq:Ueta}. Thus, $k$ is a
      homomorphism as desired.
%	 $f_0 \colon \mathbf{1}_t \to
%      UA$ be defined by $f_0(*) = \sigma_A(a)$, 
%then the unique extension $f\colon
%      F_\Sigma \mathbf{1}_t \to A$ of $f_0$ to a homomorphism takes the
%      term $\sigma(*)$ (of sort $t$ in $F_\Sigma\mathbf{1}_s$) to $\sigma_A(a)$.
%      Thus, 
%      \[
%        k(\sigma_A(a)) = k \o U_\Sigma f (\sigma (*)) = U_\Sigma f^*(\sigma
%        (*)). 
%      \]
%      Since $f^*$ is a homomorphism, the last value is $\sigma_B(f^*(a))$, as
%      required. 

      For a general signature $\Sigma$, let $k \in \mathbb{N} \cup \{\omega\}$
      be an upper bound of the arities of operation symbols in
      $\Sigma$ and let for every set $T \subseteq \mc{S}$ the
      following $\mc{S}$-sorted set $X_T$ be given: $X_T$ is empty for
      every sort outside of $T$, and for sorts $s\in T$ the elements
      are $(X_T)_s=\{\,i \mid i < k\,\}$.  Then the set
      \[
        F_\Sigma X_T \quad (T \subseteq \mc{S})
      \]
      is dense in $\SigmaAlg$.
      The proof is analogous to the unary case.

%At the end 
%      we are given $\sigma(a_{s_1}, \ldots, a_{s_n})$ for $\sigma\colon
%      s_0\ldots s_{n-1} \to t$ in $\Sigma$ and we use $f_0\colon X_T \to U_\Sigma A$
%      where $T = \{\, s \in \mc{S} \mid A_s \neq \emptyset \,\}$ and $a_i =
%      f_0(i)$ for $i = 0, \ldots, n-1$. The corresponding homomorphism $f\colon
%      F_\Sigma X_T \to A$ takes $\sigma(i) \in F_\Sigma X_T$ to $\sigma_A(a_i)$. 

    \item\label{itm:ex:category-with-dense-projective-3} 
The category of graphs, i.e.~sets with a binary relation, and graph homomorphisms satisfies~\Cref{asm:sec5}.
Strong epimorphisms are precisely the surjective homomorphisms which are also surjective on all edges.
Thus the two graphs shown below are clearly projective w.r.t.\ strong epimorphisms.
Moreover, they form a dense set: every graph is a canonical colimit of all of its vertices and all of its edges.
\[
  \begin{tikzpicture}[scale=0.5]
    \draw[dotted] (-0.5,-0.5) rectangle (0.5,0.5);
    \draw (0,0) node [circle,fill,inner xsep=0.5mm, inner ysep=0cm]{};
  \end{tikzpicture} 
  \qquad\qquad
  \begin{tikzpicture}[scale=0.5]
    \draw[dotted] (-1.5,-0.5) rectangle (1.5,0.5);
    \path
    (-1, 0) node(x) [circle,fill,inner xsep=0.5mm, inner ysep=0cm] {}
    (1, 0) node(y)  [circle,fill,inner xsep=0.5mm, inner ysep=0cm] {};
    \draw[->] (x) -- (y);
  \end{tikzpicture}
\]

\item Every variety, and even every quasivariety of $\Sigma$-algebras
  (presented by implications) satisfies \Cref{asm:sec5}.  This will
  follow from \Cref{prop:char-of-cocomplete-finite-dense} below.
  \end{enumerate}
\end{expl}
\begin{defn}\label{def:closed_under_factorizions}
A full subcategory $\D$ of $\SigmaAlg$
is said to be \emph{closed under
  $(\mathsf{StrongEpi}, \mathsf{Mono})$-factorizations} if for
every morphism $f\colon A\to B$ of $\D$ with factorization
$f=\xymatrix@1{A \ar@{->>}[r]^e & C \ar@{ >->}[r]^m & B}$, the object $C$
lies in $\D$.
\end{defn}

\begin{prop}\label{prop:char-of-cocomplete-finite-dense} 
  For every category $\D$ the following two statements are equivalent:
  \begin{enumerate}
  \item\label{itm:prop:char-of-cocomplete-finite-dense-1} $\D$ is
    cocomplete and has a finite dense set of objects which are
    projective w.r.t. strong epimorphisms.
  \item\label{itm:prop:char-of-cocomplete-finite-dense-2} There exists
    a signature $\Sigma$ such that $\D$ is equivalent to a full
    reflective subcategory of $\SigmaAlg$ closed under
    $(\mathsf{StrongEpi}, \mathsf{Mono})$-factorizations.
  \end{enumerate}
  Moreover, $\Sigma$ can always be chosen to be a unary signature. 
\end{prop}
\begin{proof}
  \ref{itm:prop:char-of-cocomplete-finite-dense-2} $\Rightarrow$
  \ref{itm:prop:char-of-cocomplete-finite-dense-1} Suppose that
  $\D \subseteq \SigmaAlg$ is a full reflective subcategory and that $\D$ is closed
  under $(\mathsf{StrongEpi}, \mathsf{Mono})$-factorizations.  Cocompleteness of $\D$ is clear because $\SigmaAlg$ is cocomplete. Denote
  by $(-)^@ \colon \SigmaAlg \to \D$ the reflector (i.e.~the left
  adjoint to the inclusion functor $\D\subto \SigmaAlg$) and by
  $\eta_X\colon X\to X^@$ the universal maps. From
  \Cref{ex:category-with-dense-projective} we know $\SigmaAlg$ has
  a finite dense set of projective objects $A_i$, $i \in I$. We prove
  that the objects $A_i^@$, $i \in I$, form a dense set in $\D$.

  To verify the density, let $\A$ be the full subcategory of
  $\SigmaAlg$ on $\{A_i\}_{i \in I}$. For every algebra $D \in \D$ the
  canonical diagram $\A/D \to \SigmaAlg$ assigning $A_i$ to each
  $f\colon A_i \to D$ has the canonical colimit $D$. Since the left
  adjoint $(-)^@$ preserves that colimit, we have that $D = D^@$ is a
  canonical colimit of all $f^@ \colon A_i^@ \to D$ for $f$ ranging
  over $\A/D$, as required. (Indeed, observe that every morphism
  $f\colon A_i^@\to D$ in $\D$ has the form $f=f^@$ because the
  subcategory $\D$ is full and contains the domain and codomain of
  $f$.)
    
Next, we observe that every strong epimorphism $e$ of $\D$ is strongly epic
      in $\SigmaAlg$. Indeed, take the  $(\mathsf{StrongEpi}, \mathsf{Mono})$-factorization $e=m\o e'$ of $e$ in $\SigmaAlg$. Since $\D$ is closed under factorizations, we have that $e',m\in \D$. Moreover, the morphism $m$ is monic in $\D$ because it is monic in $\SigmaAlg$. Since $e$ is a strong (and thus extremal) epimorphism in $\D$, it follows that $m$ is an isomorphism. Thus $e\cong e'$ is a strong epimorphism in $\SigmaAlg$. Since $\SigmaAlg$ is complete, this is equivalent to being an extremal epimorphism.

      Since each $A_i$ is projective w.r.t.\ strong epimorphisms in
      $\SigmaAlg$, it thus follows that $A_i^@$ is projective w.r.t.\ strong
      epimorphisms $e\colon B \epito C$ in $\D$. Indeed, given a morphism $h\colon
      A_i^@ \to C$, compose it with the universal arrow $\eta\colon A_i \to
      A_i^@$. Thus, $h\o\eta$ factorizes in $\SigmaAlg$ through $e$:
      \[
        \xymatrix{
          A_i \ar[r]^{\eta} \ar@{-->}[d]_{k}
          &
          A_i^@ \ar[d]^{h}
          \ar[ld]|*+{\labelstyle\overline k}
          \\
          B \ar@{->>}[r]_{e} & C
        }
      \]
      The unique morphism $\overline{k}\colon A_i^@ \to B$ of $\D$ with $k =
      \overline{k} \o \eta$ then fulfils the desired equality $h = e \o
      \overline{k}$ since $h \o \eta = e \o \overline{k} \o \eta$. 

 \medskip\noindent   \ref{itm:prop:char-of-cocomplete-finite-dense-1}$\Rightarrow$\ref{itm:prop:char-of-cocomplete-finite-dense-2}
      Let $\mathcal{S}$ be a finite dense set of objects projective w.r.t.\ strong
      epimorphisms, and consider $\mathcal{S}$ as a full subcategory of $\D$.
      Define an $\mathcal{S}$-sorted signature of unary symbols
      \[
        \Sigma
        =
        \mathsf{Mor}(\mathcal{S}^\op)
        \setminus \{\,\id_s\mid s \in \mathcal{S} \,\}.
      \]
      Every morphism $\sigma\colon s \to t$ of $\mathcal{S}^\op$ has
      arity as indicated: the corresponding unary operation has inputs
      of sort $s$ and yields values of sort $t$. Define a functor
      \[
        E \colon \D \to \SigmaAlg
      \]
      by assigning to every object $D$ the $\mathcal{S}$-sorted set
      with sorts
      \[
        (ED)^s = \D(s, D)\quad\text{for $s \in \mathcal{S}$}
      \]
      endowed with the operations 
      \[
        \sigma_{ED} \colon \D(s, D) \to \D(s', D)
      \]
      given by precomposing with $\sigma \colon s' \to s$ in
      $\mathcal{S} \subseteq \D$. To every
      morphism $f\colon D_1 \to D_2$ of $\D$ assign the $\Sigma$-homomorphism
      $Ef$ with sorts
      \[
        (Ef)^s \colon \D(s, D_1) \to \D(s, D_2)
      \]
      given by postcomposing with $f$. To say that $\mathcal{S}$ is a
      dense set is equivalent to saying that $E$ is full and
      faithful~\cite[Prop.~1.26]{Adamek1994}. Moreover, since $\D$ is
      cocomplete, $E$ is a right
      adjoint~\cite[Prop.~1.27]{Adamek1994}. Thus, $\D$ is equivalent
      to a full reflective subcategory of $\SigmaAlg$.

      Next we show that $\D$ has the factorization system
      $(\mathsf{StrongEpi}, \mathsf{Mono})$.  Indeed, being reflective
      in $\SigmaAlg$, it is a complete category. Moreover, $\D$ is
      well-powered because the right adjoint $\D\subto \SigmaAlg$
      preserves monomorphisms and $\SigmaAlg$ is well-powered.
%because for every object $D \in \D$ each subobject of $D$
%      is fully determined by those morphisms in $\coprod_{ s \in \mathcal{S}}
%      \D(s, D)$ that factorize through the subobject. 
      Consequently, the factorization system
      exists~\cite[Cor.~14.21]{AdamekEA09}.

      To prove closure under factorizations, observe first that a
      morphism $e \colon D_1 \to D_2$ is strongly epic in $\D$ iff
      $Ee$ is strongly epic in $\SigmaAlg$. Indeed, if $e$ is strongly
      epic, then $Ee$ has surjective sorts $(Ee)^s$ because $s$ is
      projective w.r.t.\ $e$. Thus, $Ee$ is a strong epimorphism in
      $\SigmaAlg$. Conversely, if $Ee$ is strongly epic in
      $\SigmaAlg$, then for every commutative square $g\o e = m\o f$
      in $\D$ with $m$ monic, the morphism $Em$ is monic in
      $\SigmaAlg$ because $E$ is a right adjoint, and thus a diagonal
      exists.

      \sloppypar
      Now let$f\colon A\to B$ be a morphism in $\D$ and let
      $f=\xymatrix@1{A \ar@{->>}[r]^e & C \ar@{ >->}[r]^m & B}$ be its
      $(\mathsf{StrongEpi}, \mathsf{Mono})$-factorization in
      $\D$. Thus $C\in \D$ and since by the above argument $Ee$ and
      $Em$ are strong epimorphisms and monomorphisms in $\SigmaAlg$,
      respectively, $C$ is the image of $f$ w.r.t. to the
      factorization system of $\SigmaAlg$.
%
%      $e$ is strongly epic in $\D$ because given a factorization 
%      \LTC{incomplete}
%      \[
%        \xymatrix{
%          D_1\ar[r]^{e} \ar[d]_{f}  & D_2 \\
%          D \ar@{ >->}[ru]_{m}
%        }
%      \]
%      with $m$ monic, every morphism in $\D(s, D_2)$ $(s \in \mathcal{S})$,
%      factorizes through $m$ (recall that $s$ is projective w.r.t.\ $e$), hence,
%      $m$ is invertible due to the fullness of $E$. This proves that $e$ is an extremal epimorphism, and since $\D$ is complete, strong and extremal epimorphisms in $\D$ coincide. 
\end{proof}
\begin{expl}
  \begin{enumerate}
    \item If $\D =\Set$, we can take $S = \{ \mathbf{1} \}$ where $\mathbf{1}$
      is a singleton set. The one-sorted signature $\Sigma$ in the above
      proof is empty, thus, $\SigmaAlg = \Set$. 
    
    \item In the category $\mathbf{Gra}$ of graphs we can take
      $S = \{G_1, G_2\}$, see
      \Cref{ex:category-with-dense-projective}\ref{itm:ex:category-with-dense-projective-3}.
      Here $\Sigma$ is a $2$-sorted signature with two operations
      $s, t\colon G_2 \to G_1$. A graph $G = (V, E)$ is represented as
      an algebra $A$ with sorts $A_{G_1} = V$ and $A_{G_2} = E$ and $s, t$
      given by the source and target of edges, respectively. More
      precisely, $\mathsf{Gra}$ is equivalent to the full subcategory
      of all $\Sigma$-algebras $(V,E)$ where for all $e,e'\in E$ with
      $s(e)=s(e')$ and $t(e)=t(e')$, one has $e=e'$.
  \end{enumerate}
\end{expl}

\begin{assumption}
  From now on we assume that
  \begin{enumerate}
  \item The category $\D$ is a full reflective subcategory of
    $\Sigma$-algebras closed under
    $(\mathsf{StrongEpi}, \mathsf{Mono})$-factorizations; the reflecting of a $\Sigma$-algebra $A$ into $\D$ is denoted by $A^@$. 

  \item The category $\Df$ consists of all $\Sigma$-algebras in
    $\D$ of finite cardinality in all sorts.
  \end{enumerate}  
\end{assumption}
In the case where the arities of operations in $\Sigma$ are bounded,
our present choice of $\Df$ corresponds well with the previous one in
\Cref{asm:sec5}: choosing the set $\S$ as in
\Cref{ex:category-with-dense-projective}\ref{itm:ex:category-with-dense-projective-2},
a $\Sigma$-algebra $D$ has finite cardinality iff the set of all
morphisms from $s$ to $D$ (for $s \in S$) is finite.

\begin{notation}
  For the profinite monad $\hatT$ of \Cref{def:profinitemonad} we denote by 
  \[
    U\colon (\Pro\Df)^{\hatT} \to \Set^\S
  \]
  the forgetful functor that assigns to a $\hatT$-algebra $(A, \alpha)$  the underlying $\S$-sorted set of $A$.
\end{notation}

Recall from \Cref{coro:pro-Df} that $\Pro\Df$ is a full subcategory of
$\Stone(\SigmaAlg)$, the category of Stone $\Sigma$-algebras and continuous
homomorphisms, closed under limits. From \Cref{ex:profinite-factorization} and  \Cref{prop:factorsation-system-subcategory}, we get the following 

\begin{lem}\label{lem:strongepimonoprofinite}
  The factorization system $(\mathsf{StrongEpi}, \mathsf{Mono})$ on
  $\D$ is profinite and yields the factorization system on $\Pro\Df$
  given by
  \begin{align*}
    \widehat{\E} & = \text{continuous homomorphisms surjective in every
      sort, and}\\
    \widehat{\M} & = \text{continuous homomorphisms injective in every
      sort.}
  \end{align*}
\end{lem}
\begin{notation}
  Let $X$ be a finite $\S$-sorted set of variables. 
  \begin{enumerate} 
    \item Denote by
      \[
        F_\Sigma X 
      \]
      the free $\Sigma$-algebra of \emph{terms}. It is carried by the smallest
      $\S$-sorted set containing $X$ and such that for every operation symbol
      $\sigma\colon s_1, \ldots, s_n \to s$ and every $n$-tuple of terms $t_i$ of
      sorts $s_i$ we have a term
      \[
        \sigma(t_1, \ldots, t_n) \quad \text{of sort $s$}. 
      \]

    \item For the reflection $(F_\Sigma X)^@$, the free object of $\D$ on $X$, we
      put
      \[
        X^\oplus = \widehat{(F_\Sigma X)^@}.
      \]
      This is a free object of $\Pro\Df$ on $X$, see \Cref{lem:the-left-adjoint-to-V}.

    \item Let $(A, \alpha)$ be a finite $\MT$-algebra. An \emph{interpretation}
      of the given variables in $(A,\alpha)$ is an $S$-sorted function $f$ from $X$ to the
      underlying sorted set $U(A, \alpha)$. We denote by 
      \[
        f^@ \colon (F_\Sigma X)^@ \to A
      \]
      the corresponding morphism of $\D$. It extends to a unique homomorphism of
      $\hatT$-algebras (since $(A, \alpha^+)$ is a $\hatT$-algebra by
      \Cref{prop:finite-T-algbras-are-hatT-algebras}) that we denote by
      \[
        f^\oplus \colon \left(\hatt X^\oplus, \mu_{X^\oplus}\right) \to (A,
        \alpha^+).
      \]
  \end{enumerate}
\end{notation}

\begin{defn}
  A \emph{profinite term} over a finite $S$-sorted set $X$ (of variables) is an
  element of $\hatt X^\oplus$. 
\end{defn}

\begin{expl}
  Let $\D = \Set$ and $TX = X^*$ be the monoid monad. For every finite set $X
  =X^@$ we have that $\hatt X^\oplus$ is the set of profinite words
  over $X$ (see \Cref{ex:profinitewords}).
\end{expl}

\begin{defn}\label{D:eqterm}
  Let $t_1, t_2$ be profinite terms of the same sort in $\hatt X^\oplus$. A finite
  $\MT$-algebra is said to \emph{satisfy the equation} $t_1 = t_2$ provided that
  for every interpretation $f$ of $X$ we have $f^\oplus(t_1) = f^\oplus(t_2)$. 
\end{defn}
\begin{remark}
  In order to distinguish equations being pairs of profinite terms
  according to \Cref{D:eqterm} from equations being quotients
  according to \Cref{def:hattsatisfaction}, we shall sometimes call
  the latter \emph{equation morphisms}.
\end{remark}
\begin{thm}[Generalized Reiterman Theorem for Monads on
  $\Sigma$-algebras]\label{thm:reiterman-for-monads-on-algebras}
  Let $\D$ be a full reflective subcategory of $\SigmaAlg$ closed under
  $(\mathsf{StrongEpi}, \mathsf{Mono})$-factorizations, and let $\MT$ be a monad
  on $\D$ preserving strong epimorphisms. Then a collection of finite $\MT$-algebras is a pseudovariety iff it can be
  presented by equations between profinite terms.
\end{thm}

\begin{proof}
  \begin{enumerate}
  \item We first verify that all assumptions needed for applying
    \Cref{thm:reiterman-for-monads} and \Cref{R:r-for-monads}
    are satisfied. Put
    \[
      \Var \defeq \set{(F_\Sigma X)^@}{\text{$X$ a finite $\mc{S}$-sorted set} },
    \]
    the set of all free objects of $\D$ on finitely many
    generators. We know from \Cref{lem:strongepimonoprofinite} that
    the factorization system $(\mathsf{StrongEpi}, \mathsf{Mono})$ is
    profinite.
    \begin{enumerate}
    \item Every object $(F_\Sigma X)^@$ of $\Var$ is projective
      w.r.t.~strong epimorphisms. Indeed, given a strong epimorphism
      $e\colon A \epito B$ in $\D$, it is a strong epimorphism in
      $\SigmaAlg$, i.e.~$e$ has a splitting $i\colon B \monoto A$ in
      $\Set^{\mc{S}}$ with $e \o i = \id$. For every morphism
      $f\colon (F_\Sigma X)^@\to B$ of $\D$ we are to prove that $f$
      factorizes through $e$. The $\mc{S}$-sorted function $X \to A$
      which is the domain-restriction of
      $i \o f\colon (F_\Sigma X)^@\to A$ has a unique extension to a
      morphism $g\colon (F_\Sigma X)^@\to A$ of $\D$. It is easy to
      see that $e \o i = \id$ implies $e \o g = f$, as required.
        
    \item Every object $D \in \Df$ is a strong quotient
      $e\colon (F_\Sigma X)^@\epito D$ of some $(F_\Sigma X)^@$ in
      $\Var$. Indeed, let $X$ be the underlying set of $D$. Then the
      underlying function of $\id\colon X\to D$ is a split epimorphism
      in $\Set^{\mc{S}}$, hence, $\id^@\colon (F_\Sigma X)^@\epito D$
      is a strong epimorphism by \cite[Prop. 14.11]{AdamekEA09}.
    \end{enumerate}

  \item\label{itm::thm:reiterman-for-monads-on-algebras-2} By applying
    \Cref{thm:reiterman-for-monads} and \Cref{R:r-for-monads}, all we need to prove is that
    the presentation of finite $\MT$-algebras by equation morphisms 
    \[
      e \colon (\hatt X^\oplus, \hatmu_{X^\oplus}) \epito (A, \alpha),
      \quad\text{$X$ finite and $e$ strongly epic},
    \]
    is equivalent to their presentation by equations between profinite
    terms.
    
    \begin{enumerate}
    \item Let $\V$ be a collection in $\Df^\MT$ presented by equations
      $t_i = t_i'$ in $\hatt X^\oplus_i$, $i \in I$. Using
      \Cref{thm:reiterman-for-monads}, we just need proving that
      $\V$ is a pseudovariety:
      \begin{enumerate}
      \item Closure under finite products
        $\prod_{k \in K} (A_k, \alpha_k)$: Let $f$ be an
        interpretation of $X_i$ in the product.  Then we have
        $f = \< f_k \>_{k \in K}$ for interpretations $f_k$ of $X_i$
        in $(A_k, \alpha_k)$. By assumption
        $f_k^\oplus(t_i) = f_k^\oplus(t_i')$ for every $k \in
        K$. Since the forgetful functor from $\hatT$-algebras to
        $\Set^{\mc{S}}$ preserves products, we have
        $f^\oplus = \< f_k^\oplus \>_{k \in K}$, hence
        $f^\oplus (t_i) = f^\oplus (t_i')$.
            
        \item Closure under subobjects
          $m \colon (A, \alpha) \monoto (B, \beta)$: Let $f$ be an
          interpretation of $X_i$ in $(A, \alpha)$.  Then
          $g = (Um) \o f$ is an interpretation in $(B, \beta)$, thus
          $g^\oplus(t_i) = g^\oplus (t_i')$. Since $m$ is a
          homomorphism of $\hatT$-algebras, we have
          $g^\oplus = m \o f^\oplus$. Moreover, $m$ is monic in every sort,
          whence $f^\oplus(t_i) = f^\oplus(t_i')$.

        \item Closure under quotients
          $e\colon (B, \beta) \epito (A, \alpha)$: Let $f$ be an
          interpretation of $X_i$ in $A$. Since $Ue$ is a split
          epimorphism in $\Set^{\mc{S}}$, we can choose
          $m \colon UA \to UB$ with $(Ue) \o m = \id$. Then
          $g = m \o f$ is an interpretation of $X_i$ in $(B, \beta)$,
          thus, $g^\oplus(t_i) = g^\oplus(t_i')$.  Since $e$ is a
          homomorphism of $\hatT$-algebras, we have
          \[
            e\o g^\oplus = (Ue \o g)^\oplus = (Ue \o m \o f)^\oplus =
            f^\oplus.
          \]
          Using this, we obtain $f^\oplus (t_i) = f^\oplus (t_i')$. 
        \end{enumerate}
        
      \item For every equation morphism
        \[
          e \colon (\hatt X^\oplus, \hatmu_{X^\oplus}) \epito (A, \alpha)
        \]
        we consider the set of all profinite equations $t = t'$ where
        $t, t' \in \hatt X^\oplus$ have the same sort and fulfil
        $e(t) = e(t')$. We prove that given a finite algebra
        $(B, \beta)$, it satisfies $e$ iff it satisfies all of those
        equations.
        \begin{enumerate}
        \item Let $(B, \beta)$ satisfy $e$ and let $f$ be an interpretation
          of $X$ in it. Then the homomorphism $f^\oplus$ factorizes through
          $e$:
          \[
            \xymatrix{
              (\hatt X^\oplus, \hatmu_{X^\oplus}) \ar[r]^-{f^\oplus}
              \ar@{->>}[rd]_{e} & (B, \beta)\\
              & (A, \alpha) \ar[u]_{h} 
            }
          \]
          Thus, $f^\oplus(t) = f^\oplus(t')$ whenever $e(t) = e(t')$, as
          required.  
          
        \item Let $(B, \beta)$ satisfy the given equations $t=t'$. We
          prove that every homomorphism $h\colon (\hatt X^\oplus,
          \hatmu_{X^\oplus}) \to (B, \beta)$ factorizes through the
          given $e$, which lies in $(\Pro\Df)^{\hatT}$. We clearly have 
          \[
            h = f^\oplus 
          \]
          for the interpretation $f\colon X \to U(B, \beta)$ obtained by the
          domain-restriction of $Uh$. Consequently, for all $t, t' \in \hatt
          X^\oplus$ of the same sort, we know that
          \[
            e(t) = e(t')
            \quad\text{implies}\quad
            h(t) = h(t'). 
          \]
          This tells us precisely that $Uh$ factorizes in
          $\Set^{\mc{S}}$ through $Ue$:
          \[
            \xymatrix@C-3em{
              & U(\hatt X^\oplus, \hatmu_{X^\oplus}) \ar@{->>}[ld]_{Ue}
              \ar[rd]^{Uh} \\
              U(A, \alpha) \ar[rr]_{k} & & U(B, \beta)
            }
          \]
          It remains to prove that $k$ is a homomorphism of
          $\hatT$-algebras. Firstly, $k$ preserves the operations of
          $\Sigma$ and is thus a morphism $k \colon A \to B$ in
          $\D$. This follows from $Ue$ being epic in $\Set^{\mc{S}}$:
          given $\sigma\colon s_1, \ldots, s_n \to s$ in $\Sigma$ and
          elements $x_i$ of sort $s_i$ in $A$, choose $y_i$ of sort
          $s_i$ in $U(\hatt X^\oplus, \hatmu_{X^\oplus})$ with
          $Ue(y_i) = x_i$. Using that $e$ and $h$ are
          $\Sigma$-homomorphism we obtain the desired equation
          \[
            k(\sigma_A(x_i))
            =
            k(\sigma_A(Ue(y_i))
            =
            k \o Ue(\sigma(y_i))
            =
            Uh(\sigma(y_i))
            =
            \sigma_B(h(y_i))
            =
            \sigma_B(k(x_i)).
          \]
          Moreover, $\hatt e$ is epic by
          \Cref{lem:strongepimonoprofinite}. In the following
          diagram
          \[
            \xymatrix@C+1pc{
              \hatt\hatt X^\oplus
              \ar[r]^-{\hatmu_{X^\oplus}}
              \ar[d]_{\hatt e}
              \ar `l[d] `[dd]_{\hatt h} [dd]
              &
              \hatt X^\oplus
              \ar[d]^{e}
              \ar `r[d] `[dd]^{h} [dd]
              \\
              \hatt A \ar[r]^-{\alpha} \ar[d]_{\hatt k}
              &
              A \ar[d]^{k}
              \\
              \hatt B \ar[r]_-{\beta}
              &
              B
            }
          \]
          the outside and upper square commute because $h$ and $e$ are
          a homomorphism of $\hatT$-algebras, respectively, and the
          left hand and right hand parts commute because $k \o e = h$. Since
          $\hatt e$ is epic, it follows that the lower square
          also commutes.\qedhere
        \end{enumerate}
      \end{enumerate}
    \end{enumerate}
  \end{proof}

\begin{remark}
  We now show that profinite terms are just another view of the implicit
  operations that Reiterman used in his paper~\cite{Reiterman1982}. We start
  with a one-sorted signature $\Sigma$ (for notational simplicity) and then
  return to the general case. We denote by 
  \[
    W\colon \Df^\MT \to \Set
  \]
  the forgetful functor assigning to every finite algebra $(A, \alpha)$ the
  underlying set $A$.
\end{remark}

\begin{defn}
  An $n$-ary \emph{implicit operation} is a natural transformation
  $\varrho\colon W^n \to W$ for $n \in \mathbb{N}$. Thus if 
  \[
    U \colon \Df \to \Set
  \]
  denotes the forgetful functor, then $\varrho$ assigns to every finite
  $\MT$-algebra $(A, \alpha)$ an $n$-ary operation on $UA$ such that every
  homomorphism in $\Df^\MT$ preserves that operation. 
\end{defn}
For the case of finitary $\Sigma$-algebras, i.e.~finitary monads $\MT$
on $\Set$, the above concept is due to Reiterman
{\cite[Sec.~2]{Reiterman1982}}.

\begin{expl}\label{ex:implicit-operation}
  Let $\D = \Set$ and $TX = X^*$ be the monoid monad. Every element $x$ of a
  finite monoid $(A, \alpha)$ has a unique idempotent power $x^k$ for some $k>0$, denoted by $x^\omega$. Since monoid morphisms preserve idempotent powers, this yields a
  unary implicit operation $\varrho$ with components $\varrho_{(A, a)}\colon x
  \mapsto x^\omega$. 
\end{expl}

\begin{notation}\label{nota:n-ary-implicit-operation}
  Consider $n$ as the set $\{0, \ldots, n-1\}$. Every profinite term $t \in \hatt n^\oplus$ defines an $n$-ary implicit operation $\varrho_t$ as
  follows: Given a finite $\MT$-algebra $(A, \alpha)$ and an $n$-tuple $f\colon n \to
  UA$, we get the homomorphism $f^\oplus \colon (\hatt n^\oplus,
  \hatmu_{n^\oplus}) \to (A, \alpha)$, and $\varrho_t$ assigns to $f$ the value
  \[
    \varrho_t(f) = f^\oplus (t).
  \]
  The naturality of $\varrho_t$ is easy to verify.
\end{notation}

\begin{lem}\label{lem:implicit-operation=profinite-terms}
  Implicit $n$-ary operations correspond bijectively to profinite terms in
  $\hatt n^\oplus$ via $t \mapsto \varrho_t$. 
\end{lem}

\begin{proof}
  Recall from \Cref{coro:pro-Df} that $\Pro\Df$ is a full subcategory of
  $\Stone(\SigmaAlg)$ closed under limits. The forgetful functor of the
  latter preserves limits, hence, so does the forgetful functor $\ol U\colon \Pro\Df \to
  \Set$. Recall further from \Cref{cons:profinite-monad} that 
  \[
    \hatt n^\oplus = \lim Q_{n^\oplus}
  \]
  where $Q_{n^\oplus}\colon n^\oplus/K\to \Pro{\Df}$ is the diagram of all morphisms
  \[
    a \colon n ^\oplus \to K(A, \alpha) = A
    \quad\text{of $\Pro\Df$}.
  \]
  Thus, profinite terms $t \in \hatt n^\oplus$ are elements of the limit of 
  \[
    \overline{U} \o Q_{n^\oplus} \colon n^\oplus/K \to \Set
  \]
  By the well-known description of limits in $\Set$, to give $t$ means
  to give a compatible collection of elements of $UA$, i.e.~for every
  $n^\oplus \xto{a} K(A, \alpha)$ one gives $t_a \in UA$ such that for
  every morphism of $n^\oplus/{K}$:
  \[
    \xymatrix@C-2em{
      & n^\oplus \ar[rd]^{b} \ar[ld]_{a} \\
    K(A, \alpha) \ar[rr]_{Kh} & & K(B, \beta)
    }
  \]
  we have $Uh(t_a) = t_b$. 
  
  Now observe that an object of $n^\oplus/K$ is precisely a finite
  $\MT$-algebra $(A, \alpha)$ together with an $n$-tuple $a_0$ of
  elements of $UA$. Thus, the given collection $a \mapsto t_a$ is
  precisely an $n$-ary operation on $UA$ for every finite algebra
  $(A, \alpha)$. Moreover, the compatibility means precisely that
  every homomorphism $h\colon (A, \alpha) \to (B, \beta)$ of finite
  $\MT$-algebras preserves that operation. Thus, $\hatt n^\oplus$
  consists of precisely the $n$-ary implicit operations. Finally, it
  is easy to see that the resulting operation is $\varrho_t$ of
  \Cref{nota:n-ary-implicit-operation} for every
  $t \in \hatt n^\oplus$.
\end{proof}

\begin{remark}
  \begin{enumerate}
    \item For $\mc{S}$-sorted signatures this is completely analogous. Let $W^s \colon \Df^\MT \to \Set$ assign to every finite
      $\MT$-algebra $(A, \alpha)$ the component of sort $s$ of the
      underlying $\mc{S}$-sorted set $UA$.  An
      \emph{implicit operation} of arity 
      \[
        \varrho\colon s_1,\ldots, s_n \to s
      \]
      is a natural transformation
      \[
        \varrho\colon W^{s_1} \times \dots \times W^{s_n} \to W^s
      \]
      Thus $\varrho$ assigns to
      every finite $\MT$-algebra $(A, \alpha)$ an operation
      \[
        \varrho_{(A, a)}\colon {UA}^{s_1} \times \dots UA^{s_n}\to UA^s
      \]
      that all homomorphisms in $\Df^\MT$ preserve. 
    
    \item\label{itm:ssortedimplicitop} Recall that we identify every
      natural number $n$ with the set $\{0, \ldots, n-1\}$. For every
      arity $s_1, \ldots, s_n \to s$ we choose a finite
      $\mc{S}$-sorted set $X$ such that for every sort $t$ we have
      \[
         X^t = \set{i \in \{1, \ldots, n\}}{t=s_i}.
      \]
      Then for every finite $\MT$-algebra $(A, \alpha)$, to give an
      $n$-tuple $a_i \in A_{s_i}$ is the same as to give $\mc{S}$-sorted
      function $f\colon X \to UA$.

    \item \Cref{nota:n-ary-implicit-operation} has the following
      generalization: given a profinite term $t \in \hatt X^\oplus$
      over $X$ of sort $s$, we define an implicit operation
      $\varrho_t \colon s_1, \ldots, s_n\to s$ by its components at
      all finite $\MT$-algebras $(A, \alpha)$ as follows:
      \[
        \varrho_t(f) = f^\oplus (t) \quad\text{for all $f\colon X \to UA$}.
      \]
      This yields a bijection between $\hatt X^\oplus$ and implicit
      operations of arity $s_1, \ldots, s_n \to s$ for $X$
      in~\ref{itm:ssortedimplicitop}. The proof is completely
      analogous to that of
      \Cref{lem:implicit-operation=profinite-terms}.
  \end{enumerate}
\end{remark}

\begin{defn}
  Let $\varrho$ and $\varrho'$ be implicit operations of the same
  arity. A finite algebra $(A,\alpha)$ \emph{satisfies the equation}
  $\varrho = \varrho'$ if their components $\varrho_{(A,\alpha)}$ and
  $\varrho_{(A,\alpha)}'$ coincide.
\end{defn}

The above formula $\varrho_t(f) = f^\oplus(t)$ shows that given
profinite terms $t, t' \in \hatt X^\oplus$ of the same sort, a finite
algebra satisfies the profinite equation $t = t'$ if and only if it
satisfies the implicit equation $\varrho_t =
\varrho_{t'}$. Consequently:

\begin{cor}
  Under the hypotheses of
  \Cref{thm:reiterman-for-monads-on-algebras}, a collection of
  finite $\MT$-algebras is a pseudovariety iff it can be presented by
  equations between implicit operations.
\end{cor}

%%% Local Variables:
%%% mode: latex
%%% TeX-master: "reiterman"
%%% End:

\section{Profinite Inequations}\label{sec:profinite-ineq}

Whereas for varieties $\D$ of algebras the equation morphisms in the
Reiterman \Cref{thm:reiterman-for-monads} can be substituted by
equations $t = t'$ between profinite terms, this does not hold for
varieties $\D$ of ordered algebras (i.e.~classes of ordered
$\Sigma$-algebras specified by inequations $t \leq t'$ between
terms). The problem is that $\Pos$ does not have a dense set of
objects projective w.r.t.\ strong epimorphisms. Indeed, only discrete
posets are projective w.r.t.\ the following regular epimorphism:

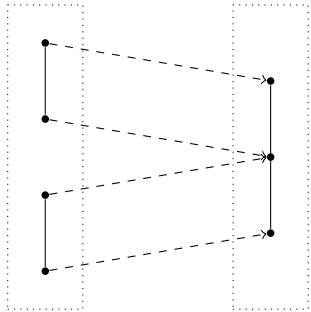
\begin{center}
  \begin{tikzpicture}
    \draw[dotted] (-2.5,-2.5) rectangle (-1.5,1.5);
    \draw[dotted] (0.5,-2.5) rectangle (1.5,1.5);

    \path
       (-2,-2) node(a) [circle,fill,inner xsep=0cm, inner ysep=0.5mm] {}
       (-2,-1) node(b)  [circle,fill,inner xsep=0cm, inner ysep=0.5mm] {}
       (-2, 0) node(c)  [circle,fill,inner xsep=0cm, inner ysep=0.5mm] {}
       (-2,1) node(d)  [circle,fill,inner xsep=0cm, inner ysep=0.5mm] {}
       (1, -0.5) node(e)  [circle,fill,inner xsep=0cm, inner ysep=0.5mm] {}
       (1, 0.5) node(f)  [circle,fill,inner xsep=0cm, inner ysep=0.5mm] {}
       (1, -1.5) node(g)  [circle,fill,inner xsep=0cm, inner ysep=0.5mm] {};

    \draw (a) -- (b);
    \draw (c) -- (d);
    \draw (f) -- (g);
	\draw[dashed,->] (d) -- (f);
    \draw[dashed,->] (b) -- (e);
    \draw[dashed,->] (a) -- (g);
    \draw[dashed,->] (c) -- (e);
  \end{tikzpicture}
\end{center}

We are going to show that for $\D = \Pos$ (and more generally varieties $\D$ of
ordered algebras) a change of the factorization system from
$(\mathsf{StrongEpi}, \mathsf{Mono})$ to (surjective, order-reflecting)
enables us to apply the results of Section \ref{sec:profinite-monad} to the proof that
pseudovarieties of finite ordered $\MT$-algebras are presentable by inequations
between profinite terms. This generalizes results of Pin and
Weil~\cite{PinWeil1996} who proved a version of Reiterman's theorem (without
monads) for ordered algebras, in fact, for general first-order structures. We begin with monads on $\Pos$, and then
show how this yields results for monads on varieties $\D$ of ordered algebras.

\begin{notation}
  Given an $\mc{S}$-sorted signature $\Sigma$ of operation symbols, let
  $\Sigma_{\leq}$ denote the $\mc{S}$-sorted first-order signature with 
    operation symbols $\Sigma$
  and
    a binary relation symbol $\leq_s$  for every $s \in \mc{S}$.
  Moreover, let
  \[
    \SigmaOAlg
  \]
  be the full subcategory of $\Sigma_{\leq}\text{-}\mathbf{Str}$ for
  which $\leq_s$ is interpreted as a partial order on the sort $s$ for
  every $s \in \mc{S}$, and moreover every $\Sigma$-operation is
  monotone w.r.t.~these orders. Thus, objects are ordered
  $\Sigma$-algebras, morphisms are monotone
  $\Sigma$-homomorphisms. Recall from \Cref{remark:factorization}
  our factorization system with
  \begin{align*}
    \E &=  \text{morphisms surjective in all sorts, and}  \\
    \M &=  \text{morphisms order-reflecting in all sorts.}
  \end{align*}
  Thus a $\Sigma$-homomorphisms $m$ lies in $\M$ iff for all $x, y$ in the same sort
  of its domain we have $x \leq y$ iff $m(x) \leq m(y)$. The notion of a subcategory $\D$ of $\SigmaOAlg$ being closed under factorizations is analogous to \Cref{def:closed_under_factorizions}.
\end{notation}

\begin{assumption}\label{assum:reflective-subcategory-ordered}
  Throughout this section, $\D$ denotes a full reflective subcategory
  of $\SigmaOAlg$ closed under factorizations. Moreover, $\Df$ is the
  full subcategory of $\D$ given by all algebras which are finite in
  every sort.
\end{assumption}
  Thus, every variety of ordered algebras (presented by inequations $t\leq t'$ betweens terms) can serve as $\D$, as well as every quasivariety (presented by implications between inequations).

\begin{remark}
  \begin{enumerate}
  \item \sloppypar Recall from \Cref{coro:pro-Df} that $\Pro\Df$ is a full
    subcategory of $\Stone(\SigmaOAlg)$, the category of ordered
    Stone $\Sigma$-algebras.

  \item The factorization system on $\D$ inherited from $\SigmaOAlg$
    is profinite, see \Cref{ex:profinite-factorization}. Moreover,
    the induced factorization system $\widehat{\E}$ and $\hatM$ of
    $\Pro{\Df}$ is given by the surjective and order-reflecting
    morphisms of $\Pro\Df$, respectively (see
    \Cref{prop:factorsation-system-subcategory}).
    
  \end{enumerate}
\end{remark}

\begin{notation}
  \begin{enumerate}
    \item We again denote by $(-)^@\colon \SigmaOAlg \to \D$ the reflector. 

    \item For every finite $\mc{S}$-sorted set $X$ we have the free algebra
      $F_\Sigma X$ (discretely ordered). 

    \item The free object of $\Pro\Df$ on a sorted set $X$ is again
      denoted by $X^\oplus$ (in lieu of $\widehat{(F_\Sigma
        X)^@}$). For every finite $\MT$-algebra $(A, \alpha)$, given an
      interpretation $f$ of $X$ in $(A,\alpha)$, we obtain a
      homomorphism
      \[
        f^\oplus \colon (\hatt X^\oplus, \hatmu_{X^\oplus}) \to (A, \alpha)
      \]
  \end{enumerate}
\end{notation}

\begin{defn}
  By a \emph{profinite term} on a finite $\mc{S}$-sorted set $X$ of variables is
  meant an element of $\hatt X^\oplus$. 
  
  Given profinite terms $t_1, t_2$ of the same sort $s$, a finite $\MT$-algebra
  $(A, \alpha)$ is said to \emph{satisfy the inequation}
  \[
    t_1 \leq t_2 
  \]
  provided that for every interpretation $f$ of $X$ we have 
  $f^\oplus(t_1) \leq f^\oplus (t_2)$. 
\end{defn}
\begin{thm}
  Let $\D$ be a full reflective subcategory of $\SigmaOAlg$ closed
  under factorizations, and let $\MT$ be a monad on $\D$ preserving
  sortwise surjective morphisms. Then a collection of finite
  $\MT$-algebras is a pseudovariety iff it can be presented by
  inequations between profinite terms.
\end{thm}
\begin{proof}
  In complete analogy to the proof of
  \Cref{thm:reiterman-for-monads-on-algebras}, we put
  \[
    \Var = \set{(F_\Sigma X)^@}{X \text{ a finite $\mc{S}$-sorted set}}. 
  \]
  and observe that \Cref{thm:reiterman-for-monads} and
  \Cref{R:r-for-monads} can be applied.
  \begin{enumerate}
  \item If $\V$ is a collection of finite $\MT$-algebras presented by
    inequations $t_i \leq t_i'$, we need to verify that $\V$ is a
    pseudovariety. This is analogous to the proof of
    \Cref{thm:reiterman-for-monads-on-algebras}; in
    part~\ref{itm::thm:reiterman-for-monads-on-algebras-2} we use that
    $m$ reflects the relation symbols $\leq_s$, hence from
    $m \o f^\oplus(t_i) \leq_s m \o f^\oplus (t_i')$ we derive
    $f^\oplus (t_i) \leq_s f^\oplus(t_i')$.
        
  \item Given an equation morphism
    $e\colon (\hatt X^\oplus, \hatmu_{X^\oplus}) \epito (A, \alpha)$,
    consider all inequations $t \leq_s t'$ where $t$ and $t'$ are
    profinite terms of sort $s$ with $Ue(t) \leq Ue(t')$ in $A$. We
    verify that a finite $\hatT$-algebra $(B, \beta)$ satisfies those
    inequations iff it satisfies $e$. This is again completely
    analogous to the corresponding argument in the proof of
    \Cref{thm:reiterman-for-monads-on-algebras}; just at the end we
    need to verify, additionally, that
    \[  
      \text{$x \leq_s x'$ in $B$}
      \qquad\text{implies}\qquad
      \text{$h(x) \leq_s h(x')$ in $A$}. 
    \]
Denote by $U\colon (\Pro{\Df})^\hatT\to \Pos^\S$ the forgetful functor.
    Since $Ue$ has surjective components, we have terms $t, t'$ in
    $\hatt X^\oplus$ of sort $s$ with $x = Ue(t)$ and $x' = Ue(t')$,
    thus $t \leq t'$ is one of the above inequations. The algebra
    $(B, \beta)$ satisfies $t \leq t'$ and (like in
    \Cref{thm:reiterman-for-monads-on-algebras}) we get
    $h = f^\oplus$, hence $Uh(t) \leq Uh(t')$. From $Uh = k \o Ue$,
    this yields $k(x) \leq_s k(x')$.\qedhere
  \end{enumerate}
\end{proof}
\begin{remark}
  In particular, if $\D$ is a variety of ordered one-sorted $\Sigma$-algebras
  and $\MT$ a monad preserving surjective morphisms, pseudovarieties of
  $\MT$-algebras can be described by inequations between profinite terms. This
  generalizes the result of Pin and Weil~\cite{PinWeil1996}. In fact, these authors consider pseudovarieties of general first-order structures, which can be treated within our categorical framework completely analogously to the case of ordered algebras.
\end{remark}

%%% Local Variables:
%%% mode: latex
%%% TeX-master: "reiterman"
%%% End:

\bibliographystyle{ACM-Reference-Format}
\bibliography{reference}

\clearpage
\appendix
\section{Ind- and Pro-Completions}

The aim of this appendix is to characterize, for an arbitrary small category $\Cat$, the free completion $\Pro{\Cat}$ under cofiltered limits and its dual concept, the free completion
$\Ind{\Cat}$
under filtered colimits (see \Cref{not:proc}). Let us first recall the construction of the latter:
\begin{remark}\label{rem:indcomp}
For any small category $\Cat$, the ind-completion is given up to equivalence by the full subcategory $\L$ of the presheaf category $[\Cat^\op, \Set]$ on filtered colimits of
    representables, and the Yoneda embedding
\[ E\colon \Cat\monoto \L,\quad C \mapsto \Cat(-,
    C).\]
We usually leave the embedding $E$ implicit and view $\Cat$ as a full subcategory of $\L$.
\end{remark} 
Dually to \Cref{rem:cofiltered}, an object $A$ of a category $\Cat$ is called \emph{finitely presentable} if the functor $\A(A,\dash)\colon \Cat\to\Set$ is finitary, i.e.~preserves filtered colimits.

\begin{defn}
Let $L$ be an object of a category $\L$. Its \emph{canonical diagram} w.r.t.~a full subcategory $\Cat$ of $\L$ is the diagram $D^L$ of all morphisms from objects of $\Cat$ to $L$:
\[ D^L\colon \Cat/L \to \L,\quad (C\xto{c} L)\mapsto C.\]
\end{defn}

\begin{lem}\label{lem:filtered-canonical-diagram}
  Let $\Cat$ be a full subcategory of $\L$ such that each object $C\in \Cat$ is finitely presentable in $\L$. An object $L$ of $\mathscr{L}$ is a colimit of some filtered diagram
  in $\Cat$ if and only if its canonical diagram is filtered and the
  canonical cocone $(C \xto{c} L)_{c\in \Cat / L}$ is a
  colimit.
\end{lem}
\begin{proof}[Proof sketch]
  The \emph{if} part is trivial. Conversely, if $L$ is a colimit of
  some filtered diagram, then we can view it as a {final}
  subdiagram of its canonical diagram.  Therefore, their colimits
  coincide.
\end{proof}

\begin{thm}\label{thm:indcomp}
Let $\Cat$ be a small category. A category $\L$ containing $\Cat$ as a full subcategory is an ind-completion of $\Cat$ if and only if the following conditions hold:
\begin{enumerate}
\item\label{itm:lem:char-of-indcompltion-1} $\L$ has filtered colimits,
\item\label{itm:lem:char-of-indcompltion-2} every object of $\L$ is the colimit of a filtered diagram in $\Cat$, and
\item\label{itm:lem:char-of-indcompltion-3} every object of $\Cat$ is finitely presentable in $\L$.
\end{enumerate}
\end{thm}

\begin{proof}
  \begin{enumerate}
  \item The \emph{only if} part follows immediately from the construction of $\Ind{\Cat}$ in \Cref{rem:indcomp}:~\ref{itm:lem:char-of-indcompltion-1} is obvious, \ref{itm:lem:char-of-indcompltion-3} follows
    from the Yoneda Lemma, and
    \ref{itm:lem:char-of-indcompltion-2} follows from
    \Cref{lem:filtered-canonical-diagram} and the fact that $\Cat$ is dense in
    $[\Cat^\op, \Set]$.
    \item We now prove the \emph{if} part. Suppose that
      \ref{itm:lem:char-of-indcompltion-1}--\ref{itm:lem:char-of-indcompltion-3}
      hold. Let $F\colon \Cat \to \KCat$ be any functor to a category $\KCat$
      with filtered colimits.
      
    \begin{enumerate}
    \item First, define the extension $\overline{F}\colon \L\to \KCat$ of $F$
      as follows.  For any object $L \in \mathscr{L}$ expressed as the
      canonical colimit $(C \xto{c} L)_{c \in \Cat/L}$, the colimit of
      $F\! D^L$ exists since the canonical diagram is filtered
      by condition~\ref{itm:lem:char-of-indcompltion-2} and $\KCat$ has filtered
      colimits.  Thus $\overline{F}$ on objects can be given by a
      choice of a colimit:
      \[
        \overline{F}L \defeq \colim\left(\Cat/L \xto{D^L} \Cat \xto{F} \KCat\right)
      \]
      We choose the colimits such that $\overline{F}L = L$ if $L$ is in $\Cat$.  For any
      morphism $f\colon L \to L'$, each colimit injection
      $\tau_c\colon FC \to \overline{F}L$, for $C \xto{c} L$,
      associates with another colimit injection
      $\tau'_{f\o c}\colon FC \to \overline{F}L'$.  Hence, there is a
      unique morphism
      $\overline{F}f\colon \overline{F}L \to \overline{F}L'$ such that
      $\tau'_{f\o c} = \overline{F}f \o \tau_c$.  By the uniqueness of mediating
      morphisms, $\overline{F}$ preserves identities and composition.
      Therefore, $\overline{F}$ extends $F$.
      
    \item Second, we show that $\overline{F}$ is finitary. Observe
      that $\overline{F}$ is in fact a pointwise left Kan
      extension of $F$ along the embedding $E\colon \Cat\monoto\L$. By~\cite[Cor.~X.5.4]{maclane} we have,
      equivalently, that for every $L \in \mathscr{C}$ and $K \in \KCat$
      the following map from
      $\KCat(\overline FL, K)$ to the set of natural transformations from
      $\mathscr{L}(E-, L)$ to $\KCat(F-, K)$ is a bijection: it assigns
      to a  morphism $f \colon \overline{F}L \to K$ the natural
      transformation whose components are
      \[
        \left(EC \xto{c} L\right)
        \mapsto 
        \left(FC = \overline{F}EC \xto{\overline{F}c} \overline{F}L
          \xto{f} K\right). 
      \]
      Hence, given any colimit cocone $(A_i \to L)_{i \in \mc{I}}$ of
      a filtered diagram, we have the following chain of isomorphisms,
      natural in $K$:
      \begin{align*}
        \KCat(\overline{F}L, K)
        &\cong [\Cat^\op, \Set](\mathscr{L}(E-, L), \KCat(F-, K))
        && \text{ see above }\\
        &\cong [\Cat^\op, \Set](\colim_i \mathscr{L}(E-, A_i), \KCat(F-, K))
        &&\text{ by condition~\ref{itm:lem:char-of-indcompltion-3} } \\
        &\cong \limit_i [\Cat^\op, \Set](\mathscr{L}(E-, A_i), \KCat(F-, K)) \\
        &\cong \limit_i \KCat(\overline{F}A_i, K)
        && \text{ see above } \\
        &\cong \KCat(\colim \overline{F}A_i, K)
      \end{align*}
      Thus, by Yoneda Lemma,
      $\colim\overline{F}A_i = \overline{F}L$, i.e.~$\overline{F}$ is
      finitary.
      
    \item The essential uniqueness of $\overline{F}$ is clear, since
      this functor is given by a colimit construction.\qedhere
    \end{enumerate}
  \end{enumerate}
\end{proof}
By dualizing \Cref{thm:indcomp}, we obtain an analogous characterization of pro-completions:
\begin{cor} \label{lem:char-of-procompletion}
 Let $\Cat$ be a small  category.
 The pro-completion of $\Cat$ is characterized, up to equivalence of categories, as a category $\L$ containing $\Cat$ as a full subcategory such that
  \begin{enumerate}
    \item \label{itm:char-of-procompletion-1} $\L$ has cofiltered limits,
    \item \label{itm:char-of-procompletion-3} every object of $\L$ is a
      cofiltered limit of a diagram in~$\Cat$, and
    \item \label{itm:char-of-procompletion-2} every object of $\Cat$ is finitely copresentable in $\L$.
  \end{enumerate}
\end{cor}

\begin{remark}\label{rem:profinitecompletion}
Let $\Cat$ be a small category.
\begin{enumerate}
  \item $\Pro{\Cat}$ is unique up to equivalence.
  \item $\Pro\Cat$ can be constructed as the full subcategory of $[\Cat, \Set]^{\op}$ given by
    all cofiltered limits of representable functors.  The category
    $\Cat$ has a full embedding into~$\Pro\Cat$ via the Yoneda
    embedding~$E\colon C\mapsto \Cat(C,\dash)$. This follows from the description of Ind-completions in \Cref{rem:indcomp} and the fact that
    \[ 
      \Pro{\Cat} = (\Ind{\Cat^{\op}})^{\op}.
    \]
  \item If the category $\Cat$ is finitely complete, then $\Pro{\Cat}$ can also
    be described as the dual of the category of all functors in
    $\func{\Cat}{\Set}$ preserving finite limits.  Again, $E$ is given
    by the Yoneda embedding.  This is dual
    to~\cite[Thm.~1.46]{Adamek1994}.  Moreover, it follows that
    $\Pro\Cat$ is complete and cocomplete.

  \item Given a small category $\KCat$ with cofiltered limits, denote
    by $[\Pro \Cat,\KCat]_{\mathrm{cfin}}$ the full subcategory of
    $[\Pro \Cat,\KCat]$ given by cofinitary functors. Then the
    pre-composition by $E$ defines an equivalence of categories
    \[
      (\dash)\o E\colon [\Pro\Cat,\KCat]_{\mathrm{cfin}} \xto{~\simeq~} [\Cat,\KCat],
    \]
    where the inverse is given by right Kan extension along $E$.
\end{enumerate}
\end{remark}

\end{document}